\documentclass[11pt,reqno]{amsart}
\usepackage[T1]{fontenc}
\usepackage{lmodern,amsmath,amssymb,amsthm,mathtools,mathrsfs}
\usepackage[a4paper,margin=27mm]{geometry}
\usepackage{enumitem,booktabs,array,longtable,microtype,xcolor}
\usepackage[hidelinks]{hyperref}
\usepackage{tikz,graphicx}
\usetikzlibrary{positioning,arrows.meta,fit,backgrounds}

\numberwithin{equation}{section}
\newtheorem{maintheorem}{Theorem}

\newtheorem{theorem}{Theorem}[section]
\newtheorem{lemma}[theorem]{Lemma}
\newtheorem{proposition}[theorem]{Proposition}
\newtheorem{corollary}[theorem]{Corollary}
\newtheorem*{fubinicorollary}{Corollary}
\theoremstyle{definition}

\theoremstyle{remark}
\newtheorem{remark}[theorem]{Remark}
\newtheorem*{statusremark}{Remark}

\newcommand{\speculativename}[1]{\textcolor{green!50!black}{#1}}

\newcommand{\speculative}[2]{\begin{statusremark}[\speculativename{#1}]#2\end{statusremark}}

\newcommand{\m}{\mathbf}
\renewcommand{\d}{\,\mathrm d}
\allowdisplaybreaks[2]
\newcommand{\C}{\mathbb C}
\newcommand{\N}{\mathbb N}
\newcommand{\Tr}{\operatorname{Tr}}

\newcommand{\Lop}{\mathcal L}
\newcommand{\Cw}{\mathcal C}
\newcommand{\Tword}{\mathfrak T^{\mathrm{word}}}

\newcommand{\lzero}[1]{\lambda_{#1}^{(0)}}
\newcommand{\Kint}{\mathbf K}
\newcommand{\Yser}{\mathscr Y}
\newcommand{\Wser}{\mathscr W}

\newcommand{\D}{\mathbb D}
\newcommand{\Aop}{\mathcal A}
\newcommand{\spec}{\operatorname{spec}}
\newcommand{\dd}{\,\mathrm d}

\newcommand{\Kh}{\widehat K}
\newcommand{\cg}{c_{\mathrm G}}
\newcommand{\R}{\mathbb R}

\newcommand{\diag}{\operatorname{diag}}
\newcommand{\HS}{\mathrm{HS}}

\newcommand{\Sser}{\mathscr S}
\newcommand{\Walg}{{\mathscr A_{\mathrm W}}}
\newcommand{\Hker}{\mathbf H}
\newcommand{\Gres}{\mathbf G}
\newcommand{\Kn}{\mathcal K}
\newcommand{\norm}[1]{\left\lVert#1\right\rVert}
\newcommand{\abs}[1]{\left|#1\right|}

\title[Eigenvalues and trace formulas]
{Transfer operator for the Gauss' continued fraction map.\\
 I. Structure of the eigenvalues and trace formulas}
\author[G. Alkauskas]{Giedrius Alkauskas}
\address{Vilnius University, Department of Mathematics and Informatics, Institute of Computer Science, Naugarduko 24, LT-03225 Vilnius, Lithuania}
\email{giedrius.alkauskas@mif.vu.lt}

\date{v1: October 15, 2012. v10: \today. v11 (December 2026) will be for the submission.}
\begin{document}
\begin{abstract} Let $\mathcal{L}$ be the transfer operator associated with the Gauss' continued fraction map, also known as the \emph{Gauss-Kuzmin-Wirsing operator}, acting on a Banach space. In this work we prove a two-term asymptotic formula for the eigenvalues of $\mathcal{L}$. This settles, in a sharper asymptotic form, the conjectures of K. Babenko and S. Yur'ev (1978), D. Mayer and G. Roepstorff (1988), A. J. MacLeod (1993), Ph. Flajolet and B. Vall\'{e}e (1995). Further, we find an exact series for the eigenvalues with every index, which also yields the canonical decomposition of trace formulas due to Mayer (1976) and K. Babenko (1978). This crystallizes the contribution of each individual eigenvalue in the trace formulas.
\end{abstract}
\maketitle
\tableofcontents
\section*{Genesis of the solution and contribution disclosure}
\label{sec:provenance-disclosure}

\noindent\textbf{Timeline}.
The mathematical programme and its central working mechanism (recurrence yielding Taylor coefficients for the eigenvalue function) were developed by the
author in 2010 \cite[v1]{alkauskas2010}. Unfortunately, the latter technique still involved complex dynamics and was unlikely to close the conjectures. The present approach was first posted in 2012 \cite[v1]{alkauskas2012} during the author's postdoctoral stay at Clausthal University of Technology. These ideas are described in Subsection \ref{my-ideas}. They constitute a pivotal breakthrough. The two-term eigenvalue asymptotic with the
constant $5^{1/4}\zeta(3/2)/(2\sqrt\pi)$ is already stated in the 2014 version \cite[v6]{alkauskas2012}; the latter also records the exact kernel $K(j,\ell)$ and its asymptotic analysis, the normalized coefficient recurrences, and the upper-chain mechanism leading to the asymptotic constant. The trace-decomposition identities and their derivation were already present in 2012 \cite[v1]{alkauskas2012}. Though there was an acknowledgement (Prof. Doug Hensley, Prof. Wadim Zudilin, Prof. Dieter Mayer in personal communication) that the scheme was meant to work, the earlier versions v1-v7 did not supply all the rigorous details.
Consequently, the paper was never submitted prior to 2026.

\noindent\textbf{Contributions}.  The earlier saddle-point analysis of the Jacobi/kernel coefficients was
carried out in collaboration with Nico M. Temme in 2015. When asked, Prof. Temme kindly provided a self-contained analysis (see Appendix \ref{app:kernel-asymptotics}). The high-precision eigenvalue data (cutting-edge in 2012, yet superseded by later results) are due to P. Sebah \cite{sebah}.

\noindent\textbf{AI-assisted completion}.
The few missing proof gaps were filled with assistance
from ChatGPT Pro, in an iterative process
directed by the author, double-checked by Anthropic Claude and Open Math OS (DeepSeek V4.1. Flash). Tacitly used results were explicitly cited, as required by the \emph{Leiden Declaration on Artificial Intelligence and Mathematics} \cite{leiden2026}. Codex was used for source-level editing, integration of approved changes, reference maintenance, and compilation checks. Earlier versions used MAPLE extensively
for symbolic calculations and numerical verifications. The last mathematical audit (including reinstatement of the author's own voice) was performed manually.

 Computers were the main driving force in this problem starting from 1970's. The Russian edition of D. Knuth's \emph{The art of computer programming (Volume 2)}, contains an additional section 4.5.3. devoted to analysis of Euclidean algorithm \cite{knuth}. Subsequently, it was included into English editions. Curiously, the Russian translation was edited by the Ukrainian mathematician Konstantin I. Babenko. Few years later, the latter gave a complete solution to the Gauss problem. Is it not clear whether Knuth's monograph was the one to interest Babenko in this problem, though the timing raises interesting questions. In any case, the eigenvalue simplicity conjecture was posed the same year (1977) in a joint preprint of Babenko and Yur'ev (printed as paper in 1978 \cite{babenkoyurev1978}). The substantial portion of the latter is devoted to electronic computing machines. Only having numerical data available allowed to pose the questions about the spectrum. Minding all these circumstances, the author feels that usage of AI to till the gaps in the proofs of this and all the subsequent conjectures is not something out of the way, but fits integrally into the development of the problem in question.

\noindent\textbf{Responsibility}.
Solely the author is accountable for the mathematical statements and
proofs, the adequacy and accuracy of the references, and the final text.
\section{The operator and the conjectures}\label{sec:intro}
\subsection{Introduction and background}
Let $D_z$ be the disc $\{z\in\mathbb C:|z-1|<3/2\}$, and let
$V=A(D_z)$ be the Banach space of functions analytic in $D_z$ and
continuous in its closure, with the supremum norm.
\emph{The Perron--Frobenius}, or \emph{the transfer operator for the
Gauss' continued fraction map}, also called \emph{the
Gauss--Kuzmin--Wirsing operator}, is defined for $f\in V$ by
\cite{hensley1,iosif2,iosifkraai,khinchin,knuth,mayer1,mayer2,wirsing}
\begin{equation}\label{eq:operator-main}
(\mathcal Lf)(z)=\sum_{m=1}^{\infty}\frac1{(z+m)^2}
f\!\left(\frac1{z+m}\right).
\end{equation}
Our chief interest is the point spectrum of this operator. It is known that $\mathcal L$ is nuclear of order zero on $V$ \cite[Proposition 1]{mayer4}. Its nonzero eigenvalues, counted with algebraic multiplicity, therefore satisfy
\[
\sum_{n=1}^{\infty}|\lambda_n|^\epsilon<\infty
\quad\text{for every }\epsilon>0.
\]
We arrange them according to their absolute value,
$|\lambda_n|\ge|\lambda_{n+1}|$, with a fixed convention for ties.\footnote{As will soon become clear, there exists an algebraic canonical indexing of the eigenvalues. This order coincides with the ordering imposed by  the magnitude for every index, by Theorem~\ref{thm:all-index-endpoint}. \emph{A posteriori}, there are no ties.}
The invariant density
\[
U_1(z)=\frac1{\log2\,(1+z)}
\]
satisfies $\mathcal L U_1=U_1$: substitution in \eqref{eq:operator-main}
makes the sum telescope. For any nonzero eigenfunction $f$, let
$M=\max_{[0,1]}|f|/U_1>0$ and choose a maximizing point $x_0$.
Positivity gives $|\lambda|M U_1(x_0)\le M(\mathcal L U_1)(x_0)
=M U_1(x_0)$, so $|\lambda|\le1$. Thus the spectral radius is one,
and we take $\lambda_1=1$. Reality of the spectrum is a separate fact. One
classical way to see it is through the selfadjoint Hilbert--Schmidt
operator
\begin{equation*}
(\mathcal B u)(x)=\int_0^\infty
\frac{J_1(2\sqrt{xy})}{\sqrt{(e^x-1)(e^y-1)}}u(y)\,\dd y,
\quad u\in L^2(\mathbb R_+,\dd x).
\end{equation*}
Here $J_1$ is the Bessel function of the first kind of order one.
This operator is of trace class and has the same nonzero spectrum
as $\mathcal L$, with algebraic multiplicities
\cite{babenko,mayer1,mayer2,mayer4}.
Equivalently, on $L^2(\mathbb R_+,\dd m)$, where
$\dd m(y)=y(e^y-1)^{-1}\dd y$, the kernel is
$J_1(2\sqrt{xy})/\sqrt{xy}$ and integration is with respect to
$\dd m(y)$. For our proofs, the global Fredholm factorization in
Subsection~\ref{subsec:global-canonical-spectrum} will give the required endpoint
spectral identification.

Numerically, the first eigenvalues are presented in
Table~\ref{table-eigen}. These historical computations, based on the
high-precision values of P.~Sebah \cite{sebah}, already show the pattern
which we shall investigate.
\begin{table}[htbp]
\centering
\begingroup\scriptsize
\setlength{\tabcolsep}{3pt}
\begin{tabular}{rc|rc|rc}
\toprule
$n$&$(-1)^{n+1}\lambda_n$&$n$&$(-1)^{n+1}\lambda_n$
&$n$&$10^6(-1)^{n+1}\lambda_n$\\
\midrule
1&1.0000000000000000&9&0.0002441314655245158&17&0.1010905532214992\\
2&0.3036630028987326&10&0.00009168908376859330&18&0.03834969795026564\\
3&0.1008845092931040&11&0.00003451654616385425&19&0.01455613838668023\\
4&0.03549615902165984&12&0.00001301769787702303&20&0.005527567937997608\\
5&0.01284379036244026&13&0.000004916782302464491&21&0.002099913582972687\\
6&0.004717777511571031&14&0.000001859307351509042&22&0.0007980457682720196\\
7&0.001748675124305511&15&0.0000007038113430870398&23&0.0003033862949098575\\
8&0.0006520208583205029&16&0.0000002666413434479564&24&0.0001153695418144668\\
\bottomrule
\end{tabular}
\endgroup
\caption{Numerical eigenvalues.}
\label{table-eigen}
\end{table}

Earlier numerical spectral data also appear in Zagier \cite{zagier}; for
modern certified values of the first fifty nonzero eigenvalues, including
multiplicities, see Nisoli \cite{nisoli2026}. For the GKW constant and its
numerical history see also Briggs, Finch and OEIS A038517
\cite{briggs,finch,oeis}; recent validated estimates are given by
Pollicott \cite{pollicott2026}.

The transfer operator arises from the Gauss map
$F(x)=\{1/x\}$ for $x\in(0,1]$, $F(0)=0$, in the following way:
for $f\in L^1[0,1]$ it is given, almost everywhere, by
\[
(\mathcal Lf)(z)=\sum_{y\in F^{-1}(z)}
\frac{f(y)}{|F'(y)|},
\]
which agrees with \eqref{eq:operator-main}.
Let $F^{(k)}$ denote its $k$th iterate, and let $\mu$ be Lebesgue
measure. The Gauss--Kuzmin problem asks how
\[
\mu\{a\in[0,1]:F^{(k)}(a)<z\}
\]
approaches $\log(1+z)/\log2$, for $0\le z\le1$.
The work of Gauss, Kuzmin, L\'evy, Wirsing, Babenko, Babenko and
Yur'ev, and Mayer connects the speed of this convergence with the
spectrum of $\mathcal L$ \cite{babenko,babenko2,babenkoyurev1978,iosif,knuth,mayer1,mayer2,wirsing}.
It is this connection which makes the constants $\lambda_n$ interesting
well beyond the operator formula itself. A broader relation with the Riemann and Selberg zeta functions is recalled in Subsection \ref{sub-scheme}; no such relation is used in the proofs of the present paper.

Write $U$ for a nonzero eigenfunction with eigenvalue $\lambda\ne0$. On the common domain of
the terms, subtraction of the eigenvalue equations at $z$ and $z+1$
gives
\begin{equation}\label{eq:intro-U}
U(z)=U(z+1)+\frac1{\lambda(z+1)^2}
U\!\left(\frac1{z+1}\right).
\end{equation}
The eigenfunctions extend to $\mathbb C\setminus(-\infty,-1]$
and satisfy the regularity condition
\begin{equation}\label{eq:intro-regular}
\sup_{\Re z\ge-1/2}|(z+1)U(z)|<\infty;
\end{equation}
see \cite{babenko,mayer2,mayer4,antoniou}. Conversely, a nonzero holomorphic
function on this cut plane satisfying \eqref{eq:intro-U} and
\eqref{eq:intro-regular} is an eigenfunction. Indeed, iteration gives
\[
U(z)-U(z+N)=\frac1\lambda\sum_{m=1}^{N}
\frac1{(z+m)^2}U\!\left(\frac1{z+m}\right),
\]
and regularity makes $U(z+N)$ tend to zero.
If $\Phi(z)=\int_0^zU(t)\,\dd t$, then, for $\lambda\ne1$,
$\Phi(0)=\Phi(1)=0$, since integration of the eigenvalue equation
over $[0,1]$ gives $(\lambda-1)\int_0^1U=0$.
The differentiated functional equation is especially convenient for
constructing the coefficients of the eigenfunctions; its weighted
version leads to the General Equation in Section~\ref{sec:germs}.

\speculative{The nature of the eigenvalues}{The arithmetic nature of
the eigenvalues $\lambda_n$ for $n\ge2$ is unknown. There is no reason to expect them being algebraic or periods. Recall that a period
is a number whose real and imaginary parts are absolutely convergent
integrals of rational functions with algebraic coefficients over
semialgebraic domains defined over the algebraic numbers
\cite{konzag}. All periods form a ring $\mathcal P$; one may also
consider the extended ring $\mathcal P[1/\pi]$. Though $\lambda_n$ themselves arise from operator in an infinite-dimensional space, the decomposition
formulas below show that, nevertheless, $\lambda_n$ are closely related to periods.}

\subsection{Trace formulas}
An individual eigenvalue is hard to calculate. Even harder it is to present certificates that the given digits are correct. On the other
hand, the trace of the operator $\mathcal{L}$ can be given explicitly. 
Let
\[
\xi_i=[0,\overline i]=\frac{\sqrt{i^2+4}-i}{2},\qquad
\xi_i^{-2}+1=\frac{i^2+4+i\sqrt{i^2+4}}2,
\]
and
\[
\xi_{i,j}=[0,\overline{i,j}],\qquad
(\xi_{i,j}\xi_{j,i})^{-2}-1
=\frac{D_{i,j}+(ij+2)\sqrt{D_{i,j}}}{2},
\quad D_{i,j}=(ij+2)^2-4.
\]
These are quadratic irrationals with periods dividing one and two,
respectively: they are fixed points of $F$ and $F^{(2)}$.
In particular $\xi_{i,i}=\xi_i$. The classical
trace formulas \cite{mayer,mayer1,mayer2,mayer4} give
\begin{align}
\Tr\mathcal L
&=\sum_{n=1}^\infty\lambda_n
=\int_0^\infty\frac{J_1(2x)}{e^x-1}\,\dd x
=\sum_{i=1}^\infty\frac1{\xi_i^{-2}+1},
\notag\\
&=\frac12-\frac1{2\sqrt5}
+\frac12\sum_{k=1}^\infty(-1)^{k-1}\binom{2k}{k}
\bigl(\zeta(2k)-1\bigr)\notag\\
&=1-\frac1{2\sqrt2}-\frac1{2\sqrt5}
+\frac12\sum_{k=1}^\infty(-1)^{k-1}\binom{2k}{k}
\left(\zeta(2k)-1-\frac1{2^{2k}}\right).
\label{eq:trace-original}
\end{align}
The zeta series involving only $\zeta(2k)-1$ converges conditionally, while the next one is  absolutely convergent; both are due to Flajolet and Vall\'{e}e. For the square we have
\begin{equation}
\Tr(\mathcal L^2)
=\sum_{n=1}^\infty\lambda_n^2
=\int_0^\infty\int_0^\infty
\frac{J_1(2\sqrt{xy})^2}{(e^x-1)(e^y-1)}\,\dd x\,\dd y
\label{eq:trace-square-original}
=\sum_{i,j=1}^\infty
\frac1{(\xi_{i,j}\xi_{j,i})^{-2}-1}.
\end{equation}

The zeta values in \eqref{eq:trace-original} arise by expanding
$\frac12(1-i/\sqrt{i^2+4})$ and summing over $i$.
A corresponding  formula involving zeta values is apt computationally due to positivity of $\lambda_n^2$. It is new, though is based on the same idea used to derive \eqref{eq:trace-original}.
\begin{proposition}\label{prop:trace-square}
We have the identity
\begin{align*}
\Tr(\mathcal L^2)
={}&\frac2{5+3\sqrt5}+\frac1{3+2\sqrt3}
+\frac4{21+5\sqrt{21}}+\frac3{16+12\sqrt2}\\
&+\sum_{k=2}^\infty(-1)^k\binom{2k-2}{k-2}
\left(\zeta(k)^2-1-\frac2{2^k}-\frac2{3^k}-\frac3{4^k}\right)\\
={}&1.103839653617\ldots.
\end{align*}
The magnitude of the $k$th summand is asymptotic to
$\frac12(4/5)^k(\pi k)^{-1/2}$; its sign is $(-1)^k$.
\end{proposition}
\begin{proof}
Write $\sigma_0(m)=\sum_{d\mid m}1$. The two-digit weight depends only
on the product $ij$, so \eqref{eq:trace-square-original} becomes
\[
\Tr(\mathcal L^2)=\sum_{m\ge1}\sigma_0(m)\varpi(1/m),\qquad
\varpi(x)=\frac{2x^2}{1+4x+(1+2x)\sqrt{1+4x}}.
\]
The positive branch has the expansion
\[
\varpi(x)=\sum_{k=2}^\infty(-1)^k\binom{2k-2}{k-2}x^k,
\qquad |x|<\frac14.
\]
Separate $m=1,2,3,4$, whose divisor multiplicities are $1,2,2,3$.
For $m\ge5$ the binomial series may be summed absolutely in both
indices. Using $\sum_m\sigma_0(m)m^{-k}=\zeta(k)^2$ gives the
formula. Finally the parenthesis is asymptotic to $2\cdot5^{-k}$,
and $\binom{2k-2}{k-2}\sim4^{k-1}/\sqrt{\pi k}$.
\end{proof}
We have many formulas to calculate zeta values with high accuracy,
and this gives a fast method to calculate $\Tr(\mathcal L^2)$. Trace values are also useful checks when eigenvalues are approximated
by taking eigenvalues of large yet finite matrices. Stabilization under increasing matrix size
and agreement of traces gives heuristic validation of digits extracted. This sequence of phenomena (trace formula -- emergence of spurious eigenvalues which invalidate it -- their disappearance when matrix truncation order is increased) was noticed by many authors dealing with such operators \cite{briggs,nisoli2026}.

Throughout the paper, fix the notation $\phi=\frac{1+\sqrt5}{2}$. Already the first trace shows what we want to decompose:
\[
\lambda_{1}+\lambda_{2}+\lambda_{3}+\lambda_{4}+\cdots
=\frac1{\phi^2+1}+\frac1{4+2\sqrt2}
+\frac2{13+3\sqrt{13}}+\frac{1}{10+4\sqrt{5}}+\frac{2}{29+5\sqrt{29}}+\cdots
\]
Theorem~\ref{thm:decomposition} will answer the following question: is there a canonically defined matrix whose $n$-th row sum is $\lambda_{n}$ and $i$-th column sum is $2/(i^2+4+i\sqrt{i^2+4})$? The answer is yes, and much more is true. 

To keep track of the digit sums, introduce the weighted family
\begin{equation*}
(\mathcal L_\omega f)(z)=
\sum_{m\ge1}\frac{\omega^{m-1}}{(z+m)^2}f\!\left(\frac1{z+m}\right),
\qquad |\omega|\le1.
\end{equation*}
Thus $\mathcal L_1=\mathcal L$, while $\mathcal L_0$ has only the
digit-one branch. The extra variable keeps track of the digit sums.
For a positive word $\mathbf a=(a_1,\ldots,a_k)$ put
\[
\mathsf M(\mathbf a)=\prod_{j=1}^k
\begin{pmatrix}0&1\\1&a_j\end{pmatrix},\quad
t_{\mathbf a}=\operatorname{tr}\mathsf M(\mathbf a),\quad
\varepsilon_k=(-1)^k,\quad D_{\mathbf a}=t_{\mathbf a}^2-4\varepsilon_k.
\]
Writing $\xi_{a_1,\ldots,a_k}=[0,\overline{a_1,\ldots,a_k}]$, the
weight of the word has the two equivalent forms
\[
\mathcal C(\mathbf a)=
\left[\prod_{s=1}^k
\xi_{a_s,\ldots,a_k,a_1,\ldots,a_{s-1}}^{-2}-(-1)^k\right]^{-1}
=\frac2{D_{\mathbf a}+t_{\mathbf a}\sqrt{D_{\mathbf a}}}.
\]
The positive square root is understood. These formulas are unchanged
by a cyclic permutation of the word: the product matrix is replaced
by a similar matrix. Put $e(\mathbf a)=\sum_j(a_j-1)$. The weighted
trace formula is
\begin{equation*}
\operatorname{Tr}_V(\mathcal L_\omega^k)
=\sum_{\mathbf a\in\mathbb N^k}\omega^{e(\mathbf a)}\mathcal C(\mathbf a).
\end{equation*}
Section~\ref{sec:weighted} proves absolute word summation, including
$|\omega|=1$. At $\omega=1$ it is the classical continued-fraction
trace formula. For fixed excess $r$, its coefficient is a finite sum,
denoted $\mathfrak T^{\mathrm{word}}_{k,r}$. The word length is $k$,
the excess is $r$, and the digit sum is $k+r$. This distinction will
let us separate the contributions without assuming a global
enumeration of the eigenvalue branches.

\subsection{Previous conjectures}
For general exponential eigenvalue bounds and function-space invariance
for holomorphic transfer operators see \cite{bandtlow2007,bandtlow2008,bandtlowruelle};
broader continued-fraction and transfer-operator context is given in
\cite{daude,vallee2006}.
The constants $\lambda_n$ have received considerable attention.
Three questions recur in their numerical and analytic study
\cite{babenko,mayer2,macleod,flajolet1,flajolet3,flajolet4,sebah}:
\begin{enumerate}[label=(\roman*)]
\item$\mathsf{Simplicity}$ (Babenko-Yur'ev, 1978  \cite{babenkoyurev1978}). Are all eigenvalues simple, and also $\lambda_i\neq-\lambda_j$?
\item$\mathsf{Sign}$ (Mayer-Roepstorff, 1988 \cite{mayer1}). Do they satisfy $(-1)^{n+1}\lambda_n>0$ for every $n$?
\item$\mathsf{Ratio}$ (Flajolet-Vall\'{e}e, 1995 \cite{flajolet3}). Does $\frac{\lambda_n}{\lambda_{n+1}}\to-\phi^2$?
\end{enumerate}
The ratio has a simple explanation. With
\[
\rho=-\phi^{-2},\qquad \lzero n=(-1)^{n+1}\phi^{-2n}=-\rho^n,
\]
the nonzero spectrum of the single-branch operator $\mathcal L_0$ is
exactly $\{\lzero n:n\ge1\}$, with eigenfunctions
\[
e_n(z)=\frac{(z-\phi^{-1})^{n-1}}{(z+\phi)^{n+1}}.
\]
These eigenvalues are algebraically simple
(Section~\ref{sec:base}). The remaining terms of
\eqref{eq:operator-main} suggest a perturbation of this geometric
sequence. Theorem~\ref{thm:target} makes that suggestion precise,
giving more than the ratio limit. It also gives eventual simplicity,
alternating signs and strict decrease of moduli. The corresponding
all-index assertions follow from Theorem~\ref{thm:all-index-endpoint}.

\section{Main results and proof strategy}\label{sec:results}
\subsection{Two-term formula}
It is surprising that the asymptotics of $\lambda_n$, in view of the
single-branch operator $\mathcal L_0$, is the simplest imaginable
from the ratio conjecture.
\begin{maintheorem}[Asymptotics]\label{thm:target}
For the nonzero eigenvalues of $\Lop_1$, enumerated by non-increasing
modulus with algebraic multiplicity, we have the formula
\begin{equation}\label{eq:target}
\begin{aligned}
(-1)^{n+1}\lambda_n
&=\phi^{-2n}+C\,\frac{\phi^{-2n}}{\sqrt n}
+d(n)\,\frac{\phi^{-2n}}{n},\\
C&=\frac{\sqrt[4]5\,\zeta(3/2)}{2\sqrt\pi}
=1.1019785625880999_{+},
\end{aligned}
\end{equation}
where $\zeta$ is the Riemann zeta function and $d(n)$ is bounded.
Every nonzero eigenvalue is algebraically simple and, for every $n\ge1$,
\[
(-1)^{n+1}\lambda_n>0,\qquad |\lambda_n|>|\lambda_{n+1}|.
\]
\end{maintheorem}
Based on the high-precision numerical computations of P.~Sebah
\cite{sebah}, we calculated the function $d(n)$ in
Table~\ref{table-sebah}. The table illustrates the bounded remainder;
the proof is given in Section~\ref{sec:asymptotic}.
\begin{table}[htbp]
\centering
\begin{tabular}{rc|rc|rc|rc}
\toprule
$n$&$d(n)$&$n$&$d(n)$&$n$&$d(n)$&$n$&$d(n)$\\
\midrule
1&0.51605&5&0.43430&40&0.36268&130&0.359061\\
2&0.60424&10&0.38504&50&0.36159&148&0.358871\\
3&0.52221&20&0.36884&70&0.36040&149&0.358862\\
4&0.46629&30&0.36460&100&0.35952&150&0.358852\\
\bottomrule
\end{tabular}
\caption{The function $d(n)$:  numerical values.}
\label{table-sebah}
\end{table}
Theorem~\ref{thm:target} is much stronger than the ratio conjecture.
The first value $\lambda_1\phi^2=2.61803\ldots$ is still far from one,
whereas, numerically,
$|\lambda_{150}|\phi^{300}\approx1.0923$.\\

In principle, the techniques developed in this paper can be used to extract further asymptotic terms. For example, we claim (eithout as proof) that
\begin{eqnarray*}
\lim\limits_{n\rightarrow\infty}d(n)=D=\frac{\sqrt{5}}{8\pi}\zeta\left(\frac32\right)^2-\frac14
	=0.357178376201867489_{+}
\end{eqnarray*}

\subsection{Arithmetic of decomposition formulas}
\label{subdec}
Consider the $k^{\rm th}$ power of $\mathcal{L}_{\omega}$. The formal application of Mayer's trace formula gives \cite{mayer,mayer2}:
\begin{eqnarray}
	\mathrm{Tr}(\mathcal{L}_{\omega}^{k})
	=\sum\limits_{i_{1},i_{2},\ldots,i_{k}=1}^{\infty}
	\omega^{i_{1}+i_{2}+\cdots+i_{k}-k}
	\Big{[}\prod\limits_{s=1}^{k}
	\xi_{i_{s},i_{s+1},\ldots, i_{k},i_{1},\ldots,i_{s-1}}^{-2}-(-1)^{k}
	\Big{]}^{-1}.\label{mayer}
\end{eqnarray}
Assume now that we are able to prove that eigenvalues of $\mathcal{L}_{\omega}$ can be given by the series
\begin{eqnarray*}
\lambda_{n}(\omega)=(-1)^{n+1}\phi^{-2n}
\sum\limits_{\ell=0}^{\infty}\omega^{\ell}\cdot W_{\ell}(n),\quad |\omega|\leq 1, \, n\in\mathbb{N}.
\end{eqnarray*}
Expressing formally the trace in terms of such eigenvalues, yields
\begin{eqnarray}
	\mathrm{Tr}(\mathcal{L}_{\omega}^{k})&=&
	\sum\limits_{n=1}^{\infty}\lambda^{k}_{n}(\omega)
	=\sum\limits_{n=1}^{\infty}(-1)^{nk+k}\phi^{-2nk}
	\Big{(}\sum\limits_{\ell=0}^{\infty}\omega^{\ell}\cdot W_{\ell}(n)\Big{)}^{k}
	\nonumber\\
	&=&\sum\limits_{\ell=0}^{\infty}\omega^{\ell}\sum\limits_{n=1}^{\infty}(-1)^{nk+k}\phi^{-2nk}
	\sum\limits_{j_{1}+\cdots +j_{k}=\ell}W_{j_{1}}(n)\cdots 
	W_{j_{k}}(n).\label{mano}
\end{eqnarray}
All we have to do now is only to compare the corresponding coefficients at $\omega^{\ell}$ of \eqref{mayer} and \eqref{mano}.\\

Surely, such manipulations are formal and seem to provide no new information. However, one of the main discoveries of this study is that coefficients $W_{\ell}(n)$, after the \emph{golden transformation} (see Subsection \ref{my-ideas}), can be defined inductively without assuming their relation to any eigenvalues or operators whatsoever. For example, 
\begin{eqnarray*}
	W_{2}(3)&=&\frac{290^2}{2^{12}\phi^{12}}+\sum\limits_{s\neq 3}\frac{25\cdot s(25s^2-45s+26)^2}
{3\cdot 2^{2s+8}\phi^{12}\big{(}\phi^{2s-6}-(-1)^{s+1}\big{)}}\\
&=&0.120655345445838262410173716774122740053340142551806782250004_{+}.
\end{eqnarray*}
\begin{maintheorem}[Arithmetic and decomposition of trace formulas]\label{thm:decomposition}
The following statements hold.
\begin{enumerate}[label=(\roman*)]
\item[\emph{(i)}] There exist canonically determined coefficients $W_\ell(n)$,
with
\[
W_0(n)=1,\qquad W_1(n)=K(n,n)
=\frac54\,\phi^{-2n}P_{n-1}^{(0,1)}(3/2),
\]
such that, for every $n\ge1$,
\begin{equation*}
\lambda_n
=(-1)^{n+1}\phi^{-2n}\sum_{\ell=0}^\infty W_\ell(n),
\end{equation*}
the series being absolutely convergent. 

\item[\emph{(ii)}] This decomposition is compatible with the classical trace
formulas. For the first, second and third powers, respectively, we have (all summation indices are $\geq 1$)
\begin{align*}
\sum_{n\ge1}(-1)^{n+1}\phi^{-2n}W_{\ell-1}(n)
&=\frac1{\xi_\ell^{-2}+1},
&&\ell\ge1,\\
\sum_{\substack{i+j=\ell}}\ \sum_{n\ge1}
\phi^{-4n}W_{i-1}(n)W_{j-1}(n)
&=\sum_{\substack{i+j=\ell}}
\frac1{(\xi_{i,j}\xi_{j,i})^{-2}-1},
&&\ell\ge2,\\
\sum_{{n\ge1}\atop{i+j+h=\ell}}
(-1)^{n+1}\phi^{-6n}W_{i-1}(n)W_{j-1}(n)W_{h-1}(n)
&=\sum_{\substack{i+j+h=\ell}}
\frac1{(\xi_{i,j,h}\xi_{j,h,i}\xi_{h,i,j})^{-2}+1},
&&\ell\ge3.
\end{align*}
Every sum over $n$ in these fixed-degree identities converges absolutely.
Analogously, for every $k\ge1$ and $r\ge0$, put
$W_r^{[k]}(n)=[\omega^r]\mathscr W_n(\omega)^k$. Then
\begin{equation}\label{eq:graded-main}
\mathfrak T^{\mathrm{word}}_{k,r}
=[\omega^r]\Tr_V(\mathcal L_\omega^k)
=\sum_{n\ge1}(\lzero n)^kW_r^{[k]}(n).
\end{equation}
Here the total digit sum is $\ell=k+r$.
\end{enumerate}
\end{maintheorem}

For part (i), the exact $W_1$ formula is proved in Section~\ref{sec:trace};
absolute endpoint convergence for every $n\ge1$ is proved in
Subsection~\ref{subsec:all-index-wiener}, and actual-index identification in
Subsection~\ref{subsec:global-canonical-spectrum}. Part (ii) is proved locally in
the parameter $\omega$, independently of those endpoint arguments.
On a common disk about zero one has the normally convergent identity
\[
\Tr_V(\mathcal L_\omega^k)
=\sum_{n\ge1}\widehat\Lambda_n(\omega)^k,
\]
where the scalar continuations $\widehat\Lambda_n$ agree with the
normalized local germs near zero. Coefficient comparison gives
\eqref{eq:graded-main}.

So the fundamental trace formulas are split into a countable number
of formulas each. To see the idea without any multi-index notation,
consider the infinite matrix
\begin{equation*}
\left(\lzero n W_{L-1}(n)\right)_{n,L\ge1}
=
\begin{pmatrix}
\phi^{-2}&\phi^{-2}W_1(1)&\phi^{-2}W_2(1)&\cdots\\
-\phi^{-4}&-\phi^{-4}W_1(2)&-\phi^{-4}W_2(2)&\cdots\\
\phi^{-6}&\phi^{-6}W_1(3)&\phi^{-6}W_2(3)&\cdots\\
\vdots&\vdots&\vdots&\ddots
\end{pmatrix}.
\end{equation*}
Its first three column sums are
\[
\frac1{\phi^2+1},\qquad
\frac1{4+2\sqrt2},\qquad
\frac2{13+3\sqrt{13}}.
\]
In general the $L$th column adds up to $(\xi_L^{-2}+1)^{-1}$.
For every $n\ge1$, the $n$th row adds up absolutely
to $\lambda_n$. Adding the column sums gives exactly
\[
\sum_{L\ge1}\left(\sum_{n\ge1}\lzero n W_{L-1}(n)\right)
=\sum_{L\ge1}\frac1{\xi_L^{-2}+1}
=\Tr\mathcal L.
\]
This is the concrete meaning of decomposition of the trace formula.
The order of summation is initially obtained from the graded trace identity.
The absolute endpoint Fubini corollary after
Theorem~\ref{thm:all-index-endpoint} shows that the double series is in fact
absolutely summable, so the row and column summations may also be
interchanged at $\omega=1$.
If the $L$th column is first multiplied by $\omega^{L-1}$ with
$|\omega|\le1$, the same corollary shows that the double series is absolutely
convergent, its rows are the continued local eigenvalue series
$\widehat\Lambda_n(\omega)$, and either order gives
$\Tr\mathcal L_\omega$.

For products of coefficients we retain the notation
\[
\Omega_{\ell_1,\ldots,\ell_s}(w)
=\sum_{n\ge1}\prod_{j=1}^sW_{\ell_j}(n)w^n,\qquad w_k=\rho^k.
\]
Repeated indices are retained. A partition is written
$\boldsymbol\ell=(\ell_1,\ldots,\ell_s)\vdash r$; it is not an eigenvalue.
The additional low-excess and arbitrary-excess refinements, including
separation by binomial degree, are preserved in the frozen supplementary
material. The independent kernel-pair identity appears among the examples
in Section~\ref{app:examples}; its full proof is in Appendix~\ref{app:kernel-collapse}.

\subsection{Main ideas of the proof}\label{my-ideas}
There few essential discoveries (made in \cite{alkauskas2012}, v1) which are responsible for opening the gates to the complete proofs of Theorem \ref{thm:target} and \ref{thm:decomposition}.\\

First, introduction of $\omega$, that is, passing from operator $\mathcal{L}$ to $\mathcal{L}_{\omega}$, is a natural thing to do, if we want to integrate rational functions $e_n(z)$ and immediately anchor at functions defined in terms of $\phi$. Next, consider a formal functional equation satisfied by the eigenfunction (denoted by  $f(\omega,z)$) of the operator  $\mathcal{L}_{\omega}$. Such equation reads as
\begin{eqnarray}
	\lambda(\omega)f(\omega,z)=\frac{1}{(z+1)^2}f\Big{(}\omega,\frac{1}{z+1}\Big{)}+
	\omega\lambda(\omega)f(\omega,z+1).\label{functi}
\end{eqnarray}
\label{sub-scheme}
Imagine that both $\lambda(\omega)$ and $f(\omega,z)$ are analytic functions in $\omega$. We could obtain a recurrence for the coefficients only of arguments $z$ and $\frac{1}{z+1}$ are force, by a certain transformation, into proportional variables. And such transformations exists. In the appendices, it its dubbed as ``the golden transformation''. Namely, let us rewrite the above functional equation for the eigenfunction $f$ in terms of a function $g$, where

\begin{eqnarray*}
	f(z)=\frac{1}{(z+\phi)^2}\,g\Big{(}\frac{z-\phi^{-1}}{z+\phi}\Big{)}.	\end{eqnarray*}
Let
\begin{eqnarray*}
	x=\frac{z-\phi^{-1}}{z+\phi},\quad z=\frac{x\phi+\phi^{-1}}{1-x},\quad \Re(z)>-\frac{1}{2}\text{ corresponds to }|x|<1.
\end{eqnarray*}
We get an avatar of equation \eqref{functi}:
\begin{eqnarray}
		\lambda(\omega)g(\omega,x)=\phi^{-2}g(\omega,-\phi^{-2}x)+
		\frac{5\omega\lambda(\omega)}{(-x+2\phi)^2}\,
		g\Big{(}\omega,\frac{2x\phi^{-1}+1}{-x+2\phi}\Big{)},\label{functi-g}
\end{eqnarray}
	$g(\omega,x)$ is defined and is analytic in $|\omega|\leq 1$, $|x|<1$, satisfies a certain regularity condition.\\
For the rest of the paper, we fix the following notation
\begin{eqnarray*}
		\beta(x)=\frac{1+2\phi^{-1}x}{2\phi-x}\Longrightarrow
		\beta'(x)=\frac5{(2\phi-x)^2}.
\end{eqnarray*}	
	
And now the formal induction works. There is some normalization needed, since \eqref{functi} remains correct if multiplied by any analytic function in $\omega$. Rather than being a drawback, this gives a more important new corollary: each eigenvalue of the operator $\mathcal{L}$ has a canonical normalization. As the last term of \eqref{functi-g} suggest, the problem is now controlled by Taylor coefficients
\begin{eqnarray*}
K(j,\ell)=[t^{j-1}]\frac{(2x\phi^{-1}+1)^{\ell-1}}{(-x+2\phi)^{\ell+1}},\quad j,\ell\in\mathbb{N}.
	\end{eqnarray*}
Let us normalize
\begin{equation}\label{eq:integer-kernel}
 K(j,\ell)=\frac5{(2\phi)^{j+\ell}}\mathbf K(j,\ell)\Longrightarrow	\mathbf{K}(j,\ell)=[t^{j-1}]\frac{(1+4t)^{\ell-1}}{(1-t)^{\ell+1}}\in\mathbb{N}.
\end{equation}
The golden transformation reveals that the whole information about eigenvalues of $\mathcal{L}$ is encoded by the golden section $\phi$ and the integers $\mathbf{K}(j,\ell)$. A double series has a particularly simple expression:
\begin{eqnarray*}
\Theta_{\mathbf K}(x,y)=\sum\limits_{j,\ell\ge1}\mathbf K(j,\ell)x^{j-1}y^{\ell-1}=\frac{1}{(1-x)(1-x-y-4xy)}.
\end{eqnarray*}
In other words, the main ideas are olympiad-like, though many technical difficulties remain. That is where AI was most helpful.
\begin{table}[h]
	\begin{center}\small
		\begin{tabular}{|| c | c | c | c | c | c | c | c | c | c | c ||}
			\toprule
			$j\backslash\ell$ & 1 & 2 & 3 & 4 & 5 & 6 & 7 & 8 & 9 & 10 \\ \midrule
			
			$1$ & \fbox{$1$} & $1$ & $1$ & $1$ & $1$ & $1$ & $1$ & $1$ & $1$ & $1$ \\
			
			$2$ & $2$ & \fbox{$7$} & $12$ & $17$ & $22$ & $27$ & $32$ & $37$ & $42$ & $47$ \\
			
			$3$ & $3$ & $18$ & \fbox{$58$} & $123$ & $213$ & $328$ & $468$ & $633$ & $823$ & $1038$ \\
			
			$4$ & $4$ & $34$ & $164$ & \fbox{$519$} & $1224$ & $2404$ & $4184$ & $6689$ & $10044$ & $14374$ \\
			
			$5$ & $5$ & $55$ & $355$ & $1530$ & \fbox{$4830$} & $12130$ & $25930$ & $49355$ & $86155$ & $140705$ \\
			
			$6$ & $6$ & $81$ & $656$ & $3606$ & $14556$ & \fbox{$46006$} & $120456$ & $273531$ & $557106$ & $1042431$ \\
			
			$7$ & $7$ & $112$ & $1092$ & $7322$ & $36302$ & $140532$ & \fbox{$445012$} & $1200367$ & $2851597$ & $6122452$ \\
			
			$8$ & $8$ & $148$ & $1688$ & $13378$ & $78968$ & $364708$ & $1371848$ & \fbox{$4352263$} & $12005328$ & $29534168$ \\
			
			$9$ & $9$ & $189$ & $2469$ & $22599$ & $155079$ & $835659$ & $3666339$ & $13505994$ & \fbox{$42920374$} & $120475854$ \\
			
			$10$ & $10$ & $235$ & $3460$ & $35935$ & $281410$ & $1737385$ & $8746360$ & $36917710$ & $133862060$ & \fbox{$426019410$} \\
			
			\bottomrule
		\end{tabular}
	\end{center}
	\caption{Integer matrix $\{\m{K}(j,\ell)$, $j,\ell\in\mathbb{N}\}$. This matrix, coupled with the constant $\phi$, contains all the information on eigenvalues of $\mathcal{L}$, trace formulas, and decompositions of trace formulas.}
	\label{table-K}
\end{table}
The boxed diagonal sequence is now labelled A386398 in OEIS \cite{oeis}.\\

\subsection{Riemann zeros, Maass forms, and canonical descent}
\label{subsec:canonical-spectral-descent}
The canonical indexing also suggests a finer question in the spectral theory of the
modular surface. Extend the weighted operator by
\[
(\mathcal L_{s,\omega}f)(z)
=\sum_{m\ge1}\frac{\omega^{m-1}}{(z+m)^{2s}}
f\!\left(\frac1{z+m}\right),\qquad
\Re s>\tfrac12,\quad |\omega|\le1.
\]
Thus $\mathcal L_{1,\omega}=\mathcal L_\omega$ and
$\mathcal L_{1,1}=\mathcal L$. For $\omega=1$ this is the Mayer--Ruelle family,
meromorphically continued in $s$. The Selberg zeta function of
$\Gamma=\mathrm{PSL}_2(\mathbb Z)$, initially defined by an Euler product over
primitive closed geodesics, satisfies Mayer's identity
\[
Z_\Gamma(s)=\det(I-\mathcal L_{s,1}^{\,2})
=D_+(s)\det(I+\mathcal L_{s,1}),\qquad
D_+(s):=\det(I-\mathcal L_{s,1}).
\]
The factor $D_+$ is the even Selberg factor, equivalently the zeta function of
the extended modular group $\mathrm{PGL}_2(\mathbb Z)$; this refinement is due to Efrat
\cite{mayer1991selberg,zagier,Efrat1993}.

Two arithmetic sources of its zeros are particularly relevant. An even Maass
cusp form with Laplace eigenvalue $\tfrac14+r^2$ gives a zero at
$s=\tfrac12+ir$. Nontrivial Riemann zeros enter through the Eisenstein
scattering coefficient
\[
\mathfrak S(s)=\sqrt\pi\,
\frac{\Gamma(s-\tfrac12)\zeta(2s-1)}{\Gamma(s)\zeta(2s)}:
\]
a zero $\varrho$ of $\zeta$ gives a scattering pole at $s=\varrho/2$
and a zero of $D_+$. In both cases $\mathcal L_{s,1}$ has eigenvalue $+1$
\cite[Chapter IV]{lewiszagier}. Thus scattering and cuspidal spectral
data meet in the same Fredholm determinant.

Theorem~5.7 supplies canonical base labels $\lambda_n$ at $s=1$.
Write $\lambda_n(s)$ for their local continuations, with $\lambda_n(1)=\lambda_n$
and subsequent paths specified. Continue the eigenvalue $+1$ from a parameter $\sigma$ in the upper half-plane
along the prescribed path $\sigma\longrightarrow1+i\Im\sigma\longrightarrow1$,
and record the label of its endpoint. This convention requires no Riemann
hypothesis; it does require unambiguous continuation along the path.
For the first twenty Riemann zeros $\varrho_j$, ordered by increasing positive
imaginary part, and the first twenty even Maass parameters
$\tfrac12+ir_j$, ordered by increasing $r_j>0$, numerical continuation gives,
respectively,
\[
\begin{aligned}
	(t_{\rm R}(j))_{j=1}^{20}
	&=(1,2,1,2,1,4,1,3,2,3,\;6,1,6,5,4,5,2,6,7,6),\\
	(t_{\rm M}(j))_{j=1}^{20}
	&=(3,3,2,5,4,1,5,5,2,3,\;4,7,1,1,6,3,7,8,4,1).
\end{aligned}
\]
Here $t_{\rm R}(j)$ refers to $\sigma=\varrho_j/2$, not $\varrho_j$.
The assignments were checked by independent discretizations and refinement;
the Maass inputs were numerically refined from published tables
\cite{SiddigZainuddin2009}. These are numerical observations, not certified
continuation theorems. The indexed Riemann-zero question was posed in
\cite[Problem 4]{AlkauskasProblems2016}.

The path matters. Put $s_4=\varrho_4/2$. The computed upper descent reaches $\lambda_2$,
whereas the straight descent reaches $\lambda_3$. Numerically, the same simple
local eigenvalue at $s_4$ is connected to different base labels: continuation
around
\[
s_*\approx0.91233284937+12.82164766060\,i
\]
exchanges labels $2$ and $3$ at a square-root collision with
$\lambda_*\approx-0.05665861945+0.27543753283\,i$.
Here $D_+(s_*)\ne0$, so this phenomenon does not produce a double zero of the
Selberg factor. It concerns the spectral equation
$\det(I-\xi\mathcal L_{s,1})=0$, with $\xi=1/\lambda$, rather than only its
slice $\xi=1$. Whether these descent distributions and their monodromy reflect
a further arithmetic structure is a question for Part~III. None of these
numerical observations is used in the proofs of the present paper.

\section{Examples of decompositions}\label{app:examples}
\subsection{Initial cases of decomposition}\label{v-lyg-2}
The decomposition identities of Theorem \ref{thm:decomposition}(ii) can be also numerically verified. This also reveals why partial decompositions hold and where cancellation occurs.

\noindent\fbox{$\ell=2$, $k\in\mathbb{N}$.}
These decompositions only involve function $W_{1}(n)$ and quadratic irrationals with partial quotient $1$. That is, 

Using the generating function for Jacobi polynomials (\ref{jac-gen}), we obtain:
\begin{eqnarray*}
	\sum\limits_{n=1}^{\infty}(-1)^{n+1}\phi^{-2n}=\frac{1}{\phi^2+1}&=&(-1)^1\Omega_{\varnothing}(-\phi^{-2}),\\
	\frac{5}{4}\sum\limits_{n=1}^{\infty}(-1)^{n+1}\phi^{-4n}P_{n-1}^{(0,1)}(3/2)&=&\frac{1}{4+2\sqrt{2}}=(-1)^1\Omega_{1}(-\phi^{-2}),\\
	\frac{5}{4}\sum\limits_{n=1}^{\infty}\phi^{-6n}P_{n-1}^{(0,1)}(3/2)&=&\frac{1}{6+4\sqrt{3}}=(-1)^2\Omega_{1}(\phi^{-4}),\\
	\frac{5}{4}\sum\limits_{n=1}^{\infty}(-1)^{n+1}\phi^{-8n}P_{n-1}^{(0,1)}(3/2)&=&\frac{1}{20+6\sqrt{10}}=(-1)^3\Omega_{1}(-\phi^{-6}).
\end{eqnarray*}
In general,
\begin{eqnarray}
	(-1)^{k}\Omega_{1}\Big{(}(-1)^{k}\phi^{-2k}\Big{)}=\frac{1}{\psi_{k}},\quad\psi_{k}=2F_{k+1}^{2}-2(-1)^{k}+2F_{k+1}\sqrt{F_{k+1}^{2}-(-1)^{k}}.
	\label{pirmas-l}
\end{eqnarray} 
Here $F_{k}$ is the standard Fibonacci sequence with the (standard) seed values $F_{1}=F_{2}=1$.\\

\noindent\fbox{$\ell=3$, $k=1$.} Though we do not have explicit values for $W_{2}(n)$ in a closed form, yet numerical values can be calculated easily to a very high precision:
\begin{eqnarray*}
	W_{2}(1)&=&\frac{25}{16\phi^{4}}+\sum\limits_{s\neq 1}\frac{25s}{4^{s+1}\phi^{4}\big{(}\phi^{2s-2}-(-1)^{s+1}\big{)}}\\
	&=&0.267632296633129768709622137665599100606923184706474041907380_{+},\\
	W_{2}(2)&=&\frac{35^2}{2^{8}\phi^{8}}+\sum\limits_{s\neq 2}\frac{s(25s-15)^2}
	{2^{2s+5}\phi^{8}\big{(}\phi^{2s-4}-(-1)^{s}\big{)}}\\
	&=&0.160274183422101231629872112692281197699627281556900277836798_{+}.
\end{eqnarray*}
(The value of $W_2(3)$ was given just before Theorem \ref{thm:decomposition}.)
In general, 
\begin{eqnarray*}
	W_{2}(n)=25\frac{Y_{n}(n)^{2}}{(2\phi)^{4n}}+\frac{25}{n}\sum\limits_{s\neq n}\frac{sY_{n}(s)^2}{2^{2s+2n}(\phi^{2s+2n}-(-1)^{s+n}\phi^{4n})},
\end{eqnarray*}
where polynomials $Y_{s}$ are introduced in Appendix \ref{funct-Y}. In particular, the first identity of Theorem \ref{thm:decomposition} (ii) for $\ell=3$ gives the following theoretical equality, which is also verified with high-precision numerics:
\begin{eqnarray*}
\mathsf{P}=\sum\limits_{n=1}^{\infty}(-1)^{n+1}W_{2}(n)\phi^{-2n}=0.0839748528310781_{+}=\frac{2}{13+3\sqrt{13}}.
\end{eqnarray*}
In fact, we need neither operator nor trace formulas to strictly prove the last identity. The mechanism behind an independent proof is a \emph{kernel cancellation}. In essence, this means the disappearance of denominators, if $\mathsf{P}$ is written as a double sum in terms of $Y_n(s)$. The next example is better suited to expose this phenomenon.

\noindent\fbox{$\ell=4$, $k=2$.} A still more interesting example occurs in the second identity of (ii) in case $\ell=4$. Numerically, we have:
\begin{eqnarray*}
\mathsf A=\sum\limits_{n=1}^{\infty}W_{2}(n)\phi^{-4n}=0.0428848639793538_{+},\quad
\mathsf B=\sum\limits_{n=1}^{\infty}W^{2}_{1}(n)\phi^{-4n}=0.0356498091111648_{+},
\end{eqnarray*}
\begin{eqnarray*}
\mathsf	C=\frac{2}{21+5\sqrt{21}}=0.0455447255899809_{+},\quad
\mathsf	D=\frac{1}{16+12\sqrt{2}}=0.0303300858899106_{+}.
\end{eqnarray*}
And thus, we get the correct identity 
\begin{eqnarray}
	2\mathsf A+\mathsf B=2\mathsf C+\mathsf D=0.1214195370698725_{+}. \label{a-plus-b}
\end{eqnarray}
This identity does not decompose into two algebraic identities, since there is no \emph{a priori} reason why $	\mathsf A$ or $	\mathsf B$ should be algebraic numbers. In fact, since generating function of coefficients $W_1(n)$  is quadratic algebraic by Proposition \ref{prop:exact-kernel}, the constant $	\mathsf B$ can be expressed in terms of complete elliptic integrals of the first and the third kind. Namely, 
\begin{eqnarray*}
	4\pi\cdot\Big{(}\mathsf B+\frac{1}{2}\Big{)}
	=\frac{4\sqrt{5}}{3}\cdot  K\Big{(}\frac{1}{3}\Big{)}+
	\frac{2}{3}\cdot \Pi\Big{(}\frac{\phi^{2}}{3},\frac{1}{3}\Big{)}-\frac{2}{3}\cdot\Pi\Big{(}\frac{\phi^{-2}}{3},\frac{1}{3}\Big{)}.
\end{eqnarray*}
We recall that
\begin{eqnarray*}
	\Pi(\nu,k)=\int\limits_{0}^{1}\frac{\d t}{(1-\nu t^2)\sqrt{(1-t^2)(1-k^2 t^2)}},
	\quad K(k)=\Pi(0,k).
\end{eqnarray*}
According to the notion of Zagier-Kontsevitch \cite{konzag}, the number $	\mathsf B$ is \emph{arithmetic} in the sense that it belongs to the extended ring of periods $\widehat{\mathcal{P}}=\mathcal{P}[\frac{1}{\pi}]$, and then so does $ \mathsf A$, since their sum is algebraic. In general, we will see that
\begin{eqnarray*}
	\Omega_{2}\Big{(}(-1)^{k}\phi^{-2k}\Big{)}\in\widehat{\mathcal{P}}\text{ for }k\in\mathbb{N}.
\end{eqnarray*}
Appendix \ref{app:kernel-collapse} contains a direct proof of \eqref{a-plus-b} where, as mentioned, the phenomenon behind the identity is exposed.

\noindent\fbox{$\ell=5$, $k=3$.} Further, the last identity of (ii) in case $\ell=5$ gives the following. Let us define
\begin{eqnarray*}
\mathsf	E&=&\sum\limits_{n=1}^{\infty}(-1)^{n+1}W_{2}(n)\phi^{-6n}=0.0144368544431652_{+},\\
\mathsf	F&=&\sum\limits_{n=1}^{\infty}(-1)^{n+1}W^{2}_{1}(n)\phi^{-6n}=0.0123983653921924_{+},\\
\mathsf	G&=&\frac{1}{34+8\sqrt{17}}=0.0149287499273340_{+},
	\quad\mathsf H=\frac{2}{85+9\sqrt{85}}=0.0119064699080236_{+}.
\end{eqnarray*}
Then
\begin{eqnarray*}
	E+F=G+H=0.0268352198353576_{+}.
\end{eqnarray*}
 As noted, $\mathsf E,\mathsf F\in\widehat{\mathcal{P}}$.\\

\noindent\fbox{$\ell=6$, $k=4$.} Now the sum of three quadratic irrationals appears. Indeed, let
\begin{eqnarray*}
\mathsf	I&=&\sum\limits_{n=1}^{\infty}W_{2}(n)\phi^{-8n}=0.0057706896109551_{+},\\
\mathsf	J&=&\sum\limits_{n=1}^{\infty}W^{2}_{1}(n)\phi^{-8n}=0.0048993082411535_{+},
\end{eqnarray*}
\begin{eqnarray*}
\mathsf	K&=&\frac{2}{165+13\sqrt{165}}
	=\frac{1}{(\xi_{1,1,1,3}\xi_{1,1,3,1}\xi_{1,3,1,1}\xi_{3,1,1,1})^{-2}-1}=0.0060243137049899_{+},\\
\mathsf	L&=&\frac{2}{221+15\sqrt{221}}=
	\frac{1}{(\xi_{2,2,1,1}\xi_{2,1,1,2}\xi_{1,1,2,2}\xi_{1,2,2,1})^{-2}-1}=0.00450459549723434_{+},\\
\mathsf	M&=&\frac{1}{96+56\sqrt{3}}=
	\frac{1}{(\xi_{2,1,2,1}\xi_{1,2,1,2}\xi_{2,1,2,1}\xi_{1,2,1,2})^{-2}-1}
	=0.0051814855409225_{+}.
\end{eqnarray*}
Numerical values also confirm that
\begin{eqnarray*}
2\mathsf I+3\mathsf J=2\mathsf K+2\mathsf L+\mathsf M=0.0262393039453711_{+}.
\end{eqnarray*}
We will finish with the last three examples, now all involving $W_{3}$.\\

\noindent\fbox{$\ell=4$, $k=1$} Using the formula for $W_3(n)$, we  get the correct value:
\begin{eqnarray*}
	\sum\limits_{n=1}^{\infty}(-1)^{n+1}W_{3}(n)\phi^{-2n}=0.052786404500042_{+}=
	\frac{1}{10+4\sqrt{5}}=\frac{1}{\xi^{-2}_{4}+1}.
\end{eqnarray*}  
\fbox{$\ell=5$, $k=2$.} Next, let us define
\begin{eqnarray*}
\mathsf N&=&\sum\limits_{n=1}^{\infty}W_{3}(n)\phi^{-4n}=0.0268901808819348_{+},\\
\mathsf	O&=&\sum\limits_{n=1}^{\infty}W_{1}(n)W_{2}(n)\phi^{-4n}=0.01983768450229800_{+},
\end{eqnarray*}
\begin{eqnarray*}
\mathsf	P&=&\frac{1}{(\xi_{1,4}\xi_{4,1})^{-2}-1}=\frac{1}{16+12\sqrt{2}}=0.0303300858899106_{+},\\
\mathsf	Q&=&\frac{1}{(\xi_{2,3}\xi_{3,2})^{-2}-1}=\frac{1}{30+8\sqrt{15}}=0.0163977794943222_{+}.
\end{eqnarray*}  
Note that, according to our previous formula, $\xi_{2,2}\xi_{2,2}=\xi_{1,4}\xi_{4,1}$. Thus, we get the correct identity:
\begin{eqnarray*}
\mathsf	N+\mathsf O=\mathsf P+\mathsf Q=0.0467278653842328_{+}.
\end{eqnarray*}
However, contrary to $\Omega_{2}(\phi^{-4})$, we cannot yet claim that $\Omega_3(\phi_4)$ is an extended period. The investigation on the genuine arithmetic nature of this and other constants emerging from Theorem \ref{thm:decomposition} will constitute an important part of the second, forthcoming part of this study.

\noindent\fbox{$\ell=6$, $k=3$.} This is the final example. Let
\begin{eqnarray*}
\mathsf	R&=&\sum\limits_{n=1}^{\infty}(-1)^{n+1}W_{3}(n)\phi^{-6n}=0.0090670451347138_{+},\\
\mathsf	S&=&\sum\limits_{n=1}^{\infty}(-1)^{n+1}W_{1}(n)W_{2}(n)\phi^{-6n}=0.0069673221927849_{+},\\
\mathsf	T&=&\sum\limits_{n=1}^{\infty}(-1)^{n+1}W^{3}_{1}(n)\phi^{-6n}=0.0059673995757538_{+},
\end{eqnarray*}
\begin{eqnarray*}
\mathsf	U&=&\frac{1}{52+10\sqrt{26}}=\frac{1}{(\xi_{1,1,4}\xi_{1,4,1}\xi_{4,1,1})^{-2}+1}=0.0097096621545399_{+},\\
\mathsf	V&=&\frac{1}{74+12\sqrt{37}}=\frac{1}{(\xi_{1,2,3}\xi_{2,3,1}\xi_{3,1,2})^{-2}+1}=0.0068030380839281_{+},\\
\mathsf	W&=&\frac{1}{74+12\sqrt{37}}=\frac{1}{(\xi_{1,3,2}\xi_{3,2,1}\xi_{2,1,3})^{-2}+1}=0.0068030380839281_{+},\\
\mathsf	X&=&\frac{1}{100+70\sqrt{2}}=\frac{1}{(\xi_{2,2,2}\xi_{2,2,2}\xi_{2,2,2})^{-2}+1}=0.0050252531694167_{+}.\\
\end{eqnarray*}
The correct identity, confirmed numerically, is
\begin{eqnarray*}
3\mathsf R+6\mathsf S+\mathsf T=3\mathsf U+3\mathsf V+3\mathsf W+\mathsf X=0.0074972468136605_{+}.
\end{eqnarray*}

\section{Proof of decomposition formulas}\label{sec:trace}
\subsection{The weighted operator family and the direct trace bridge}\label{sec:weighted}\label{t2:sec:shared}
On the space $V=A(D_z)$ fixed in the Introduction, define
\[
 h_m(z)=\frac1{z+m},\quad
 (B_mf)(z)=\frac{f(h_m(z))}{(z+m)^2},\quad
 \Lop_\omega=\sum_{m\ge1}\omega^{m-1}B_m,\qquad  A_m=\begin{pmatrix}0&1\\1&m\end{pmatrix}.
\]
Further, let
\[
 \eta_{\mathbf a}
 =\frac{t_{\mathbf a}+\sqrt{t_{\mathbf a}^2-4\varepsilon_k}}2,
 \qquad
 \Cw(\mathbf a)=\frac1{\eta_{\mathbf a}^2-\varepsilon_k}.
\]
The square root in this real formula is positive.  With the same excess
$e(\mathbf a)$,
\[
 \Tword_{k,r}=\sum_{\substack{\mathbf a\in\N^k\\e(\mathbf a)=r}}
 \Cw(\mathbf a)
\]
Is the notation used in the Section \ref{sec:intro}.
This is a finite sum.  As usual,
$\xi_{a_1,\ldots,a_k}=[0,\overline{a_1,\ldots,a_k}]$; the displayed
period need not be minimal.  If $p_j$ are its cyclic shifts, then
$p_j=(a_j+p_{j+1})^{-1}$, with cyclic indices.  Applying the matrices
$A_{a_j}$ to $(p_{j+1},1)^t$ gives
\begin{equation}\label{t2:eq:periodic-weight}
 \eta_{\mathbf a}=\prod_{j=1}^k p_j^{-1},\qquad
 \Cw(\mathbf a)
 =\left(\prod_{j=1}^k p_j^{-2}-(-1)^k\right)^{-1}.
\end{equation}

Recall definition of coefficients $K(j,i)$. We will sue the following normalization:
\begin{equation}\label{t2:eq:kernel}
 K(j,i)=[x^{j-1}]\beta'(x)\beta(x)^{i-1},\qquad
 H_{ji}=\sqrt{\frac ij}\,K(j,i).
\end{equation}
\subsubsection{Nuclearity and the trace convention}

\begin{lemma}[Uniform nuclear family]\label{t2:lem:nuclear}
The family $\Lop_\omega$ is holomorphic in nuclear norm for $|\omega|<1$
and continuous in nuclear norm for $|\omega|\le1$.  Each $\Lop_\omega$
is $2/3$-nuclear, with a uniformly bounded $2/3$-summable representation
on the closed unit disk.
\end{lemma}
\begin{proof}
For $z\in\overline D_z$,
\[
 |z+m|\ge m-\tfrac12,\qquad
 |h_m(z)-1|\le1.
\]
The second inequality follows from
$|z+m|^2-|z+m-1|^2=2\Re z+2m-1\ge0$.
Write
\[
 \ell_\nu(f)=\frac{f^{(\nu)}(1)}{\nu!},\qquad
 \norm{\ell_\nu}\le(2/3)^\nu,
 \qquad \kappa_m=(m-\tfrac12)^{-2}.
\]
Taylor expansion about $1$ yields
\[
 B_m=\sum_{\nu\ge0}\ell_\nu\otimes
 \frac{(h_m-1)^\nu}{(\,\cdot+m)^2}.
\]
Thus $\norm{B_m}\le\kappa_m$ and its nuclear norm satisfies
$\nu(B_m)\le3\kappa_m$.  Since $\sum_m\kappa_m<\infty$, the defining digit
series converges uniformly in nuclear norm on $|\omega|\le1$, proving
the asserted continuity and interior holomorphy.

For the stronger nuclearity, group the digit terms first:
\[
 b_{\nu,\omega}(z)=\sum_{m\ge1}\omega^{m-1}
 \frac{(h_m(z)-1)^\nu}{(z+m)^2},\qquad
 \Lop_\omega=\sum_{\nu\ge0}\ell_\nu\otimes b_{\nu,\omega}.
\]
If $A=\sum_m\kappa_m$, then $\norm{b_{\nu,\omega}}\le A$ uniformly in
$\nu$ and $|\omega|\le1$.  Consequently
\begin{equation*}
 \sum_{\nu\ge0}
 \bigl(\norm{\ell_\nu}\norm{b_{\nu,\omega}}\bigr)^{2/3}
 \le A^{2/3}\sum_{\nu\ge0}(2/3)^{2\nu/3}<\infty.
\end{equation*}
\end{proof}

The disk algebra has the metric approximation property: after rescaling
to the unit disk, its Fej\'er polynomial means are finite-rank
contractions converging to the identity on every function.  Uniform
boundedness then gives convergence uniformly on compact sets of the
Banach space.  Hence the nuclear trace on $V$ is intrinsic and satisfies
$|\Tr T|\le\nu(T)$.

We use the following standard spectral-trace facts.  For a Hilbert-space
trace-class operator, the trace equals the sum of the eigenvalues with
algebraic multiplicity \cite{simon,gohberg}. For a $2/3$-nuclear Banach-space operator, the
Grothendieck trace theorem gives the analogous absolutely convergent
spectral trace and the entire Fredholm product.  In either case,
\begin{equation}\label{t2:eq:fredholm-convention}
 \det(I-zT)=\prod_j(1-z\mu_j(T)),\qquad
 \det(I-zT)=\exp\left(-\sum_{k\ge1}\frac{z^k}{k}\Tr T^k\right)
\end{equation}
for small $z$ in the second identity.  Zeros are reciprocals of nonzero
eigenvalues and carry their algebraic multiplicities.  Precise statements
are recorded in \cite[p.~19 and Appendix, Theorem~A1, p.~71]{mayer5};
see also \cite{grothendieck,pietsch,reinovlatif} for the spectral-trace assertions.
Powers of $\Lop_\omega$ retain the $2/3$-nuclear property by the ideal
property.  Ordinary nuclearity alone is not being used as a substitute
for this spectral-trace theorem.

\subsubsection{The Bergman model}

Let $\mathcal H=A^2(\D)$ with normalized area measure and orthonormal
basis $u_j(x)=\sqrt j\,x^{j-1}$.  Define
\[
 (Mg)(x)=\beta'(x)g(\beta(x)),\qquad
 (D_0g)(x)=|\rho|g(\rho x).
\]
The golden map
\[
 x=\frac{z-\phi^{-1}}{z+\phi}
\]
sends $\Re z>-1/2$ onto $\D$, and translation by $1$ becomes $\beta$.
Thus $\beta$ is univalent and $\beta(\D)\subset\D$.  Change of
variables gives
\begin{equation*}
 \norm{Mg}_{A^2}^2
 =\frac1\pi\int_{\beta(\D)}|g(y)|^2\,\dd A(y)
 \le\norm{g}_{A^2}^2.
\end{equation*}
Its matrix is exactly $H$ in \eqref{t2:eq:kernel}. Weighted symmetry
$iK(j,i)=jK(i,j)$, proved in Appendix~\ref{app:exact-kernel},
makes this bounded matrix symmetric; hence $M=M^*$.  Also
\[
 D_0u_j=\lzero{j}u_j,\qquad \norm{D_0}_1=\sum_{j\ge1}|\rho|^j=\frac{|\rho|}{1-|\rho|}.
\]
Therefore
\begin{equation}\label{t2:eq:R}
 \Aop_\omega=(I-\omega M)^{-1}D_0
 =\sum_{m\ge1}\omega^{m-1}C_m,\qquad
 C_m=M^{m-1}D_0,
\end{equation}
is a holomorphic trace-class family for $|\omega|<1$.  In coefficient coordinates write $\Hker=H$ and
$D=\diag(\lzero{1},\lzero{2},\ldots)$, so this same operator is
$(I-\omega\Hker)^{-1}D$.
Only $\norm{M}\le1$, not positivity or selfadjointness, is needed in
the local trace argument below.

\subsubsection{One weighted-composition calculation}

We recall the classical weighted-composition trace calculation
\cite[\S2]{mayer}, also \cite{ruelle1976,mayer1980}, simultaneously for the two spaces.
Suppose that, on a disk $B(c,R)$, the operator
$Tf=a\,f\circ\psi$ has bounded analytic weight and
$|\psi-c|\le r<R$.  On the disk algebra assume also continuity on the
closed disk.  Taylor expansion about $c$ proves nuclearity on the disk
algebra and trace-class membership on normalized Bergman space.  Indeed,
the coefficient functionals have norms at most $R^{-\nu}$ and
$\sqrt{\nu+1}R^{-\nu}$, respectively, while
$\norm{a(\psi-c)^\nu}\le\norm{a}_\infty r^\nu$.

In both spaces the trace is
$\sum_{\nu\ge0}[z^\nu]a(c+z)(\psi(c+z)-c)^\nu$;
for Bergman space the normalization factors cancel in each diagonal
entry.  Choosing $r<s<R$ and summing the uniformly convergent geometric
series on $|z-c|=s$ gives
\begin{equation}\label{t2:eq:composition-trace}
 \Tr T=\frac1{2\pi i}\int_{|z-c|=s}
 \frac{a(z)}{z-\psi(z)}\,\dd z
 =\frac{a(p)}{1-\psi'(p)}.
\end{equation}
Rouch\'e's theorem shows that $z-\psi(z)$ has exactly one zero inside
the contour, counted with multiplicity.  Every fixed point lies inside
$|z-c|\le r$, so $p$ is the unique fixed point and it is simple.

\begin{proposition}[Direct word trace and spectral bridge]
\label{t2:prop:bridge}
For every $k\ge1$,
\begin{align}
 \Tr_V\Lop_\omega^k
 &=\sum_{\mathbf a\in\N^k}\omega^{e(\mathbf a)}\Cw(\mathbf a),
 &&|\omega|\le1,\label{t2:eq:word-trace}\\
 \Tr_{\mathcal H}\Aop_\omega^k
 &=\Tr_V\Lop_\omega^k,
 &&|\omega|<1.\label{t2:eq:trace-bridge}
\end{align}
The word series is absolutely convergent on the indicated domain.
For each fixed $|\omega|<1$, $\Aop_\omega$ and $\Lop_\omega$ have identical
nonzero spectra with algebraic multiplicity.
\end{proposition}
\begin{proof}
For a word $\mathbf a$, the symbol of
$B_{a_1}\cdots B_{a_k}$ is
$\psi=h_{a_k}\circ\cdots\circ h_{a_1}$, and its weight is
$\varepsilon_k\psi'$.  Its symbol sends $\overline D_z$ into
$|z-1|\le1$.  The reversed matrix product is $\mathsf M(\mathbf a)^t$, since
each $A_m$ is symmetric; its trace, determinant and Perron multiplier
are those defining $\eta_{\mathbf a}$.  At the positive fixed point,
\[
 \psi'(p)=\frac{\varepsilon_k}{\eta_{\mathbf a}^2},\qquad
 \varepsilon_k\psi'(p)=\eta_{\mathbf a}^{-2}.
\]
Equation~\eqref{t2:eq:composition-trace} therefore gives the word trace
$\Cw(\mathbf a)$.

For the Bergman operators in \eqref{t2:eq:R}, put
$\theta_m(x)=\rho\beta^{m-1}(x)$.  Iteration here means composition.
A direct calculation gives
\[
 C_mg=-\theta_m'\,g\circ\theta_m,
 \qquad |\theta_m(x)|<|\rho|<1.
\]
The map $\theta_m$ is the golden-coordinate conjugate of $h_m$.
Also $|\beta'(x)|\le1$ on $\D$, so its iterated derivative, and hence
the weight of each $C_m$, is bounded.  Thus the same calculation
\eqref{t2:eq:composition-trace} applies to the Bergman word.
Conjugacy preserves the fixed-point derivative, so its trace is again
$\Cw(\mathbf a)$.  This use of the golden coordinate compares symbols;
it does not assume a bounded similarity between $V$ and $\mathcal H$.

The expansions may be traced termwise.  On $V$, the ideal property gives
\[
 \nu(B_{a_1}\cdots B_{a_k})
 \le3\prod_{j=1}^k\kappa_{a_j},\qquad
 \sum_{\mathbf a\in\N^k}\nu(B_{a_1}\cdots B_{a_k})\le3A^k.
\]
For the Bergman words,
$\norm{C_m}_1\le\norm{D_0}_1$ and $\norm{C_m}\le |\rho|$, so, for $0\le t<1$,
\[
 \sum_{\mathbf a\in\N^k}t^{e(\mathbf a)}
 \norm{C_{a_1}\cdots C_{a_k}}_1
 \le\frac{\norm{D_0}_1|\rho|^{k-1}}{(1-t)^k}<\infty.
\]
This proves \eqref{t2:eq:word-trace}--\eqref{t2:eq:trace-bridge}.

For fixed $|\omega|<1$, the logarithmic determinant expansions
\eqref{t2:eq:fredholm-convention} agree near $z=0$ by equality of all
power traces.  Both determinants are entire, hence
\begin{equation*}
 \det_{\mathcal H}(I-z\Aop_\omega)=\det_V(I-z\Lop_\omega)
 \qquad(z\in\C).
\end{equation*}
Their zeros give the asserted equality of nonzero spectra and algebraic
multiplicities.
\end{proof}

\begin{remark}[Logical position of the bridge]\label{t2:rem:bridge}
Proposition~\ref{t2:prop:bridge} contains no spectral coefficients and
uses no eigenvalue asymptotic.  Its first identity is the word/operator
formula; its second is the direct operator/operator bridge.  Neither is
the graded spectral identity of Theorem~2(ii).  They may therefore be
used as shared preliminaries in the proofs of both main theorems.
In particular, at $\omega=0$ the bridge and $\Aop_0=D_0$ already show
\[
 \spec(\Lop_0)\setminus\{0\}=\{\lzero{n}:n\ge1\},
\]
with every $\lzero{n}$ algebraically simple.
\end{remark}

\subsection{The base spectrum and canonical local germs}\label{sec:base}\label{sec:germs}\label{t2:sec:local}
\begin{lemma}[Uniform relative-resolvent localization]
\label{t2:lem:localization}
For $|\omega|<1/3$, every disk
\[
 \Delta_n=\{\lambda:|\lambda-\lzero{n}|<|\lzero{n}|/2\}
\]
contains exactly one eigenvalue of $\Aop_\omega$, counted with algebraic
multiplicity.  There are no other nonzero eigenvalues.  The same is true
of $\Lop_\omega$.  These eigenvalues, denoted by $\Lambda_n(\omega)$,
are holomorphic on $|\omega|<1/3$ and satisfy
\begin{equation}\label{t2:eq:local-bound}
 |\Lambda_n(\omega)-\lzero{n}|<\frac{|\lzero{n}|}{2},\qquad
 |\Lambda_n(\omega)|<\frac32|\rho|^n.
\end{equation}
\end{lemma}
\begin{proof}
The closed disks are pairwise disjoint and avoid zero.  Opposite-sign
centers have distance $|\lzero{n}|+|\lzero{m}|$, greater than the sum of the radii.
For same-sign centers, suppose $m>n$.  Then $m\ge n+2$ and
$|\lzero{m}|/|\lzero{n}|\le |\rho|^2<1/3$,
so $|\lzero{n}|-|\lzero{m}|>(|\lzero{n}|+|\lzero{m}|)/2$.

If $\lambda\ne0$ lies outside every open $\Delta_j$, then
\[
 |\lambda-\lzero{j}|\ge|\lzero{j}|/2,\qquad
 |\lambda|\le|\lambda-\lzero{j}|+|\lzero{j}|\le3|\lambda-\lzero{j}|.
\]
The distance to the remaining spectral point $0$ is $|\lambda|$.
Since $D_0$ is diagonal and normal,
\begin{equation*}
 |\lambda|\norm{(\lambda I-D_0)^{-1}}\le3.
\end{equation*}
Now
\[
 \lambda I-\Aop_\omega
 =(I-\omega M)^{-1}
 \bigl((\lambda I-D_0)-\omega\lambda M\bigr),
\]
and
$\norm{\omega\lambda M(\lambda I-D_0)^{-1}}\le3|\omega|<1$.
The Neumann series proves invertibility.  This proves spectral
containment and shows that every circle $\partial\Delta_n$ stays in
the resolvent set throughout the parameter disk.

Thus the Riesz projection (see \cite[III, \S6.4--6.5; VII, \S1.3]{kato})
\[
 P_n(\omega)=\frac1{2\pi i}\int_{\partial\Delta_n}
 (\lambda I-\Aop_\omega)^{-1}\,\dd\lambda
\]
is holomorphic.  It has finite rank because $\Aop_\omega$ is compact and
the contour avoids zero.  Its rank is locally constant, hence constant
on the connected parameter disk; at zero that rank is one.  Therefore
there is exactly one enclosed eigenvalue and it is algebraically
simple.  It is holomorphic, for example by
\[
 \Lambda_n(\omega)=\Tr(\Aop_\omega P_n(\omega)).
\]
The disk location gives \eqref{t2:eq:local-bound}; the spectral bridge
transfers all assertions about nonzero spectral values to $\Lop_\omega$.
\end{proof}

\subsubsection{Identification, not redefinition, of the local coefficients}

The golden-coordinate construction below continues the author's earlier
formulation \cite[v7, \S4.4 and \S8, Eq.~(41)]{alkauskas2012}.
For the base eigenfunction $e_n$ given in the Introduction,
\[
 e_n\in V,\qquad \Lop_0e_n=\lzero{n}e_n.
\]
Both facts follow by direct substitution; the pole $-\phi$ is outside
$\overline D_z$.  For $x$ near zero write
\[
 z(x)=\frac{\phi x+\phi^{-1}}{1-x},\qquad
 g_f(x)=(z(x)+\phi)^2f(z(x)),\qquad
 c_n(f)=[x^{n-1}]g_f(x).
\]
Each $c_n$ is continuous on $V$: a sufficiently small closed disk in
$x$ maps compactly into $D_z$, and Cauchy's estimate applies.  Also
$g_{e_n}(x)=x^{n-1}$ and $c_n(e_n)=1$.

Let $\Pi_n(\omega)$ be the local Riesz projection of $\Lop_\omega$
about $\lzero{n}$.  It has rank one by
Lemma~\ref{t2:lem:localization}.  On a possibly $n$-dependent
neighborhood of zero,
\[
 G_n(\omega)=
 \frac{\Pi_n(\omega)e_n}{c_n(\Pi_n(\omega)e_n)}
\]
is well defined and analytic, with $G_n(0)=e_n$ and
$c_n(G_n(\omega))=1$.  It is the unique normalized eigenfunction germ.
Its scalar eigenvalue must be the $\Lambda_n$ constructed in the lemma.
Consequently \eqref{eq:YW-series} is exactly the canonical
series of the manuscript.  In the original eigenfunction notation this
normalization says
\[
 q_0^{(j)}(n)=\delta_{jn},\qquad q_v^{(n)}(n)=0\quad(v\ge1).
\]
The normalized $q$--$Y$--$W$ recurrences and the justification of
composition at $\beta(0)\ne0$ are given in
Appendix~\ref{app:coefficients}.

Only the \emph{scalar} eigenvalue branches are asserted to have the
common disk.  Nonvanishing of the displayed eigenfunction normalization
functional is required only locally; no global assertion about that
functional is needed.

For the normalized eigenpair, the operator identity under translation
gives the General Equation
\[
\Lambda_n(\omega)g_n(\omega,x)
=\phi^{-2}g_n(\omega,\rho x)
+\Lambda_n(\omega)\omega\beta'(x)g_n(\omega,\beta(x)).
\]
Its analytic coefficient extraction and resonant normalization are
proved in Appendix~\ref{app:coefficients}.
In particular, with $r_{j,n}=\rho^{j-n}$, the canonical recurrence is
\[
(1-r_{j,n})q_v^{(j)}(n)
=\sum_{i\ge1}K(j,i)q_{v-1}^{(i)}(n)
-r_{j,n}\sum_{a=1}^vY_a(n)q_{v-a}^{(j)}(n),\qquad v\ge1.
\]
Its resonant equation and the conversion to the direct coefficients are
\[
Y_v(n)=\sum_{i\ge1}K(n,i)q_{v-1}^{(i)}(n),\qquad
W_v(n)=\sum_{a=1}^vY_a(n)W_{v-a}(n).
\]

\subsection{Common local trace decomposition}\label{sec:common}
\begin{corollary}[Uniform local coefficient bounds]
\label{t2:cor:coefficients}
For every $n\ge1$ and $\ell\ge1$,
\begin{equation}\label{t2:eq:W-bound}
 |W_\ell(n)|\le\tfrac12\,3^\ell.
\end{equation}
For every $k\ge1$ and $r\ge0$,
\begin{equation*}
 |W_r^{[k]}(n)|\le(3/2)^k3^r.
\end{equation*}
If $\ell_1,\ldots,\ell_s\ge1$ and $r=\sum_j\ell_j$, then
$\Omega_{\ell_1,\ldots,\ell_s}(w)$, as defined in Section~\ref{sec:results},
is holomorphic on $|w|<1$ and satisfies
\begin{equation*}
 |\Omega_{\ell_1,\ldots,\ell_s}(w)|
 \le2^{-s}3^r\frac{|w|}{1-|w|}.
\end{equation*}
\end{corollary}
\begin{proof}
By \eqref{t2:eq:local-bound},
$|\Wser_n-1|<1/2$ and $|\Wser_n|<3/2$ on $|\omega|<1/3$.
Apply Cauchy's estimate on $|\omega|=t<1/3$ to $\Wser_n-1$ and
$\Wser_n^k$, and let $t\uparrow1/3$.  The first two bounds follow.
The last follows by multiplying \eqref{t2:eq:W-bound} and summing a
geometric series in $|w|$.  All these bounds are uniform in $n$.
\end{proof}

\subsubsection{Proof of the graded trace decomposition}
\label{t2:sec:graded}

The coefficient-comparison derivation goes back to
\cite[v7, \S3.3, Eqs.~(14)--(15)]{alkauskas2012}; the uniform convergence
and spectral accounting are justified here as follows.
For fixed $k\ge1$, Lemma~\ref{t2:lem:localization} accounts for the
entire nonzero spectrum of $\Aop_\omega$ on $|\omega|<1/3$.
The Hilbert-space spectral-trace theorem and the shared bridge give
\begin{equation*}
 \Tr_V\Lop_\omega^k
 =\Tr_{\mathcal H}\Aop_\omega^k
 =\sum_{n\ge1}\Lambda_n(\omega)^k
 =\sum_{n\ge1}(\lzero{n})^k\Wser_n(\omega)^k.
\end{equation*}
The series is normally convergent on that complex disk, because its
terms are bounded by the summable majorant $(3/2)^k|\rho|^{nk}$.
No real-parameter restriction or subsequent analytic continuation is
necessary.

Cauchy's coefficient formula on any circle $|\omega|=t<1/3$ now gives
\begin{equation*}
 [\omega^r]\Tr_V\Lop_\omega^k
 =\sum_{n\ge1}(\lzero{n})^kW_r^{[k]}(n).
\end{equation*}
The interchange is justified by the preceding uniform majorant on the
circle.  Absolute convergence is also quantified directly by
\begin{equation*}
 \sum_{n\ge1}|\lzero{n}|^k|W_r^{[k]}(n)|
 \le(3/2)^k3^r\frac{|\rho|^k}{1-|\rho|^k}<\infty.
\end{equation*}
On the other hand, the independently proved word expansion
\eqref{t2:eq:word-trace} gives
\[
 [\omega^r]\Tr_V\Lop_\omega^k
 =\sum_{\substack{a_1+\cdots+a_k=k+r\\a_j\ge1}}
 \Cw(a_1,\ldots,a_k)=\Tword_{k,r}.
\]
Together these identities prove
\eqref{eq:graded-main}.

For $k=1,2,3$, expand the finite coefficient convolution and put
$L=k+r$.  Its factors are $W_{a_j-1}(n)$ with
$a_1+\cdots+a_k=L$.  The signs are
$(\lzero{n})^k=(-1)^{k(n+1)}\phi^{-2nk}$,
and the word weights are exactly \eqref{t2:eq:periodic-weight}.
This proves all three displayed examples in Theorem~2(ii), with their
claimed absolute convergence.  As a sign and index check, at excess zero
both sides reduce to
\[
 \sum_{n\ge1}(\lzero{n})^k=
 \frac1{\phi^{2k}-(-1)^k}
 =\Cw(\underbrace{1,\ldots,1}_{k}).
\]
This completes the proof of part~(ii), without any input from
Theorem~1.

\begin{remark}[The endpoint order of summation]
The degree-one column identities imply
\[
 \sum_{L\ge1}\left(\sum_{n\ge1}\lzero{n}W_{L-1}(n)\right)
 =\sum_{L\ge1}\Cw(L)=\Tr_V\Lop_1.
\]
Here the inner sums converge absolutely, and the outer equality follows
from the positive endpoint word formula. The absolute endpoint Fubini
corollary after Theorem~\ref{thm:all-index-endpoint} proves absolute
convergence of the unrestricted double series at $\omega=1$ and permits
exchange of its two summations. Endpoint row convergence for every $n$ is a
separate assertion, proved in Subsection~\ref{subsec:all-index-wiener}.
\end{remark}

\subsubsection{The exact first coefficient}

Differentiate $\Aop_\omega=(I-\omega M)^{-1}D_0$ at zero:
$\Aop'_0=MD_0$.  The simple eigenvalue $\lzero{n}$ of $D_0$ has the same unit
left and right eigenvector $u_n$.  Differentiating its local eigenvalue
equation and taking the $u_n$ coordinate gives
\[
 \Lambda_n'(0)=\langle MD_0u_n,u_n\rangle=\lzero{n}H_{nn}.
\]
Therefore
\begin{equation}\label{t2:eq:W1}
 W_1(n)=H_{nn}=K(n,n).
\end{equation}
No eigenfunction correction term is needed in this calculation.

To express the diagonal in Jacobi notation, put $m=n-1$.  Rescaling
$x=2\phi t$ in \eqref{t2:eq:kernel} gives
\[
 K(n,n)=\frac5{(2\phi)^{2n}}
 [t^m]\frac{(1+4t)^m}{(1-t)^{m+2}}.
\]
Since $1+4t=(1-t)+5t$, the last coefficient equals
\[
 \sum_{b=0}^m5^b\binom mb\binom{m+1}{b+1}.
\]
For the classical Jacobi normalization
\cite[Eq.~(4.3.2)]{szego}, equivalently \cite[Eq.~18.5.8]{DLMF},
\[
 P_m^{(0,1)}(x)=
 \sum_{a=0}^m\binom ma\binom{m+1}{a}
 \left(\frac{x-1}{2}\right)^a
 \left(\frac{x+1}{2}\right)^{m-a}.
\]
At $x=3/2$, multiplication by $4^m$ and the substitution $b=m-a$
identify the two sums.  Hence
\[
 K(n,n)=\frac5{(2\phi)^{2n}}4^{n-1}P_{n-1}^{(0,1)}(3/2)
 =\frac54\phi^{-2n}P_{n-1}^{(0,1)}(3/2),
\]
which is the exact formula asserted in Theorem~2(i).

\begin{statusremark}[The algebraic generating function for $W_1$]
The Jacobi generating function
\cite[Eq.~18.12.1]{DLMF}, specialized to $(\alpha,\beta)=(0,1)$, is
\begin{eqnarray}
 \sum_{m\ge0}P_m^{(0,1)}(3/2)z^m
 =\frac{2}{R(z)(1+z+R(z))},\qquad
 R(z)=\sqrt{1-3z+z^2},\quad R(0)=1.
 \label{jac-gen}
\end{eqnarray}
Multiplication by $5z/4$ and
$(1+z-R(z))(1+z+R(z))=5z$ therefore give
\[
 \Omega_1(w):=\sum_{n\ge1}W_1(n)w^n
 =\frac12\left(\frac{1+\phi^{-2}w}{\sqrt{1-3\phi^{-2}w+\phi^{-4}w^2}}-1\right),
 \qquad |w|<1.
\]
Equivalently, this is the diagonal of $xy\Theta_K(x,y)$, where $\Theta_K$
is the rational bivariate kernel generating function
\eqref{eq:scaled-bivar} of Appendix~\ref{app:exact-kernel}. This is an explicit
instance of algebraicity of diagonals of rational functions in two
variables; see \cite{furst,lip}.

The excess-one words are the $k$ cyclic rotations of $(1,\ldots,1,2)$,
all with the same weight. Since $W_1^{[k]}(n)=kW_1(n)$ and
$\lzero n=-\rho^n$, \eqref{eq:graded-main}, after division by $k$, gives
\[
 \Cw(1^{k-1},2)=(-1)^k\Omega_1(\rho^k)\qquad(k\ge1).
\]
Thus the entire excess-one layer of the graded trace formula is explicit.
\end{statusremark}

\subsubsection{Completion of part (i)}

The constant term is $W_0(n)=1$ by definition of the normalized germ.
For every $n\ge1$, Theorem~\ref{thm:all-index-endpoint} in
Subsection~\ref{subsec:global-canonical-spectrum} gives
\[
 (-1)^{n+1}\lambda_n
 =\phi^{-2n}\sum_{\ell\ge0}W_\ell(n),
\]
with absolute convergence, in the actual modulus enumeration.
Together with \eqref{t2:eq:W1} and its Jacobi evaluation, this proves
all of Theorem~2(i), and completes Theorem~2.

\section{Proof of the two-term asymptotic result}\label{sec:asymptotic}
We use the direct trace bridge of Proposition~\ref{t2:prop:bridge},
the canonical germs of Section~\ref{sec:germs}, and the endpoint
$2/3$-nuclear trace theory of Lemma~\ref{t2:lem:nuclear}.
Use the coefficient matrices $D,\Hker$ from Section~\ref{sec:weighted},
writing $\Hker=(\Kh(j,k))_{j,k\ge1}$.
Set $c=\sqrt5$ and $\cg=5^{1/4}/(2\sqrt\pi)$.

\subsection{Kernel estimates and target Wiener asymptotic}
The exact kernel $K(j,k)$ of \eqref{t2:eq:kernel} is positive by
Appendix~\ref{app:exact-kernel}, and its column generating function is
\begin{equation}\label{st1:eq:beta-input}
	\sum_{j\ge1}K(j,k)z^{j-1}=\beta'(z)\beta(z)^{k-1}.
\end{equation}
It also has weighted symmetry $kK(j,k)=jK(k,j)$.  Hence
$\widehat K(j,k):=\sqrt{\frac{k}{j}}K(j,k)$
is symmetric.  If $p=\sqrt{jk}$ and $\xi_{\mathrm K}=\sqrt j-\sqrt k$,
then Theorem~\ref{ka:thm:kernel} of Appendix~\ref{app:kernel-asymptotics}
gives, with absolute constants $C,\gamma_{\mathrm K}>0$,
\begin{align}
	0<\widehat K(j,k)&\le Cp^{-1/2}e^{-\gamma_{\mathrm K}\xi_{\mathrm K}^2},\label{st1:eq:weakK1}\\
	\left|\widehat K(j,k)-\cg p^{-1/2}e^{-c\xi_{\mathrm K}^2}\right|
	&\le Cp^{-3/2}e^{-\gamma_{\mathrm K}\xi_{\mathrm K}^2}.\label{st1:eq:weakK2}
\end{align}
Only \eqref{st1:eq:weakK1}--\eqref{st1:eq:weakK2}, not the quartic refinement of
the kernel, will be used below.

The strict upper chains are
\begin{equation}\label{ka:eq:chains}
	S_1(n)=K(n,n),\qquad
	S_\ell(n)=\sum_{i_1,\ldots,i_{\ell-1}>n}
	\prod_{r=0}^{\ell-1}K(i_r,i_{r+1}),
	\quad i_0=i_\ell=n,
\end{equation}
for $\ell\ge2$, and
\[
\Sser_n(\omega)=\sum_{\ell\ge1}S_\ell(n)\omega^\ell.
\]
All $S_\ell(n)$ are nonnegative.

The central quantitative statement proved below is stronger than what is
needed merely to evaluate at $\omega=1$.

\begin{proposition}[Coefficientwise complete-chain and eigenvalue control]
	\label{st1:prop:master}
	For all sufficiently large $n$,
	\begin{equation}\label{st1:eq:masterS}
		\left\|\Sser_n-
		\frac{\cg}{\sqrt n}\sum_{\ell\ge1}\frac{\omega^\ell}{\ell^{3/2}}
		\right\|_{\Walg}
		\le \frac{C}{n},
	\end{equation}
	and
	\begin{equation}\label{st1:eq:masterYW}
		\left\|\Yser_n-
		\frac{\cg}{\sqrt n}\sum_{\ell\ge1}\frac{\omega^\ell}{\ell^{3/2}}
		\right\|_{\Walg}
		+
		\left\|\Wser_n-1-
		\frac{\cg}{\sqrt n}\sum_{\ell\ge1}\frac{\omega^\ell}{\ell^{3/2}}
		\right\|_{\Walg}
		\le \frac{C}{n},
	\end{equation}
	where $\Walg$ is the scalar Wiener algebra
	\[
	\Walg=\left\{F(\omega)=\sum_{r\ge0}F_r\omega^r:
	\norm{F}_{\Walg}=\sum_{r\ge0}|F_r|<\infty\right\}.
	\]
	In particular,
	\begin{equation}\label{st1:eq:individualW}
		\sum_{\ell\ge1}
		\left|W_\ell(n)-\frac{\cg}{\sqrt n\,\ell^{3/2}}\right|
		\le \frac{C}{n},
	\end{equation}
	so the endpoint eigenvalue series is absolutely convergent for all
	sufficiently large $n$.
\end{proposition}

\subsection{All-index Wiener control}
\label{subsec:all-index-wiener}

We first establish Wiener control for every canonical branch label.
An analytic estimate treats $n\ge16$, and the exact rational certificate
treats $1\le n\le15$. The argument uses the canonical normalization from
Appendix~\ref{app:coefficients} in an essential way.

Put
\[
        d_n=\lambda_n^{(0)}=(-1)^{n+1}\phi^{-2n},
\]
and set
\[
 b_n=
 \begin{cases}
  17/20,&n=1,\\
  3/4,&n=2,\\
  13/20,&n=3,\\
  3/5,&n\ge4.
 \end{cases}
\]

\begin{theorem}[All-index Wiener control]
\label{thm:all-index-wiener}
For every \(n\ge1\),
\begin{equation}
        \sum_{\ell\ge1}|Y_\ell(n)|<b_n
        \label{eq:allindex-Y-bound}
\end{equation}
and consequently
\begin{equation}
        \sum_{\ell\ge0}|W_\ell(n)|
        <
        \frac1{1-b_n}
        =
        \begin{cases}
        20/3,&n=1,\\
        4,&n=2,\\
        20/7,&n=3,\\
        5/2,&n\ge4.
        \end{cases}
        \label{eq:allindex-W-bound}
\end{equation}
In particular, every canonical row series converges absolutely at \(\omega=1\).
\end{theorem}

\begin{statusremark}[A uniform arithmetic bound]
For later arithmetic use, Theorem~\ref{thm:all-index-wiener} also gives, for
$\ell_1,\ldots,\ell_q\ge1$,
\[
 |\Omega_{\ell_1,\ldots,\ell_q}(w)|
 \le \left(\frac{17}{3}\right)^q\frac{|w|}{1-|w|}
 \qquad(|w|<1),
\]
uniformly in the positive indices $\ell_j$. Indeed $W_0(n)=1$ and
\eqref{eq:allindex-W-bound} give
$\sum_{\ell\ge1}|W_\ell(n)|\le17/3$ for every $n$.
\end{statusremark}

We divide the proof into a coefficient majorant and a tridiagonal positivity
estimate, with an exact certificate for the finite range.

\subsubsection*{A positive majorant for the canonical recurrence}

Fix \(n\).  Use the symmetric scaled kernel $\Hker$ from
Section~\ref{sec:weighted}, and put $r_j=\frac{d_j}{d_n}=\rho^{\,j-n}$.
Rescale the canonical eigenfunction coefficients by
\[
        f_v(j)=\sqrt{\frac nj}\,q_v^{(j)}(n).
\]
Then the normalized recurrence of Appendix~\ref{app:coefficients} becomes
\begin{equation}
 (1-r_j)f_v(j)
 =
 \sum_{i\ge1}\Hker_{ji}f_{v-1}(i)
 -
 r_j\sum_{a=1}^{v}Y_a(n)f_{v-a}(j),
 \qquad v\ge1,
 \label{eq:allindex-scaled-recurrence}
\end{equation}
with
\[
        f_0=e_n,\qquad f_v(n)=0\quad(v\ge1),
\]
and the resonant equation is
\[
        Y_v(n)=\sum_{i\ge1}\Hker_{ni}f_{v-1}(i).
\]

For \(0<b<1\), define
\[
 a_n=b,\qquad
 a_j=|1-r_j|-b|r_j|\quad(j\ne n),
 \qquad
 \mathsf A=\operatorname{diag}(a_j),
\]
and
\[
        \mathsf C=\mathsf A-\Hker.
\]

\begin{lemma}[Majorant criterion]
\label{lem:allindex-majorant}
Assume that all \(a_j>0\), that \(\mathsf C\ge0\), and that
\[
 G_{jj}:=
 \lim_{\varepsilon\downarrow0}
 \left\langle
 e_j,(\mathsf C+\varepsilon I)^{-1}e_j
 \right\rangle
 <\infty
 \qquad(j\ge1).
\]
Then
\begin{equation}
        \sum_{v\ge1}|Y_v(n)|<b.
        \label{eq:allindex-majorant-Y}
\end{equation}
Moreover, for every fixed \(j\),
\begin{equation}
        \sum_{v\ge1}|f_v(j)|<\infty.
        \label{eq:allindex-majorant-f}
\end{equation}
\end{lemma}

\begin{proof}
For completeness we justify the Neumann expansion in operator norm.
For fixed $n$, the numbers $a_j$ tend to $1$ as $j\to\infty$, while only
finitely many indices lie below $n$; hence $a^*:=\sup_j a_j<\infty$.
By Lemma~\ref{cl:lem:Laguerre}, $\Hker\ge0$, and the hypothesis
$\mathsf C=\mathsf A-\Hker\ge0$ gives $\Hker\le\mathsf A$ in the form sense.
Therefore, for $\varepsilon>0$,
\[
\begin{gathered}
 B_\varepsilon:=(\mathsf A+\varepsilon I)^{-1/2}
 \Hker(\mathsf A+\varepsilon I)^{-1/2},\\
 0\le B_\varepsilon
 \le(\mathsf A+\varepsilon I)^{-1/2}\mathsf A
       (\mathsf A+\varepsilon I)^{-1/2}
 \le\frac{a^*}{a^*+\varepsilon}I.
\end{gathered}
\]
Thus $\|B_\varepsilon\|<1$, and
\[
 (\mathsf C+\varepsilon I)^{-1}
 =(\mathsf A+\varepsilon I)^{-1/2}
   \sum_{k\ge0}B_\varepsilon^k
   (\mathsf A+\varepsilon I)^{-1/2}
\]
converges in operator norm. Each term has nonnegative entries, since
$\Hker$ has nonnegative entries and $\mathsf A+\varepsilon I$ is positive
diagonal. Equivalently,
\[
 (\mathsf C+\varepsilon I)^{-1}
 =
 \sum_{k\ge0}
 \bigl((\mathsf A+\varepsilon I)^{-1}\Hker\bigr)^k
 (\mathsf A+\varepsilon I)^{-1}.
\]
As \(\varepsilon\downarrow0\), every entry increases.  The diagonal
hypothesis and Cauchy--Schwarz therefore give finite limits for all
entries.

Let
\[
 g_j=
 \lim_{\varepsilon\downarrow0}
 \left\langle
 e_j,(\mathsf C+\varepsilon I)^{-1}e_n
 \right\rangle,
 \qquad
 g_n\ge\lim_{\varepsilon\downarrow0}\frac1{b+\varepsilon}
 =\frac1b>0,
 \qquad
 F_j=\frac{g_j}{g_n}.
\]
Here the inequality is the degree-zero term of the regularized
nonnegative Neumann series.
Then \(F_j\ge0\), \(F_n=1\), and monotone convergence in the coordinate
resolvent equations gives
\begin{equation}
 \sum_iH_{ji}F_i=a_jF_j
 \qquad(j\ne n),
 \label{eq:allindex-F-offdiag}
\end{equation}
whereas at \(j=n\),
\begin{equation}
 \sum_iH_{ni}F_i
 =
 b-\frac1{g_n}
 =:m<b.
 \label{eq:allindex-m}
\end{equation}

Put
\[
 Q_M(j)=\sum_{v=1}^{M}|f_v(j)|,
 \qquad
 Z_M=\sum_{v=1}^{M}|Y_v(n)|.
\]
We prove simultaneously
\[
        Q_M(j)\le F_j\quad(j\ne n),
        \qquad
        Z_M\le m.
\]
Assume the assertion through order \(M-1\).  From the resonant equation,
\[
\begin{aligned}
 Z_M
 &\le
 \sum_iH_{ni}
 \left(
 |f_0(i)|+Q_{M-1}(i)
 \right)\\
 &\le
 \sum_iH_{ni}F_i
 =m.
\end{aligned}
\]
For \(j\ne n\), summing \eqref{eq:allindex-scaled-recurrence} through
order \(M\) gives
\[
\begin{aligned}
 |1-r_j|Q_M(j)
 &\le
 \sum_iH_{ji}F_i
 +
 |r_j|
 \sum_{\substack{a,b\ge1\\a+b\le M}}
 |Y_a(n)|\,|f_b(j)|\\
 &\le
 a_jF_j+b|r_j|F_j
 =
 |1-r_j|F_j.
\end{aligned}
\]
Here the restriction \(b\ge1\) in the convolution is precisely where the
canonical normalization is used: the possible term \(f_0(j)\) vanishes for
\(j\ne n\).  Thus \(Q_M(j)\le F_j\), completing the induction.

Letting \(M\to\infty\) proves
\eqref{eq:allindex-majorant-Y}--\eqref{eq:allindex-majorant-f}.
\end{proof}

Since \(\Wser_n(\omega)=(1-\Yser_n(\omega))^{-1}\),
Lemma~\ref{lem:allindex-majorant} immediately gives
\[
        \Wser_n\in \Walg,
        \qquad
        \|\Wser_n\|_{\Walg}
        \le\frac1{1-\|\Yser_n\|_{\Walg}}
        <\frac1{1-b}.
\]
Thus it remains only to verify the hypotheses of the lemma.

\subsubsection*{Reduction to a discrete Schr\"odinger form}

By the exact Laguerre realization of Lemma~\ref{cl:lem:Laguerre},
\(\Hker\) is multiplication by \(e^{-x/c}\) in the orthonormal basis
\[
        p_j(x)=j^{-1/2}L_{j-1}^{(1)}(x).
\]
Let \(\mathcal J\) denote multiplication by \(x\) in this basis.  The
standard Laguerre three-term recurrence \cite[\S18.9(i)]{DLMF} gives
\[
 \mathcal J_{jj}=2j,\qquad
 \mathcal J_{j,j+1}=\mathcal J_{j+1,j}
 =-\sqrt{j(j+1)}.
\]
Consequently $\Hker=e^{-\mathcal J/c}$.
The scalar inequality
\[
        e^{-x/c}\le\frac{c}{x+c},
        \qquad x\ge0,
\]
gives, by functional calculus,
\begin{equation}
        \Hker\le c(\mathcal J+cI)^{-1}.
        \label{eq:allindex-H-resolvent-majorant}
\end{equation}

Define
\[
        V_j=c\left(1-\frac1{a_j}\right)
\]
and the half-line Jacobi form
\[
 \mathsf S
 =
 \Delta+
 \operatorname{diag}\left(\frac{V_j}{j}\right)_{j\ge1},
\]
where
\[
 \langle\Delta u,u\rangle
 =
 \sum_{j\ge0}|u_{j+1}-u_j|^2,
 \qquad u_0=0.
\]

\begin{lemma}[Tridiagonal sufficient condition]
\label{lem:allindex-tridiagonal}
If all \(a_j>0\), \(\mathsf S\ge0\), and its zero-energy Green diagonals
are finite, then
\(\mathsf C=\mathsf A-\Hker\) satisfies all hypotheses of
Lemma~\ref{lem:allindex-majorant}.
\end{lemma}

\begin{proof}
For the choices \(b=b_n\) used below, all \(a_j>0\).  Indeed,
\(a_n=b_n\ge3/5\); for odd offsets \(j-n\), one has \(a_j>1\).
For positive even offsets,
\[
 a_j\ge1-\frac{37}{20}\frac3{20}=\frac{289}{400}>\frac35,
\]
and for negative even offsets (which require \(n\ge3\)),
\[
 a_j\ge\frac7{20}\frac{20}{3}-1=\frac43.
\]
For fixed \(n\), the diagonal operator \(\mathsf A\) is therefore bounded
and boundedly invertible.

For a nonnegative selfadjoint operator or closed form \(T\), write
\[
 \mathcal G_T(j):=\lim_{\varepsilon\downarrow0}
 \langle e_j,(T+\varepsilon I)^{-1}e_j\rangle\in[0,\infty]
\]
for its zero-energy Green diagonal.  Equivalently, when \(c_{00}\) is a
form core,
\[
 \mathcal G_T(j)=\sup_{0\ne u\in c_{00}}\frac{|u_j|^2}{q_T[u]},
\]
with the usual extended-value convention for a vanishing denominator.
By \eqref{eq:allindex-H-resolvent-majorant},
\[
 \mathsf C
 \ge
 \mathsf C_0
 :=
 \mathsf A-c(\mathcal J+cI)^{-1}.
\]
The two Schur complements of
\[
 \begin{pmatrix}
 \mathsf A&-\sqrt c\,I\\
 -\sqrt c\,I&\mathcal J+cI
 \end{pmatrix}
\]
are \(\mathsf C_0\) and
\[
 \mathsf R
 =
 \mathcal J+cI-c\mathsf A^{-1}
 =
 \mathcal J+\operatorname{diag}(V_j).
\]
On finitely supported sequences, the substitution \(f_j=u_j/\sqrt j\)
gives exactly \(q_{\mathsf R}[f]=q_{\mathsf S}[u]\).
In the Laguerre model, the form norm for multiplication by $x$ is
the $L^2$-norm with weight $x(1+x)e^{-x}\dd x$.  Polynomials are dense
in this norm: if $h$ is orthogonal to them, then

\[
 \mathcal F(z)=\int_0^\infty h(x)x(1+x)e^{-x}e^{zx}\dd x,
 \qquad \Re z<\tfrac12,
\]
is holomorphic by Cauchy--Schwarz and all its derivatives vanish at zero.
The identity theorem gives \(\mathcal F(it)=0\) for every real $t$, and
uniqueness of the Fourier transform gives $h=0$.  Hence $c_{00}$ is a
form core.  Bounded diagonal perturbations have the same form domain and
preserve the core (after adding a constant, the form norms are equivalent).
Thus positivity passes to the closed forms; no bounded similarity on
\(\ell^2\) is asserted.
Hence \(\mathsf S\ge0\) implies \(\mathsf R\ge0\), and therefore
\(\mathsf C_0\ge0\).

It remains to transfer the Green bound.  Put
\[
 \mathsf A_\varepsilon=\mathsf A+\varepsilon I,\qquad
 \mathsf R_\varepsilon
 =
 \mathcal J+cI-c\mathsf A_\varepsilon^{-1}.
\]
If \(a^*=\sup_j a_j<\infty\), then
\[
 \mathsf R_\varepsilon-\mathsf R
 =c\bigl(\mathsf A^{-1}-(\mathsf A+\varepsilon I)^{-1}\bigr)
 \ge\frac{c\varepsilon}{a^*(a^*+\varepsilon)}I.
\]
Thus \(\mathsf R_\varepsilon\) is strictly positive for every
\(\varepsilon>0\), and the regularized Schur identity is legitimate:
\[
 (\mathsf C_0+\varepsilon I)^{-1}
 =
 \mathsf A_\varepsilon^{-1}
 +
 c\mathsf A_\varepsilon^{-1}
 \mathsf R_\varepsilon^{-1}
 \mathsf A_\varepsilon^{-1}.
\]
Thus
\[
 \left[
 (\mathsf C_0+\varepsilon I)^{-1}
 \right]_{jj}
 \le
 \frac1{a_j+\varepsilon}
 +
 \frac{c}{(a_j+\varepsilon)^2}
 \left[
 \mathsf R_\varepsilon^{-1}
 \right]_{jj}.
\]
The variational characterization on \(c_{00}\), with
\(f_j=u_j/\sqrt j\), gives
\[
 \mathcal G_{\mathsf R}(j)=\frac1j\,\mathcal G_{\mathsf S}(j).
\]
Since \([\mathsf R_\varepsilon^{-1}]_{jj}\le\mathcal G_{\mathsf R}(j)\),
the assumed Green bound for \(\mathsf S\) makes the preceding resolvent
bound uniformly finite as \(\varepsilon\downarrow0\).
Finally \(\mathsf C\ge\mathsf C_0\), so inverse order transfers the same
bound to \(\mathsf C\).
\end{proof}

We now verify Lemma~\ref{lem:allindex-tridiagonal} with \(b=b_n\).

\subsubsection*{The uniform estimate for \(n\ge16\)}

Assume \(n\ge16\) and \(b=3/5\).  For \(j<n\), direct use of
\(r_j=\rho^{j-n}\) gives
\[
        V_j\ge\frac9{10}.
\]
At the central coordinate,
\[
        V_n=-\frac{2c}{3}.
\]
For \(j>n\), all odd offsets have \(V_j\ge0\). Writing $t=\phi^{-4}$,
the negative even offsets are
\[
 -V_{n+2m}
 =
 c\,
 \frac{(8/5)t^m}{1-(8/5)t^m},
 \qquad m\ge1.
\]

Let
\[
 B=-V_n+\sum_{m\ge1}(-V_{n+2m}),
 \qquad
 M=\sum_{m\ge1}2m(-V_{n+2m}).
\]
Using
\[
        c<\frac94,\qquad t<\frac3{20},
\]
one obtains
\[
        B
        <
        \frac32+\frac{270}{323}
        <
        \frac{12}{5},
        \qquad
        M
        <
        \frac{10800}{5491}
        <2.
\]

Compare \(\mathsf S\) with
\[
 \mathsf S_+
 =
 \Delta+
 \operatorname{diag}
 \left(
 \frac9{10j}\mathbf 1_{\{j<n\}}
 \right).
\]
Let \(G_+\) be its zero-energy Green kernel and put
\[
        g_n=G_+(n,n).
\]
The zero-energy solution on the left may be described by
\[
        h_j=\frac{w_j-w_{j-1}}{w_j}.
\]
Since \(9/(10j)\ge9/(10n)\) for \(j<n\), comparison with the constant
potential \(9/(10n)\) gives
\begin{equation}
 g_n
 \le
 \frac12+\sqrt{\frac{10n}{9}+\frac14}
 \le
 1+\frac{10}{9}\sqrt n.
 \label{eq:allindex-gn}
\end{equation}
On the free right tail,
\begin{equation}
        G_+(n+k,n+k)=g_n+k.
        \label{eq:allindex-Gplus-tail}
\end{equation}

Let
\[
        Q_-=
        \operatorname{diag}
        \left(
        \frac{(-V_j)_+}{j}
        \right).
\]
On every finite Dirichlet restriction, positivity gives
\[
\begin{aligned}
 \operatorname{Tr}
 \bigl(
 Q_-^{1/2}G_+Q_-^{1/2}
 \bigr)
 &\le
 \frac{Bg_n+M}{n}\\
 &<
 \frac{8}{3\sqrt n}+\frac{22}{5n}\\
 &\le
 \frac{113}{120}
 <1.
\end{aligned}
\]
Hence the Birman--Schwinger operator has norm \(<113/120\), and therefore
\[
        Q_-\le\frac{113}{120}\,\mathsf S_+.
\]
It follows that
\begin{equation}
        \mathsf S
        \ge
        \mathsf S_+-Q_-
        \ge
        \frac7{120}\,\mathsf S_+.
        \label{eq:allindex-S-lower}
\end{equation}
The Green diagonals of \(\mathsf S_+\) are finite by
\eqref{eq:allindex-gn}--\eqref{eq:allindex-Gplus-tail}, so
\eqref{eq:allindex-S-lower}, first on finite restrictions and then by
monotone convergence, also gives finite Green diagonals for \(\mathsf S\).

Thus Lemma~\ref{lem:allindex-tridiagonal} applies for every \(n\ge16\).

\subsubsection*{The finite exact certificate for \(1\le n\le15\)}

The remaining cases are finite and are verified by exact rational
arithmetic.  We record the certificate sufficiently explicitly to make
the computation reproducible.

Use the rational enclosures
\begin{equation}
 \frac{381}{1000}<\phi^{-2}<\frac{191}{500},
 \qquad
 \frac{11}{5}<c<\frac94.
 \label{eq:allindex-rational-enclosures}
\end{equation}
For \(b=b_n\) and \(k=j-n\), define the rational lower bound
\[
 \underline a_j=
 \begin{cases}
 b,&k=0,\\
 1+(1-b)(381/1000)^k,&k>0\text{ odd},\\
 1-(1+b)(191/500)^k,&k>0\text{ even},\\
 1+(1-b)(191/500)^k,&k<0\text{ odd},\\
 (1-b)(191/500)^k-1,&k<0\text{ even}.
 \end{cases}
\]
Here negative powers have their usual rational meaning.  In every case
\(\underline a_j\le a_j\).

Put
\[
 z_j=1-\underline a_j^{-1}
\]
and define
\[
 \underline V_j
 =
 10^{-9}
 \left\lfloor
 10^9 c_j z_j
 \right\rfloor,
 \qquad
 c_j=
 \begin{cases}
 11/5,&z_j\ge0,\\
 9/4,&z_j<0.
 \end{cases}
\]
Then
\[
        \underline V_j\le V_j.
\]

Set \(L=n+20\).  Starting with \(\underline h_1=1\), iterate for
\(1\le j<L\)
\[
 \underline h_{j+1}
 =
 10^{-12}
 \left\lfloor
 10^{12}
 \frac{\underline h_j+\underline V_j/j}
      {1+\underline h_j+\underline V_j/j}
 \right\rfloor.
\]
All operations are rational.  The exact certificate checks at every step that
\[
        \underline h_j+\frac{\underline V_j}{j}>0.
\]
The terminal values satisfy
\[
\begin{array}{c|c|c}
 n&b_n&\text{certified lower bound for }\underline h_{n+20}\\
 \hline
 1&17/20&7/1000\\
 2&3/4&11/1000\\
 3&13/20&3/1000\\
 4\le n\le15&3/5&11/1000
\end{array}
\]
and in particular
\[
        \underline h_L>\frac3{1000}.
\]

For the infinite tail one has, for every \(j\ge L\),
\[
        \frac{V_j}{j}
        \ge
        -\alpha_j,
        \qquad
        \alpha_j=
        \frac{9(191/500)^{\,j-n}}{j}.
\]
Indeed the negative tail terms occur only at even offsets, and
\(2(191/500)^{20}<1/2\).

Define
\[
 A_{\rm tail}=\sum_{j\ge L}\alpha_j,
 \qquad
 B_{\rm tail}=\sum_{j\ge L}(j-L)\alpha_j.
\]
Then
\[
 A_{\rm tail}
 \le
 \frac{9(191/500)^{20}}
      {L(1-191/500)},
\qquad
 B_{\rm tail}
 \le
 \frac{9(191/500)^{21}}
      {L(1-191/500)^2},
\]
and exact rational arithmetic gives, uniformly for all fifteen cases,
\begin{equation}
 \frac{A_{\rm tail}}{3/1000}
 +B_{\rm tail}
 <10^{-5}.
 \label{eq:allindex-tail-margin}
\end{equation}

Define the comparison potential and its exact solution by
\[
 \widetilde v_j=
 \begin{cases}
  \underline V_j/j,&j<L,\\
  -\alpha_j,&j\ge L,
 \end{cases}
 \qquad
 w_0=0,\quad w_1=1,\quad
 w_{j+1}=(2+\widetilde v_j)w_j-w_{j-1}.
\]
For \(j<L\), the relative slope \(h_j=(w_j-w_{j-1})/w_j\) satisfies
\[
 h_{j+1}=\frac{h_j+\widetilde v_j}{1+h_j+\widetilde v_j}.
\]
Starting from \(h_1=\underline h_1=1\), the certificate's positivity
checks and downward rounding, together with monotonicity of
\(x\mapsto x/(1+x)\) for \(x>0\), give inductively positive slopes and
\(h_j\ge\underline h_j\) through \(j=L\).  In particular,
\(h_L>3/1000\), the lower bound used in the tail estimate.

On the tail the comparison equation is
\[
 -w_{j-1}+(2-\alpha_j)w_j-w_{j+1}=0
\]
and put
\[
        \delta_j=w_j-w_{j-1}.
\]
As long as the slopes remain positive,
\[
        w_j\le w_L+(j-L)\delta_L,
\]
so their total possible decrease is at most
\[
 w_LA_{\rm tail}+\delta_LB_{\rm tail}
 =
 \delta_L
 \left(
 \frac{A_{\rm tail}}{h_L}+B_{\rm tail}
 \right)
 <\delta_L.
\]
A first-failure argument therefore shows that all slopes remain positive,
indeed bounded below by a positive constant.  Hence \(w_j\) grows at least
linearly.

Writing \(\widetilde{\mathsf S}=\Delta+\operatorname{diag}(\widetilde v_j)\),
the ground-state identity (for general context see \cite{keller2018})
\[
 q_{\widetilde{\mathsf S}}[u]
 =\sum_{j\ge1}w_jw_{j+1}
 \left|\frac{u_{j+1}}{w_{j+1}}-\frac{u_j}{w_j}\right|^2,
 \qquad u\in c_{00},
\]
gives nonnegativity of the comparison Jacobi form.  On the finite
Dirichlet interval \(\{1,\ldots,N\}\), its zero-energy Green diagonal is
\[
 G_N(j,j)=w_j^2\sum_{r=j}^{N}\frac1{w_rw_{r+1}}.
\]
Letting \(N\to\infty\) gives
\[
        G(j,j)
        =
        w_j^2
        \sum_{r=j}^{\infty}
        \frac1{w_rw_{r+1}}
        <\infty.
\]
The sum is finite because \(w_j\) grows at least linearly.
Since the exact potential is everywhere no smaller than the comparison
potential, the same conclusions hold for \(\mathsf S\).

The calculation
\eqref{eq:allindex-rational-enclosures}--\eqref{eq:allindex-tail-margin}
is implemented in the ancillary file
\texttt{problem\_A\_certificate.py}; its exact executed output is included
as \texttt{problem\_A\_certificate\_output.txt}.  The verification uses
only exact rational arithmetic; no floating-point decision enters the
certificate.

Combining the finite certificate with the preceding \(n\ge16\) estimate
proves the hypotheses of Lemma~\ref{lem:allindex-majorant} for every
\(n\), with \(b=b_n\).  Therefore
\eqref{eq:allindex-Y-bound}--\eqref{eq:allindex-W-bound} are proved.

\subsection{Global spectral realization and all-index ordering}
\label{subsec:global-canonical-spectrum}

The preceding argument gives, for every \(n\ge1\),
\begin{equation}\label{eq:streamline-G-1}
 \|\Yser_n\|_{\Walg}<b_n<1,\qquad
 \|\Wser_n\|_{\Walg}<\frac1{1-b_n},
 \qquad
 \Lambda_n(\omega):=\lambda_n^{(0)}\Wser_n(\omega).
\end{equation}
In particular
\begin{equation}\label{eq:streamline-G-2}
 \sup_{n\ge1}\|\Wser_n\|_{\Walg}\le \frac{20}{3}.
\end{equation}
Near \(\omega=0\), Theorem~\ref{thm:canonical-local-realization} identifies \(\Lambda_n\) with the genuine
normalized local eigenvalue germ.  We now identify these scalar functions
globally, without constructing a second family of eigenvectors.

\begin{proposition}[Global canonical Fredholm factorization]
\label{prop:global-canonical-factorization}
For every \(|\omega|\le1\),
\begin{equation}\label{eq:streamline-G-3}
 \det\nolimits_V(I-z\Lop_\omega)
 =
 \prod_{n\ge1}\bigl(1-z\Lambda_n(\omega)\bigr),
 \qquad z\in\mathbb C.
\end{equation}
The product converges locally uniformly in \(z\).  Consequently the nonzero
spectrum of \(\Lop_\omega\), with algebraic multiplicity, is exactly the
multiset
\begin{equation}\label{eq:streamline-G-4}
 \{\Lambda_n(\omega):n\ge1\}.
\end{equation}
\end{proposition}

\begin{proof}
Fix \(k\ge1\).  By \eqref{eq:streamline-G-2},
\begin{equation}\label{eq:streamline-G-5}
 \sum_{n\ge1}|\Lambda_n(\omega)|^k
 \le
 \left(\frac{20}{3}\right)^k
 \sum_{n\ge1}\phi^{-2nk}
 =
 \left(\frac{20}{3}\right)^k
 \frac{\phi^{-2k}}{1-\phi^{-2k}},
 \qquad |\omega|\le1,
\end{equation}
Hence
\[
 F_k(\omega):=\sum_{n\ge1}\Lambda_n(\omega)^k
\]
is holomorphic for \(|\omega|<1\) and continuous for \(|\omega|\le1\).

On the common local disk of Subsection~\ref{t2:sec:graded}, the spectral-trace identity already
proved there gives
\begin{equation}\label{eq:streamline-G-6}
 F_k(\omega)=\operatorname{Tr}_V(\Lop_\omega^k).
\end{equation}
This local identity uses only the simple base spectrum, Lemma~\ref{t2:lem:localization} and the
spectral-trace theorem; it is independent of the endpoint argument here.
By Lemma~\ref{t2:lem:nuclear}, \(\Lop_\omega\) is holomorphic in nuclear norm for \(|\omega|<1\),
so \(\operatorname{Tr}_V(\Lop_\omega^k)\) is holomorphic there.  The identity
theorem therefore extends \eqref{eq:streamline-G-6} throughout the open unit disk.
Continuity in nuclear norm, together with \eqref{eq:streamline-G-5}, extends it to
\(|\omega|=1\).

For fixed \(|\omega|\le1\), \eqref{eq:streamline-G-5} with \(k=1\) shows that
\[
 P_\omega(z):=
 \prod_{n\ge1}\bigl(1-z\Lambda_n(\omega)\bigr)
\]
is an entire function.  For sufficiently small \(z\),
\begin{equation}\label{eq:streamline-G-7}
\begin{aligned}
 \log P_\omega(z)
 &=
 -\sum_{k\ge1}\frac{z^k}{k}
   \sum_{n\ge1}\Lambda_n(\omega)^k\\
 &=
 -\sum_{k\ge1}\frac{z^k}{k}
   \operatorname{Tr}_V(\Lop_\omega^k).
\end{aligned}
\end{equation}
By the Fredholm determinant convention of Subsection~\ref{sec:weighted}, the last expression
is \(\log\det_V(I-z\Lop_\omega)\) near the origin.  The two entire functions
therefore agree everywhere, proving \eqref{eq:streamline-G-3}.  Their zeros, with
multiplicity, give \eqref{eq:streamline-G-4}.
\end{proof}

The all-index Wiener bounds also separate the canonical values.  Since
\[
 \Wser_n(\omega)=\frac1{1-\Yser_n(\omega)},\qquad |\Yser_n(\omega)|<b_n<1,
\]
we have
\begin{equation}\label{eq:streamline-G-8}
 \frac{\phi^{-2n}}{1+b_n}
 <
 |\Lambda_n(\omega)|
 <
 \frac{\phi^{-2n}}{1-b_n}.
\end{equation}
Moreover
\begin{equation}\label{eq:streamline-G-9}
 \Re \Wser_n(\omega)>0,
\end{equation}
because \(\Re(1-\Yser_n(\omega))>0\).  Hence odd and even labels lie in opposite
open half-planes.

For labels of the same parity,
\begin{equation}\label{eq:streamline-G-10}
 \frac{|\Lambda_{n+2}(\omega)|}{|\Lambda_n(\omega)|}
 <
 \phi^{-4}\frac{1+b_n}{1-b_{n+2}}
 \le
 \frac3{20}\frac{37/20}{7/20}
 =
 \frac{111}{140}
 <1.
\end{equation}
Here we used \(\phi^{-4}<3/20\), \(b_n\le17/20\), and
\(b_{n+2}\le13/20\).  Therefore all the values in \eqref{eq:streamline-G-4} are distinct.
By Proposition~\ref{prop:global-canonical-factorization}, every nonzero
eigenvalue of \(\Lop_\omega\) is consequently algebraically simple.

For real \(0\le t<1\), all coefficients are real, so
\begin{equation}\label{eq:streamline-G-11}
 (-1)^{n+1}\Lambda_n(t)>0.
\end{equation}
Together with \eqref{eq:streamline-G-10}, this already identifies
\(\Lambda_{2k-1}(t)\) as the \(k\)-th positive eigenvalue and
\(-\Lambda_{2k}(t)\) as the \(k\)-th negative eigenvalue in absolute value.

We next record the only selfadjoint estimate needed below.  Put
\begin{equation}\label{eq:streamline-G-12}
 R_t=I-t\Hker,\qquad
 T_t=R_t^{-1/2}DR_t^{-1/2},
 \qquad 0\le t<1.
\end{equation}
By Lemma~\ref{cl:lem:Laguerre}, \(0<R_t\le I\), so \(T_t\) is compact selfadjoint and is similar to the
Bergman model \((I-t\Hker)^{-1}D\). By Proposition~\ref{t2:prop:bridge},
these models have the same nonzero spectrum as \(\Lop_t\), with algebraic
multiplicity.  Let \(\mu_k^+(T_t)\) denote its \(k\)-th
positive eigenvalue and let \(\mu_k^-(T_t)\) denote the \(k\)-th largest
absolute value among its negative eigenvalues.

\begin{lemma}[Signed min--max lower bounds]
\label{lem:interior-minmax}
For \(0\le t<1\),
\begin{equation}\label{eq:streamline-G-13}
 \mu_k^+(T_t)\ge \phi^{-4k+2},
 \qquad
 \mu_k^-(T_t)\ge \phi^{-4k}.
\end{equation}
Consequently, for every \(n\ge1\),
\begin{equation}\label{eq:streamline-G-14}
 |\Lambda_n(1)|\ge \phi^{-2n}.
\end{equation}
\end{lemma}

\begin{proof}
Let \(V_k^+\) be the span of the first \(k\) positive coordinate
eigenvectors of \(D\).  For \(v\in V_k^+\) and \(u=R_t^{1/2}v\),
\[
 \langle T_tu,u\rangle
 =
 \langle Dv,v\rangle
 \ge
 \phi^{-4k+2}\|v\|^2
 \ge
 \phi^{-4k+2}\|u\|^2.
\]
The min--max principle gives the positive inequality.  Apply the same
argument to \(-D\) on the first \(k\) negative coordinate eigenvectors for
the second inequality.

By \eqref{eq:streamline-G-10}--\eqref{eq:streamline-G-11} and the global spectral factorization,
\[
 \mu_k^+(T_t)=\Lambda_{2k-1}(t),
 \qquad
 \mu_k^-(T_t)=-\Lambda_{2k}(t).
\]
The Wiener series are continuous at \(t=1\).  Letting \(t\uparrow1\) gives
\eqref{eq:streamline-G-14}.
\end{proof}

We can now identify the modulus enumeration at the endpoint.  For \(n\ge3\),
\(b_{n+1}=3/5\), so \eqref{eq:streamline-G-8} and \eqref{eq:streamline-G-14} give
\begin{equation}\label{eq:streamline-G-15}
 |\Lambda_{n+1}(1)|
 <
 \frac52\phi^{-2(n+1)}
 <
 \phi^{-2n}
 \le
 |\Lambda_n(1)|,
\end{equation}
because \(\phi^{-2}<2/5\).

The invariant density of Section~\ref{sec:intro} gives the eigenvalue \(1\), while the spectral radius
of \(\Lop_1\) is one.  Proposition~\ref{prop:global-canonical-factorization}
therefore gives \(1=\Lambda_m(1)\) for some odd \(m\).  If \(m>1\),
\eqref{eq:streamline-G-10} would give \(|\Lambda_1(1)|>1\), contradicting the spectral
radius bound.  Hence
\begin{equation}\label{eq:streamline-G-16}
 \Lambda_1(1)=1.
\end{equation}
Also
\[
 |\Lambda_2(1)|<4\phi^{-4}<1.
\]

It remains only to compare the labels \(2\) and \(3\).  For
\(0\le t<1\), take \(u_t=R_t^{1/2}e_2\).
Then
\begin{equation}\label{eq:streamline-G-17}
 \frac{\langle T_tu_t,u_t\rangle}{\|u_t\|^2}
 =
 -\frac{\phi^{-4}}{1-t\Hker_{22}}.
\end{equation}
Since \(\mu_1^-(T_t)=|\Lambda_2(t)|\),
\begin{equation}\label{eq:streamline-G-18}
 |\Lambda_2(1)|
 \ge
 \frac{\phi^{-4}}{1-\Hker_{22}}.
\end{equation}
The exact diagonal value is
\[
 \Hker_{22}=\frac{35}{16\phi^4}>\frac14,
\]
and therefore
\begin{equation}\label{eq:streamline-G-19}
 |\Lambda_2(1)|>\frac43\phi^{-4}.
\end{equation}
On the other hand,
\begin{equation}\label{eq:streamline-G-20}
 |\Lambda_3(1)|
 <
 \frac{20}{7}\phi^{-6}
 <
 \frac43\phi^{-4}.
\end{equation}
Combining \eqref{eq:streamline-G-15}--\eqref{eq:streamline-G-20} gives
\begin{equation}\label{eq:streamline-G-21}
 |\Lambda_1(1)|
 >
 |\Lambda_2(1)|
 >
 |\Lambda_3(1)|
 >
 \cdots .
\end{equation}

\begin{theorem}[All-index endpoint spectral identification]\label{thm:all-index-endpoint}
The canonical label agrees with the modulus enumeration of the actual
GKW spectrum:
\begin{equation}\label{eq:allindex-eigenvalue-series}
 \boxed{\lambda_n=\Lambda_n(1)
 =\lambda_n^{(0)}\Wser_n(1)
 =\lambda_n^{(0)}
   \sum_{\ell\ge0}W_\ell(n)}
 \qquad(n\ge1).
\end{equation}
The series converges absolutely for every \(n\ge1\). Every nonzero eigenvalue
is algebraically simple, and the signs alternate and the moduli strictly decrease
for every index.
\end{theorem}
\begin{proof}
The factorization, separation and endpoint comparisons above prove all the
spectral assertions. Absolute convergence follows from
Theorem~\ref{thm:all-index-wiener}.
\end{proof}

\begin{fubinicorollary}[Absolute endpoint Fubini property]
For every $k\ge1$,
\[
 \sum_{n\ge1}\sum_{r\ge0}|\lzero n|^k|W_r^{[k]}(n)|
 \le\left(\frac{20}{3}\right)^k\frac{\phi^{-2k}}{1-\phi^{-2k}}<\infty.
\]
Consequently, for $|\omega|\le1$,
\[
 \Tr_V\Lop_\omega^k
 =\sum_{n\ge1}\sum_{r\ge0}(\lzero n)^kW_r^{[k]}(n)\omega^r
 =\sum_{r\ge0}\Tword_{k,r}\omega^r,
\]
and the two sums may be interchanged absolutely. In particular,
\[
 \sum_{n,L\ge1}|\lzero n W_{L-1}(n)|<\infty.
\]
Hence the matrix $(\lzero n W_{L-1}(n))_{n,L\ge1}$ has row sums
$\lambda_n$, column sums $\Cw(L)$, and total sum $\Tr\Lop$,
independently of the order of summation.
\end{fubinicorollary}
\begin{proof}
By \eqref{eq:streamline-G-2} and submultiplicativity of the Wiener norm,
\[
 \sum_{r\ge0}|W_r^{[k]}(n)|
 =\|\Wser_n^k\|_{\Walg}
 \le\|\Wser_n\|_{\Walg}^k\le(20/3)^k.
\]
Since $|\lzero n|=\phi^{-2n}$, summation over $n$ gives the bound.
Proposition~\ref{prop:global-canonical-factorization},
\eqref{eq:graded-main}, and the endpoint word formula
\eqref{t2:eq:word-trace} give the trace identities by Fubini.
Theorem~\ref{thm:all-index-endpoint} identifies the row sums.
\end{proof}

\subsection{Square-root mesh and a single Green barrier}

Put
\begin{equation}\label{st1:eq:mesh}
	s_j=\sqrt j,\qquad \mu_j=\frac1{2s_j},\qquad
	\Kn(j,k)=\frac{\widehat K(j,k)}{\sqrt{\mu_j\mu_k}},
	\qquad
	g(x)=\sqrt{\frac c\pi}e^{-cx^2}=2\cg e^{-cx^2}.
\end{equation}
Then \eqref{st1:eq:weakK1}--\eqref{st1:eq:weakK2} become
\begin{align}
	0\le\Kn(j,k)&\le Ce^{-\gamma_{\mathrm K}(s_j-s_k)^2},\label{st1:eq:rootbound}\\
	\abs{\Kn(j,k)-g(s_j-s_k)}
	&\le\frac{C}{s_js_k}e^{-\gamma_{\mathrm K}(s_j-s_k)^2}.\label{st1:eq:rooterror}
\end{align}
The mesh measure $\sum_j\mu_j\delta_{s_j}$ has uniformly bounded mass in
every interval of fixed length; indeed $\mu_j\le s_j-s_{j-1}$.

Define
\begin{equation*}
	P(j,k)=\mu_k\Kn(j,k),
	\qquad P_n=P|_{\{n+1,n+2,\ldots\}}.
\end{equation*}
The kernel $P_n$ is symmetric on $\ell^2(\mu)$.
The same kernel will also be used on bounded functions.  Once the
sub-Markov estimate $P1\le1$ is established below, $P_n$ is a positive
contraction on $\ell^\infty$.  By weighted symmetry,
$\mu_iP(i,k)=\mu_kP(k,i)$, it is also a contraction on $\ell^1(\mu)$:
for $u\ge0$,
\[
\sum_{i>n}\mu_i(P_nu)(i)
=\sum_{k>n}\mu_ku(k)\sum_{i>n}P_n(k,i)
\le\sum_{k>n}\mu_ku(k),
\]
with the interchange justified by Tonelli.  The Green-barrier argument
below uses $P_n$ on $\ell^\infty$; no assertion that the barrier belongs
to $\ell^2(\mu)$ is needed.

\begin{lemma}[Sub-Markov property]\label{st1:lem:submarkov}
	One has $P1\le1$.
\end{lemma}

\begin{proof}
	Weighted symmetry and the coefficient formula imply
	\[
	K(j,k)=[z^k]\beta(z)^j\qquad(k\ge1).
	\]
	If $p_k=[z^k]\beta(z)^j$ for $k\ge0$, then $(p_k)$ is a probability
	measure and $\sum k p_k=j$, because $\beta(1)=\beta'(1)=1$.
	Moreover
	\[
	P(j,k)=\left(\frac{k}{j}\right)^{1/4}K(j,k).
	\]
	Concavity of $x^{1/4}$ gives
	\[
	P1(j)=j^{-1/4}\sum_{k\ge1}k^{1/4}p_k
	\le j^{-1/4}\left(\sum_{k\ge0}kp_k\right)^{1/4}=1.
	\]
\end{proof}

We shall use the following bound for every source index.

\begin{lemma}[Uniform barrier and all-source Green bound]\label{st1:lem:barrier}
	There exist $b,\gamma_{\mathrm b},C>0$ and $N$ such that, for $n\ge N$, with
	$R=s_n=\sqrt n$ and $x_i=s_i-R$, $i>n$, the bounded positive function
	\[
	F_n(i)=e^{bd}-e^{-bx_i}
	\]
	for a suitable fixed $d>0$ satisfies
	\begin{equation}\label{st1:eq:supersol}
		(I-P_n)F_n(i)\ge \gamma_{\mathrm b} e^{-bx_i},\qquad i>n.
	\end{equation}
	Consequently, for every source index $j\ge1$,
	\begin{equation}\label{st1:eq:allsource}
		\sup_{i>n}\sum_{m\ge0}
		[P_n^m\Kn(\cdot,j)](i)
		\le C e^{b(s_j-R)}.
	\end{equation}
	For $j=n$ the right-hand side is bounded by an absolute constant.
	
	The continuum half-line operator
	\begin{equation*}
		(\mathcal Gf)(x)=\int_0^\infty g(x-y)f(y)\dd y
	\end{equation*}
	also has bounded potential
	\begin{equation}\label{st1:eq:fdef}
		f=\sum_{m\ge0}\mathcal G^m g,
		\qquad \norm{f}_\infty<\infty.
	\end{equation}
\end{lemma}

\begin{proof}
	For fixed $b>0$, put $M_i(b)=\sum_{k\ge1}P(i,k)e^{-b(s_k-s_i)}$.
	The weak Gaussian estimate, the square-root Riemann mesh, and Gaussian
	tail domination imply
	\begin{equation*}
		M_i(b)\longrightarrow
		\int_{\R}g(u)e^{-bu}\dd u=e^{b^2/(4c)}>1.
	\end{equation*}
	Here is the uniformity needed in this passage.  By the mesh-mass bound
	above, for every fixed $\gamma>0$, $b\in\R$, and
	integer $d\ge0$,
	\[
	\sup_t\sum_k\mu_k(1+|s_k-t|)^d
	e^{-\gamma(s_k-t)^2+b|s_k-t|}<\infty,
	\]
	and the corresponding tails tend uniformly to zero as the window
	$|s_k-t|\le A$ expands.  On every fixed window $|s_k-s_i|\le A$,
	\[
	\frac{\mu_k}{s_k-s_{k-1}}
	=\frac{s_k+s_{k-1}}{2s_k}\longrightarrow1
	\]
	uniformly as $i\to\infty$, while \eqref{st1:eq:rooterror} gives
	\[
	|\Kn(i,k)-g(s_i-s_k)|=O_A(i^{-1}).
	\]
	Thus the sum over a fixed window converges to the corresponding Riemann
	integral.  The preceding Gaussian mesh bound, applied to
	\eqref{st1:eq:rootbound}, controls the exponential tilt uniformly and
	permits $A\to\infty$, proving the asserted limit.
	
	Choose $\delta>0$ and $N$ such that $M_i(b)\ge1+\delta$ for $i\ge N$.
	Then choose $d$ so large that
	\begin{equation}\label{st1:eq:tiltedtail}
		\sup_i\sum_{s_i-s_k>d}P(i,k)e^{b(s_i-s_k)}\le\delta/2.
	\end{equation}
	Extending the formula for $F_n$ temporarily to $k\le n$, one obtains
	\[
	(I-P_n)F_n(i)
	=e^{bd}(1-P1(i))+e^{-bx_i}(M_i(b)-1)
	+\sum_{k\le n}P(i,k)F_n(k).
	\]
	Only the last term can have a negative contribution.  If $k\le n$ and
	$F_n(k)<0$, then $R-s_k>d$, and
	\[
	-F_n(k)=e^{b(R-s_k)}-e^{bd}
	\le e^{b(R-s_k)}=e^{-bx_i}e^{b(s_i-s_k)}.
	\]
	Hence, by the choice of $d$ in \eqref{st1:eq:tiltedtail},
	\[
	\sum_{k\le n}P(i,k)(-F_n(k))_+
	\le\frac\delta2e^{-bx_i}.
	\]
	The first term is nonnegative by Lemma~\ref{st1:lem:submarkov}, while
	the second is at least $\delta e^{-bx_i}$ for $i\ge N$.  This proves
	\eqref{st1:eq:supersol}, after decreasing $\gamma_{\mathrm b}$ if necessary.
	
	From \eqref{st1:eq:rootbound}, for all $i>n$ and all $j\ge1$,
	\[
	\Kn(i,j)
	\le C_b e^{b(s_j-R)}e^{-b(s_i-R)},
	\]
	because $e^{-ct^2}\le C_be^{-bt}$ for all real $t$.  A fixed multiple
	of $e^{b(s_j-R)}F_n$ is therefore a supersolution for this source.
	For every finite $M$, positivity and \eqref{st1:eq:supersol} give
	\[
	\sum_{m=0}^MP_n^m\Kn(\cdot,j)
	\le Ce^{b(s_j-R)}\sum_{m=0}^MP_n^m(I-P_n)F_n.
	\]
	The last sum telescopes:
	\[
	\sum_{m=0}^MP_n^m(I-P_n)F_n
	=F_n-P_n^{M+1}F_n\le F_n,
	\]
	because $F_n\ge0$ on the upper half-line and $P_n$ is positive.
	Thus the partial Green potentials are bounded uniformly in $M$ by
	$Ce^{b(s_j-R)}F_n$.  Monotone convergence proves
	\eqref{st1:eq:allsource}.  No disappearance or terminal-iterate
	argument is needed.
	
	For the continuum assertion, put $m_b(x)=\int_{-x}^{\infty}g(t)e^{-bt}\dd t$.
	Then $m_b(x)\uparrow e^{b^2/(4c)}>1$.  Choose $x_0$ and $\delta>0$ so
	that $m_b(x)\ge1+\delta$ for $x\ge x_0$.  If
	$F(x)=A-e^{-bx}$, then
	\[
	(I-\mathcal G)F(x)
	=A\bigl(1-\mathcal G1(x)\bigr)+e^{-bx}(m_b(x)-1).
	\]
	For $x\ge x_0$ this is at least $\delta e^{-bx}$.  On the compact
	interval $[0,x_0]$, the function $1-\mathcal G1(x)$ is strictly positive;
	choosing the constant $A$ sufficiently large makes the same expression
	at least $\gamma_{\mathrm b} e^{-bx}$ there as well.  Since $g(x)\le C_be^{-bx}$ on
	$[0,\infty)$, a fixed multiple of $F$ is a supersolution for $g$.
	Choose $A$ also large enough that $F\ge0$ on $[0,\infty)$.
	Summing only to a finite order and using positivity gives
	\[
	\sum_{m=0}^M\mathcal G^mg
	\le C(F-\mathcal G^{M+1}F)\le CF.
	\]
	Monotone convergence then proves \eqref{st1:eq:fdef}.
\end{proof}

For later use, define $a_m=\mathcal G^m g$, $m\ge0$.
Thus $f=\sum_{m\ge0}a_m$.

\begin{lemma}[Absolute derivative sums]\label{st1:lem:derivatives}
	For $r=0,1,2$,
	\begin{equation*}
		\sup_{x\ge0}\sum_{m\ge0}|a_m^{(r)}(x)|\le C_r.
	\end{equation*}
\end{lemma}

\begin{proof}
	The case $r=0$ is \eqref{st1:eq:fdef}.  For $m\ge1$,
	\[
	a_m^{(r)}(x)=\int_0^\infty g^{(r)}(x-y)a_{m-1}(y)\dd y.
	\]
	Therefore, by positivity and Tonelli,
	\[
	\sum_{m\ge1}|a_m^{(r)}(x)|
	\le \int_0^\infty |g^{(r)}(x-y)|f(y)\dd y\le C_r.
	\]
	Add the $m=0$ term.
\end{proof}

\subsection{Complete chains directly in the Wiener norm}\label{sec:closure-sum}

The following variable-weight quadrature estimate is the only quadrature
formula needed in the proof.

\begin{lemma}[Square-root-mesh quadrature]\label{st1:lem:quad}
	Let $R=\sqrt n$ and $x_i=\sqrt i-R$.  For a smooth rapidly decreasing
	function $h$ on $[0,\infty)$,
	\begin{equation}\label{st1:eq:quad}
		\sum_{i>n}\mu_i h(x_i)
		=\int_0^\infty h(y)\dd y-\frac{h(0)}{4R}+\mathcal R_n[h],
	\end{equation}
	where
	\begin{equation}\label{st1:eq:quadR}
		|\mathcal R_n[h]|\le\frac1{32}\int_0^\infty
		\left(
		\frac{3|h|}{(R+y)^4}
		+\frac{3|h'|}{(R+y)^3}
		+\frac{|h''|}{(R+y)^2}
		\right)\dd y.
	\end{equation}
\end{lemma}

\begin{proof}
	Apply the trapezoidal identity (cf. the Euler--Maclaurin formula
	\cite[\S2.10(i)]{DLMF}), with absolute remainder at most
	$\frac18\int_n^\infty|u''(s)|\dd s$, to
	\[
	u(s)=\frac1{2\sqrt s}h(\sqrt s-R).
	\]
	A direct differentiation gives
	\[
	u''(s)=\frac38s^{-5/2}h-\frac38s^{-2}h'
	+\frac18s^{-3/2}h''.
	\]
	The substitution $s=(R+y)^2$ gives \eqref{st1:eq:quad}--\eqref{st1:eq:quadR}.
\end{proof}

For each fixed $m$, the functions to which the quadrature lemma is
applied satisfy its decay hypotheses.  Indeed, positivity of the
half-line convolution gives
\[
0\le a_m(x)\le g^{*(m+1)}(x),\qquad x\ge0,
\]
and differentiation of the final Gaussian kernel gives rapid decay for
the derivatives needed below.  Uniform decay in $m$ is not required
here; the sum over $m$ is controlled separately by
Lemma~\ref{st1:lem:derivatives}.

Fix $n$, put $k_n(i)=\Kn(i,n)$, $i>n$,
and sample $a_m$ by $a_m^{\mathrm s}(i)=a_m(x_i).$ Define residuals
\begin{equation}\label{st1:eq:em}
	e_0=a_0^{\mathrm s}-k_n,
	\qquad
	e_m=a_m^{\mathrm s}-P_na_{m-1}^{\mathrm s},\quad m\ge1.
\end{equation}

\begin{lemma}[Residuals summed over all lengths]\label{st1:lem:residuals}
	For all sufficiently large $n$,
	\begin{equation}\label{st1:eq:residualsum}
		\sum_{m\ge0}\norm{e_m}_{\ell^1(\mu)}\le\frac{C}{\sqrt n}.
	\end{equation}
\end{lemma}

The two error sources are kernel replacement and quadrature.
The detailed proof is given in Appendix~\ref{app:technical-estimates},
Subsection~\ref{app:residual-proof}.

Let
\begin{equation}\label{st1:eq:fn}
	f_n=\sum_{h\ge0}P_n^hk_n.
\end{equation}
Lemma~\ref{st1:lem:barrier} gives $\norm{f_n}_\infty\le C$.

\begin{lemma}[Coefficientwise discrete--continuum comparison]
	\label{st1:lem:coeffcompare}
	For all sufficiently large $n$,
	\begin{equation}\label{st1:eq:paircompare}
		\sum_{m\ge0}
		\left|
		\langle k_n,P_n^mk_n\rangle_\mu
		-\int_0^\infty g(x)a_m(x)\dd x
		\right|
		\le \frac{C}{\sqrt n},
	\end{equation}
	where $\langle u,v\rangle_\mu=\sum_{i>n}\mu_i u(i)v(i)$.
\end{lemma}

\begin{proof}
	The finite telescoping identity
	\begin{equation}\label{st1:eq:finite-telescope}
		a_m^{\mathrm s}-P_n^mk_n
		=\sum_{r=0}^mP_n^{m-r}e_r
	\end{equation}
	holds for every $m$.  The limiting Green potential $f_n$ is used here
	only as a bounded function.  The residuals belong to $\ell^1(\mu)$ by
	Lemma~\ref{st1:lem:residuals}, so the pairings below are
	$\ell^\infty$--$\ell^1(\mu)$ pairings; no claim
	$f_n\in\ell^2(\mu)$ is needed.  For finite powers, weighted symmetry gives
	\[
	\langle k_n,P_n^he_r\rangle_\mu
	=\langle P_n^hk_n,e_r\rangle_\mu.
	\]
	After taking absolute values, positivity, Tonelli,
	Lemma~\ref{st1:lem:residuals}, and \eqref{st1:eq:fn} give
	\begin{align}
		\sum_{m\ge0}
		\left|\langle k_n,a_m^{\mathrm s}-P_n^mk_n\rangle_\mu\right|
		&\le \sum_{r,h\ge0}\langle P_n^hk_n,|e_r|\rangle_\mu\notag\\
		&=\sum_{r\ge0}\langle f_n,|e_r|\rangle_\mu
		\le\frac{C}{\sqrt n}.\label{st1:eq:firstpair}
	\end{align}
	The identity \eqref{st1:eq:finite-telescope} is finite before the sum over
	$m$ is taken.  The last inequality uses
	\[
	\sum_{r\ge0}\langle f_n,|e_r|\rangle_\mu
	\le\norm{f_n}_\infty\sum_{r\ge0}\norm{e_r}_{\ell^1(\mu)}.
	\]
	
	It remains to compare $\langle k_n,a_m^{\mathrm s}\rangle_\mu$ with its
	continuum integral.  Replacing $k_n(i)$ by $g(x_i)$ contributes
	$O(R^{-2})$ after summing over $m$, by \eqref{st1:eq:rooterror} and boundedness
	of $f$.  Apply Lemma~\ref{st1:lem:quad} to
	$h_m(y)=g(y)a_m(y)$.  Lemma~\ref{st1:lem:derivatives} again makes the sums of
	all boundary and remainder terms absolutely convergent and gives a total
	$O(R^{-1})$ error.  Combining this with \eqref{st1:eq:firstpair} proves
	\eqref{st1:eq:paircompare}.
\end{proof}

\begin{lemma}[Gaussian half-line chains]\label{st1:lem:cycle}
	For $\ell\ge2$,
	\begin{equation}\label{st1:eq:cycle}
		\int_0^\infty g(x)a_{\ell-2}(x)\dd x
		=\frac1\ell g^{*\ell}(0)
		=\frac{2\cg}{\ell^{3/2}}.
	\end{equation}
	The same formula for $\ell=1$ is $g(0)=2\cg$.
\end{lemma}

\begin{proof}
	The first integral is the density at total increment zero of an
	$\ell$-step Gaussian bridge whose intermediate partial sums are positive.
	The zero-sum cyclic-shift principle is classical
	\cite[Theorem~2.1]{spitzer}; in the present setting its argument is short.
	Without conditioning on a probability-zero event, normalize the density
	\[
	\prod_{r=1}^{\ell}g(x_r-x_{r-1}),\qquad x_0=x_\ell=0,
	\]
	on $\R^{\ell-1}$ by the factor $g^{*\ell}(0)$.
	Cyclic permutation of the increments preserves this density and Lebesgue
	measure.  The exceptional set on which two of the partial sums
	$x_0,\ldots,x_{\ell-1}$ agree is a finite union of proper hyperplanes
	and has measure zero.  Off this set there is a unique minimum, and
	exactly one cyclic shift starts immediately after it; that shift has
	all intermediate partial sums positive.  Hence the normalized
	positive-half-line mass is exactly $1/\ell$, and the original integral
	is $g^{*\ell}(0)/\ell$.  Since $g$ is the centered Gaussian
	with variance $1/(2c)$,
	\[
	g^{*\ell}(0)=\sqrt{\frac{c}{\pi\ell}}
	=\frac{2\cg}{\sqrt\ell}.
	\]
	This proves \eqref{st1:eq:cycle}.
\end{proof}

\begin{proposition}[Wiener-norm chain asymptotic]\label{st1:prop:chainWiener}
	For all sufficiently large $n$,
	\begin{equation}\label{st1:eq:chainWiener}
		\sum_{\ell\ge1}
		\left|S_\ell(n)-\frac{\cg}{\sqrt n\,\ell^{3/2}}\right|
		\le\frac{C}{n}.
	\end{equation}
	Equivalently, \eqref{st1:eq:masterS} holds.
\end{proposition}

\begin{proof}
	The exact symmetrized chain identities are
	\begin{equation*}
		S_1(n)=\mu_n\Kn(n,n),
		\qquad
		S_{m+2}(n)=\mu_n\langle k_n,P_n^mk_n\rangle_\mu,
		\quad m\ge0.
	\end{equation*}
	Since $\mu_n=(2\sqrt n)^{-1}$, Lemma~\ref{st1:lem:cycle} gives
	\[
	\mu_n\frac{2\cg}{\ell^{3/2}}
	=\frac{\cg}{\sqrt n\,\ell^{3/2}}.
	\]
	Thus the summed $O(n^{-1/2})$ comparison error of
	Lemma~\ref{st1:lem:coeffcompare} becomes $O(n^{-1})$ after multiplication
	by $\mu_n$, over all $m\ge0$.
	For $\ell=1$, \eqref{st1:eq:rooterror} gives
	\[
	\mu_n\Kn(n,n)=\frac{\cg}{\sqrt n}+O(n^{-3/2}).
	\]
	This proves \eqref{st1:eq:chainWiener}.
\end{proof}

In particular,
\begin{equation*}
	\sum_{\ell\ge1}S_\ell(n)
	=\frac{\cg\zeta(3/2)}{\sqrt n}+O(n^{-1}),
	\qquad
	\norm{\Sser_n}_{\Walg}=O(n^{-1/2}).
\end{equation*}
The estimate controls the sum of the absolute errors over all lengths.

\subsection{Localized upper resolvents}

For a Banach space $X$, let
\[
\Walg(X)=\left\{F(\omega)=\sum_{r\ge0}F_r\omega^r:
\norm{F}_{\Walg(X)}=\sum_{r\ge0}\norm{F_r}_X<\infty\right\}.
\]
Fix $n$ and put
\[
L=\{1,\ldots,n-1\},\qquad U=\{n+1,n+2,\ldots\},
\]
\begin{equation}\label{st1:eq:QZ}
	Q=\diag_{i>n}|\rho|^{(i-n)/2},
	\qquad
	Z=\diag_{j<n}|\rho|^{(n-j)/2},
	\qquad
	\Gres(\omega)=(I-\omega\Hker_{UU})^{-1}.
\end{equation}
For $|\omega|<1$, define
\begin{align}
	C_U&=Q\Gres Q,\label{st1:eq:CU}\\
	b&=Q\Gres\,\omega\Hker_{Un},\label{st1:eq:bdef}\\
	F&=Q\Gres\,\omega\Hker_{UL}Z,\label{st1:eq:Fdef}\\
	E^0_{ij}&=\omega\Hker_{ij}
	+\omega^2\Hker_{iU}\Gres\Hker_{Uj},
	\qquad i,j\le n,\label{st1:eq:Edef}\\
	v&=ZE^0_{Ln}.\label{st1:eq:vdef}
\end{align}
The coefficient series of these sandwiched quantities extend to the
closed unit disk in the Wiener norms below.

\begin{lemma}[Localized Wiener estimates]\label{st1:lem:localization}
	For all sufficiently large $n$,
	\begin{align}
		\norm{C_U}_{\Walg(\mathrm{op})}&\le |\rho|+Cn^{-1/2},\label{st1:eq:locCU}\\
		\norm{b}_{\Walg(\ell^2)}+\norm{F}_{\Walg(\mathrm{op})}
		&\le Cn^{-1/2},\label{st1:eq:locbF}\\
		\norm{ZE^0Z}_{\Walg(\mathrm{op})}+\norm{v}_{\Walg(\ell^2)}
		&\le Cn^{-1/2},\notag
	\end{align}
	where $ZE^0Z$ means $ZE^0_{LL}Z$.  Moreover
	\begin{equation}\label{st1:eq:E0nn}
		E^0_{nn}(\omega)=\Sser_n(\omega).
	\end{equation}
\end{lemma}

The detailed entrywise and trace-majorant proof is given in
Appendix~\ref{app:technical-estimates}, Subsection~\ref{app:localization-proof}.

\subsection{Direct transfer from complete chains to the canonical series}
\label{subsec:direct-canonical-transfer}\label{sec:closure-transfer}

We now use the localized estimates only to compare the already-established
canonical function \(\Yser_n\) with the strict upper-chain series \(\Sser_n\).

Retain the notation of the preceding localized-resolvent subsection and put
$J=\diag_{i>n}(-1)^{i-n}$.
\begin{proposition}[Direct Wiener transfer to the canonical series]
	\label{prop:direct-canonical-transfer}
	For all sufficiently large $n$,
	\[
	\norm{\Yser_n-\Sser_n}_{\Walg}
	+\norm{\Wser_n-1-\Sser_n}_{\Walg}\le Cn^{-1}.
	\]
	Consequently \eqref{st1:eq:masterYW} and \eqref{st1:eq:individualW} hold.
\end{proposition}
\begin{proof}  For
	the rest of this subsection \(n\) is sufficiently large.
	The spectral input used here is upstream of the present quantitative
	argument: Proposition~\ref{prop:global-canonical-factorization} identifies
	the canonical scalar functions \(\Lambda_n(\omega)\) globally from the
	all-index Wiener bounds and the local trace identity.  The present Schur
	reduction uses those already identified eigenvalues to obtain the sharper
	Wiener estimate for \(\Yser_n-\Sser_n\); it does not use the asymptotic
	being proved to identify the branch.
	
	By the all-index
	Wiener bound,
	\begin{equation}\label{eq:streamline-T-1}
		\|\Yser_n\|_{\Walg}<\frac35,
		\qquad
		\|1-\Yser_n\|_{\Walg}<\frac85.
	\end{equation}
	For \(j<n\), define
	\begin{equation}\label{eq:streamline-T-2}
		B_n(\omega)_{jj}
		=
		\frac1{\phi^{-2(n-j)}-(1-\Yser_n(\omega))(-1)^{j-n}}.
	\end{equation}
	If \(n-j\) is odd, the constant part of the denominator in
	\eqref{eq:streamline-T-2} has modulus \(1+\phi^{-2(n-j)}\); if \(n-j\) is even, it has modulus
	\(1-\phi^{-2(n-j)}\ge1-\phi^{-4}\).  This is an operator-valued Wiener estimate, not
	merely a pointwise bound on the diagonal entries.  Let \(B_n^{(0)}\)
	denote the diagonal inverse obtained by setting \(\Yser_n=0\); its
	denominators have modulus at least \(1-\phi^{-4}\), so
	\(\|B_n^{(0)}\|_{\mathrm{op}}\le(1-\phi^{-4})^{-1}\).
	Factoring out this constant diagonal inverse and expanding geometrically
	in the scalar Wiener element \(\Yser_n\) gives
	\begin{equation}\label{eq:streamline-T-3}
		\|B_n\|_{\Walg(\mathrm{op})}
		\le
		\frac1{1-\phi^{-4}-\|\Yser_n\|_{\Walg}}
		<
		4,
	\end{equation}
	because \(\phi^{-4}<3/20\).
	
	Put \(\alpha_n=1-\Yser_n\), and define the block operator
	\begin{equation}\label{eq:streamline-T-4}
		\mathcal M_n=
		\begin{pmatrix}
			\alpha_n C_UJ & F\\
			\alpha_n B_nF^tJ & B_nZE^0Z
		\end{pmatrix}.
	\end{equation}
	Here \(t\) denotes ordinary transpose, not Hilbert-space adjoint.
	The coefficient matrices of \(\Gres\) are symmetric, so
	\(\Gres(\omega)^t=\Gres(\omega)\) and
	\(F^t=Z\,\omega\Hker_{LU}\Gres Q\).
	No complex conjugation of \(\omega\) is involved.  For norm estimates,
	ordinary transpose has the same Hilbert-space operator norm.
	The localized Wiener estimates give
	\begin{equation*}
		\|C_U\|_{\Walg(\mathrm{op})}
		\le
		\phi^{-2}+O(n^{-1/2}),
	\end{equation*}
	while
	\begin{equation*}
		\|F\|_{\Walg(\mathrm{op})}
		+\|ZE^0Z\|_{\Walg(\mathrm{op})}
		=O(n^{-1/2}).
	\end{equation*}
	For the maximum direct-sum norm, and also for the Hilbert direct-sum
	norm used below, the triangle inequality for the four blocks gives
	\[
	\begin{aligned}
		\|\mathcal M_n\|_{\Walg(\mathrm{op})}
		\le{}&\|\alpha_n\|_{\Walg}\|C_U\|_{\Walg(\mathrm{op})}
		+\|F\|_{\Walg(\mathrm{op})}\\
		&+\|\alpha_n\|_{\Walg}\|B_n\|_{\Walg(\mathrm{op})}
		\|F\|_{\Walg(\mathrm{op})}
		+\|B_n\|_{\Walg(\mathrm{op})}\|ZE^0Z\|_{\Walg(\mathrm{op})}.
	\end{aligned}
	\]
	Using \eqref{eq:streamline-T-1} and \eqref{eq:streamline-T-3}, the right
	side is at most $(8/5)\phi^{-2}+O(n^{-1/2})$, so
	\begin{equation}\label{eq:streamline-T-7}
		\|\mathcal M_n\|_{\Walg(\mathrm{op})}
		\le
		\frac85\phi^{-2}+O(n^{-1/2})
		<1
	\end{equation}
	for all sufficiently large \(n\).  The norm here is the ordinary operator
	norm on the Hilbert direct sum \(\ell^2(U)\oplus\ell^2(L)\),
	coefficientwise summed in the Wiener algebra.  For sufficiently large
	\(n\), the right side of \eqref{eq:streamline-T-7} is bounded by a fixed
	\(\theta<1\).  Hence
	\[
	(I-\mathcal M_n)^{-1}=\sum_{r\ge0}\mathcal M_n^r
	\in\Walg(\mathrm{op}),\qquad
	\|(I-\mathcal M_n)^{-1}\|_{\Walg(\mathrm{op})}\le\frac1{1-\theta}.
	\]
	Thus only the elementary Neumann theorem in the Banach algebra is used.
	
	For each \(|\omega|<1\), the global spectral factorization and
	Proposition~\ref{t2:prop:bridge} show that
	\[
	\Lambda_n(\omega)
	=
	\frac{\lambda_n^{(0)}}{1-\Yser_n(\omega)}
	\]
	is an eigenvalue of the symmetrized pencil
	\begin{equation}\label{eq:streamline-T-8}
		Df=\Lambda_n(\omega)(I-\omega\Hker)f.
	\end{equation}
	Set \(u=Qf_U\) and \(z=Z^{-1}f_L\).  The upper equation gives
	\[
	f_U=\Gres\bigl(\alpha_n QJu+\omega\Hker_{UL}Zz
	+\omega\Hker_{Un}f_n\bigr).
	\]
	On the lower coordinates,
	\[
	\bigl(I-\alpha_n\diag_{j<n}(d_j/d_n)\bigr)^{-1}=ZB_nZ.
	\]
	Before normalizing the eigenvector, substitution into the lower equation
	gives the block system
	\[
	\binom{u}{z}=\mathcal M_n\binom{u}{z}+\binom{b}{B_nv}\,f_n,
	\]
	and the central equation is
	\[
	\Yser_nf_n=\Sser_nf_n+\alpha_n b^tJu+v^tz.
	\]
	Thus the following argument does not assume \(f_n\ne0\) in advance.
	If \(f_n=0\), invertibility from \eqref{eq:streamline-T-7} forces
	\(u=z=0\); the upper reconstruction equation then gives \(f_U=0\),
	while \(f_L=Zz=0\).  This contradicts \(f\ne0\).
	Hence \(f_n\ne0\), and normalizing \(f_n=1\), we obtain
	\begin{equation}\label{eq:streamline-T-9}
		\binom{u}{z}
		=
		\mathcal M_n
		\binom{u}{z}
		+
		\binom{b}{B_nv},
	\end{equation}
	and the central equation becomes
	\begin{equation}\label{eq:streamline-T-10}
		\Yser_n=\Sser_n+\mathcal E_n,
	\end{equation}
	where
	\begin{equation}\label{eq:streamline-T-11}
		\mathcal E_n
		=
		\bigl(\alpha_n b^tJ,\ v^t\bigr)
		(I-\mathcal M_n)^{-1}
		\binom{b}{B_nv}.
	\end{equation}
	Here \(b\) is the localized vector from \eqref{st1:eq:bdef}.
	The derivation up to this point is made only for \(|\omega|<1\), where
	\(\Gres(\omega)=(I-\omega\Hker_{UU})^{-1}\) is a bounded operator.
	The identity holds throughout that disk, and both scalar sides of
	\eqref{eq:streamline-T-10}--\eqref{eq:streamline-T-11} belong to the
	scalar Wiener algebra; hence it is an identity of their coefficient
	series.  Their continuous values at \(\omega=1\) are obtained by Wiener
	evaluation.  No bounded operator \(\Gres(1)=(I-\Hker_{UU})^{-1}\) is
	invoked in this passage.
	
	The localized Wiener estimates give
	\begin{equation*}
		\|b\|_{\Walg(\ell^2)}
		+\|v\|_{\Walg(\ell^2)}
		=O(n^{-1/2}).
	\end{equation*}
	Together with \eqref{eq:streamline-T-1} and \eqref{eq:streamline-T-3},
	this shows that the exterior row and column in
	\eqref{eq:streamline-T-11} both have Wiener norm $O(n^{-1/2})$:
	\[
	\|(\alpha_n b^tJ,v^t)\|_{\Walg(\mathrm{op})}=O(n^{-1/2}),\qquad
	\left\|\binom{b}{B_nv}\right\|_{\Walg(\mathrm{op})}=O(n^{-1/2}).
	\]
	Here the row and column are operators between the Hilbert direct sum
	and $\C$.  Since $(I-\mathcal M_n)^{-1}$ is uniformly bounded in
	$\Walg(\mathrm{op})$, submultiplicativity gives
	$\|\mathcal E_n\|_{\Walg}=O(n^{-1})$ directly.  Thus
	\begin{equation}\label{eq:streamline-T-13}
		\boxed{\|\Yser_n-\Sser_n\|_{\Walg}=O(n^{-1}).}
	\end{equation}
	Since the complete-chain estimate gives
	\[
	\|\Sser_n\|_{\Walg}=O(n^{-1/2}),
	\]
	we obtain
	\begin{equation}\label{eq:streamline-T-14}
		\|\Yser_n\|_{\Walg}=O(n^{-1/2}).
	\end{equation}
	
	Finally
	\begin{equation}\label{eq:streamline-T-15}
		\Wser_n-1-\Yser_n
		=
		\Yser_n^2\Wser_n.
	\end{equation}
	The all-index bound \(\|\Wser_n\|_{\Walg}<5/2\), together with
	\eqref{eq:streamline-T-14}, gives
	\[
	\|\Wser_n-1-\Yser_n\|_{\Walg}=O(n^{-1}),
	\]
	and hence
	\begin{equation}\label{eq:streamline-T-16}
		\boxed{\|\Wser_n-1-\Sser_n\|_{\Walg}=O(n^{-1}).}
	\end{equation}
	
	Combining \eqref{eq:streamline-T-13} and \eqref{eq:streamline-T-16} with the complete-chain
	Wiener asymptotic gives
	\begin{equation}\label{eq:streamline-T-17}
		\left\|
		\Yser_n-\frac{\cg}{\sqrt n}
		\sum_{\ell\ge1}\frac{\omega^\ell}{\ell^{3/2}}
		\right\|_{\Walg}
		+
		\left\|
		\Wser_n-1-\frac{\cg}{\sqrt n}
		\sum_{\ell\ge1}\frac{\omega^\ell}{\ell^{3/2}}
		\right\|_{\Walg}
		=
		O(n^{-1}).
	\end{equation}
	This is exactly the quantitative coefficientwise conclusion required for
	the two-term asymptotic.
	
	Coefficient extraction also retains the fixed-length conclusion
	\[
	W_\ell(n)=\frac{\cg}{\ell^{3/2}\sqrt n}+O_\ell(n^{-1})
	\qquad(\ell\ge1\text{ fixed}),
	\]
	and the absolute coefficient transfer
	\[
	\sum_{\ell\ge1}|Y_\ell(n)-S_\ell(n)|
	+\sum_{\ell\ge1}|W_\ell(n)-S_\ell(n)|\le Cn^{-1}.
	\]
\end{proof}

\subsection{Completion of the two-term asymptotic}
\label{subsec:asymptotic-completion}

\begin{proof}[Completion of Theorem~\ref{thm:target}]
	Evaluate \eqref{eq:streamline-T-17} at \(\omega=1\).  Since
	\(\sum_{\ell\ge1}\ell^{-3/2}=\zeta(3/2)\),
	we obtain
	\begin{equation}\label{eq:streamline-C-1}
		\Wser_n(1)
		=
		1+\frac{\cg\zeta(3/2)}{\sqrt n}+O(n^{-1}).
	\end{equation}
	The exact all-index spectral identity \eqref{eq:allindex-eigenvalue-series} was proved before the
	large-\(n\) approximation.  Therefore
	\begin{equation}\label{eq:streamline-C-2}
		\lambda_n
		=
		(-1)^{n+1}\phi^{-2n}
		\left(
		1+\frac{\cg\zeta(3/2)}{\sqrt n}+O(n^{-1})
		\right).
	\end{equation}
	This is precisely the asserted two-term formula
	\begin{equation}\label{eq:streamline-C-3}
		(-1)^{n+1}\lambda_n
		=
		\phi^{-2n}
		+
		\frac{5^{1/4}\zeta(3/2)}{2\sqrt\pi}\,
		\frac{\phi^{-2n}}{\sqrt n}
		+
		O\!\left(\frac{\phi^{-2n}}{n}\right).
	\end{equation}
	Define \(d(n)\) by \eqref{eq:target}. The \(O(n^{-1})\) remainder makes it
	bounded on the large-\(n\) tail, and the finitely many remaining values
	preserve boundedness.
	
	All-index algebraic simplicity, alternating signs, and strict modulus
	ordering were already obtained in
	Subsection~\ref{subsec:global-canonical-spectrum}.  This completes
	Theorem~\ref{thm:target}.
\end{proof}

\subsection{Consequences of the completed theorem}
The asymptotic also gives
\[
\frac{\lambda_n}{\lambda_{n+1}}=-\phi^2\left(1+O(n^{-1})\right).
\]
Indeed, the $C/\sqrt n$ terms in Theorem~\ref{thm:target} differ by
$O(n^{-3/2})$, while the remainders are $O(n^{-1})$.
Effective cutoffs for the asymptotic estimates remain a separate question.

\appendix
\section{The coefficients \texorpdfstring{$K(i,j)$ and $\mathbf K(i,j)$}{K(i,j) and the integer array}}\label{app:exact-kernel}

\subsection{Generating functions, symmetry and sums}
\begin{proposition}\label{prop:exact-kernel}
	The following properties hold:
\begin{itemize}
\item[\emph{i)}]The coefficients are positive and
\[
\mathbf K(j,\ell)=\sum_{a=0}^{\min(j-1,\ell-1)}
4^a\binom{\ell-1}{a}\binom{\ell+j-a-1}{j-a-1}.
\]
\item[\emph{ii)}]Their bivariate generating functions are
\begin{align}
\Theta_{\mathbf K}(x,y)&=\sum_{j,\ell\ge1}\mathbf K(j,\ell)x^{j-1}y^{\ell-1}
=\frac1{(1-x)(1-x-y-4xy)},\label{eq:int-bivar}\\
\Theta_K(x,y)&=\sum_{j,\ell\ge1}K(j,\ell)x^{j-1}y^{\ell-1}
=\frac5{(2\phi-x)(2\phi-x-y-2\phi^{-1}xy)}.\label{eq:scaled-bivar}
\end{align}
\item[\emph{iii)}] $\mathbf K(1,\ell)=1$, $\mathbf K(j,1)=j$, and 

\begin{eqnarray}
\ell K(j,\ell)=jK(\ell,j),&\,& \ell\mathbf{K}(j,\ell)=j\mathbf{K}(\ell,j),\label{eq:weighted-symmetry}\\
\sum_{j\ge1}K(j,\ell)=1,&\,&
\sum_{\ell\ge1}K(j,\ell)=1-(2\phi)^{-j}.\label{eq:kernel-sums}.
\end{eqnarray}
In terms of $\Theta_{\mathbf K}$, \ref{eq:weighted-symmetry} reads as
\begin{eqnarray*}
	\frac{\partial}{\partial y}\Big{(}y\Theta_{\mathbf K}(x,y)\Big{)}=\frac{1}{(1-x-y-4xy)^{2}},
\end{eqnarray*}
which is symmetric in $(x,y)$;
\item[\emph{iv)}]With $\beta$ from \eqref{t2:eq:kernel},
\begin{equation}\label{eq:row-generator}
\sum_{\ell\ge1}K(j,\ell)y^{\ell-1}
=\frac{\beta(y)^j-\beta(0)^j}{y}.
\end{equation}
The apparent singularity at zero is removable, with value
$5j/(2\phi)^{j+1}$.

\item[\emph{v)}]$\m{K}(j,\ell)-\m{K}(j-1,\ell)=\m{K}(\ell,j)-\m{K}(\ell-1,j)$, for $j,\ell\geq 2$; in terms of $\Theta_{\mathbf K}$, this reads as
\begin{eqnarray*}
	(1-x)\Theta_{\mathbf K}(x,y)=\frac{1}{1-x-y-4xy},
\end{eqnarray*}
which is symmetric in $(x,y)$;
\item[\emph{vi)}]$\m{K}(j,\ell)=\m{K}(j-1,\ell)+4\m{K}(j-1,\ell-1)+\m{K}(j,\ell-1)$, for $j,\ell\geq 2$;
\item[\emph{vii)}]For $n\geq 3$, we have the recurrence
\begin{eqnarray*}
	(n+1)\m{K}(n+1)-(8n+10)\m{K}(n)-(32n-72)\m{K}(n-1)+64(n-2)\m{K}(n-2)=0.
\end{eqnarray*}

	\item[\emph{viii)}]Let $\m{K}(n)=\m{K}(n,n)$; the diagonal of the rational function $\Theta_{\mathbf K}$ is given by
\begin{eqnarray*}
	\m{\Omega}(w)=\sum\limits_{n=1}^{\infty}\m{K}(n)w^{n}=\frac{\Omega_{1}(4\phi^2 w)}{5}
	=\frac{1+4w-\sqrt{1-12w+16w^2}}{10\sqrt{1-12w+16w^2}};
\end{eqnarray*}

\end{itemize}
\end{proposition}
\begin{proof}
The finite sum follows by expanding both factors in
\eqref{eq:integer-kernel}; summing the column generating functions gives
\eqref{eq:int-bivar}. Rescaling gives \eqref{eq:scaled-bivar}.
The derivative coefficient identity
$K(j,\ell)=[t^{j-1}]\beta'(t)\beta(t)^{\ell-1}$ is the kernel identity
used in Section~\ref{sec:germs}. Its row generating function is
\eqref{eq:row-generator}. Coefficient comparison gives
\eqref{eq:weighted-symmetry}; evaluation at one gives
\eqref{eq:kernel-sums}, since $\beta(1)=\beta'(1)=1$.
The positive finite sum justifies these evaluations by monotone
convergence as well. These calculations also give the value at zero.
\end{proof}

\subsection{The increment law and its moments}
Let $Z_i$ have probability generating function $\beta$. Its increment law is
\[
\Pr\{Z_i=0\}=\frac1{2\phi},\qquad
\Pr\{Z_i=h\}=\frac5{(2\phi)^{h+1}}\quad(h\ge1),
\]
with mean one and variance $2/\sqrt5$. Normalization and the moments
follow by evaluating the generating function and its derivatives at one.

\subsection{Infinite integer matrix}

Thus, $\m{\Omega}$ satisfies
\begin{eqnarray*}
	\big{(}5\m{\Omega}^2+\m{\Omega}\big{)}(1-12+16w^2)=w.
\end{eqnarray*}

The sequence $\{\m{K}(n),n\in\mathbb{N}\}$ is not yet contained in Online Encyclopedia of Integer Sequences \cite{oeis}. However, it can be expressed in terms of the sequence A084772, which is given by
\begin{eqnarray*}
	\frac{1}{\sqrt{1-12w+16w^2}}=\sum\limits_{n=0}^{\infty}a(n)w^{n},
\end{eqnarray*}
as follows:
\begin{eqnarray*}
	\m{K}(n)=\frac{a(n)+4a(n-1)}{10},\quad n\geq 1.
\end{eqnarray*}
Finally, note that
\begin{eqnarray*}
	\m{\Omega}'(w)=\frac{1-4w}{(1-12w+16w^2)^{3/2}}.
\end{eqnarray*} 
This gives the identity
\begin{eqnarray*}
	\big{(}10\m{\Omega}(w)+1\big{)}\cdot(1-4w)=\m{\Omega}'(w)\cdot(1+4w)\cdot\big{(}1-12w+16w^2\big{)}.
\end{eqnarray*}
Comparing the coefficient at $w^{n}$ for $n\geq 3$ gives item vii) of Proposition \ref{prop:exact-kernel}.\\
The function $\m{K}(n)$, defined for natural numbers, has a natural extension as a meromorphic function 
\[
\m{K}(z):=\frac{\Gamma(2z)}{\Gamma(z)\Gamma(z+1)}
\,{}_2F_1(1-z,\,1-z;\,1-2z;\,-4)=\frac{\Gamma(2z)\cdot 5^{z-1}}{\Gamma(z)\Gamma(z+1)}
\,{}_2F_1(1-z,\,-z;\,1-2z;\,4/5).
\]
\subsection{$\m{K}(j,\ell)$ as a function in $\ell$}\label{funct-Y}
\label{poli-Y}
We can rewrite item v) of Proposition \ref{prop:exact-kernel}, coupled with the weighted symmetry property iii), as follows:
\begin{eqnarray*}
	(j-\ell)\m{K}(j,\ell)=j\m{K}(j-1,\ell)-(\ell-1)\m{K}(j,\ell-1).
\end{eqnarray*}
Now, let us use vi) to express $\m{K}(j,\ell-1)$, and substitute this into the above expression. We obtain:
\begin{eqnarray}
	(j-1)\m{K}(j,\ell)=(\ell+j-1)\m{K}(j-1,\ell)+4(\ell-1)\m{K}(j-1,\ell-1).
	\label{poly-rec}
\end{eqnarray}
Let us now introduce
\begin{eqnarray*}
	\m{K}(j,t)=Y_{j}(t),
\end{eqnarray*}
where $t$ is an integer, but this time we consider the latter expression as a polynomial in $t$. The recurrence (\ref{poly-rec}) in terms of $Y$ reads as
\begin{eqnarray*}
	(j-1)Y_{j}(t)=(t+j-1)Y_{j-1}(t)+4(t-1)Y_{j-1}(t-1),\quad Y_{1}(t)=1.
\end{eqnarray*}
Thus, the polynomials we obtain are as follows:\small
\begin{eqnarray*}
	Y_{1}(t)&=&1,\\
	Y_{2}(t)&=&5t-3,\\
	Y_{3}(t)&=&\frac{25}{2}t^2-\frac{45}{2}t+13,\\
	Y_{4}(t)&=&\frac{125}{6}t^3-75t^2+\frac{655}{6}t-51,\\
	Y_{5}(t)&=&\frac{625}{24}{t^4}-\frac{625}{4}t^3+\frac{9875}{24}t^2-\frac{1925}{4}t+205,\\
	Y_{6}(t)&=&\frac{625}{24}t^5-\frac{1875}{8}t^4+\frac{23125}{24}t^3-\frac{16125}{8}t^2+\frac{25025}{12}t-819,\\
	Y_{7}(t)&=&\frac{3125}{144}t^6-\frac{4375}{16}t^5+\frac{231875}{144}t^4-\frac{83125}{16}t^3+\frac{84875}{9}t^2-\frac{35455}{4}t+3277,\\
	Y_{8}(t)&=&\frac{15625}{1008}t^7-\frac{3125}{12}t^6+\frac{149375}{72}t^5-\frac{56875}{6}t^4+\frac{3769375}{144}t^3-\frac{512225}{12}t^2+\frac{3130985}{84}t-13107.
\end{eqnarray*}\normalsize
Now, let
\begin{eqnarray*}
	\sum\limits_{j=1}^{\infty}Y_{j}(t)x^{j-1}=\psi(t,x).
\end{eqnarray*}
We know that for $\ell\in\mathbb{N}$,
\begin{eqnarray*}
	\psi(\ell,x)=\frac{(4x+1)^{\ell-1}}{(1-x)^{\ell+1}}.
\end{eqnarray*}
And in general, for $t\in\mathbb{R}$,
\begin{eqnarray*}
	\psi(t,x)=\frac{(4x+1)^{t-1}}{(1-x)^{t+1}}.
\end{eqnarray*}
Indeed, the initial condition is $\psi(t,0)=1$, and we note that the Taylor coefficients of the latter function satisfy the same recurrence, minding the differential equation
\begin{eqnarray*}
	(1-x)\frac{\partial}{\partial x}\,\psi(t,x)=(t+1)\psi(t,x)+4(t-1)\psi(t-1,x).
\end{eqnarray*}
Thus, 
\begin{eqnarray*}
	Y_{j}(t)=\sum\limits_{s=0}^{j-1}\binom{t-1}{s}4^{s}(-1)^{j+s+1}\binom{-t-1}{j-s-1}.
\end{eqnarray*} 
We will need these polynomials further, starting from Subsection \ref{v-lyg-2}.

\section{Asymptotics of the scaled kernel}\label{app:kernel-asymptotics}
Use the integer array and scaled kernel of \eqref{eq:integer-kernel}.
By \eqref{t2:eq:kernel},
\begin{equation}\label{ka:eq:beta}
 \sum_{j\ge1}K(j,k)z^{j-1}=\beta'(z)\beta(z)^{k-1}.
\end{equation}
Then
\begin{equation}\label{ka:eq:coefficient}
 K(j,k)=\frac jk[z^j]\beta(z)^k.
\end{equation}

The weighted symmetry $kK(j,k)=jK(k,j)$ was proved in
\eqref{eq:weighted-symmetry} of Appendix~\ref{app:exact-kernel}. Hence the kernel $\Kh$ used in Section~\ref{sec:asymptotic}
is symmetric. Write
\begin{equation}\label{ka:eq:root-coords}
 p=\sqrt{jk},\qquad \xi_{\mathrm K}=\sqrt j-\sqrt k,
 \qquad P_4(\xi_{\mathrm K})=\frac{\sqrt5}{40}-\frac54\xi_{\mathrm K}^2+
                 \frac{5\sqrt5}{12}\xi_{\mathrm K}^4.
\end{equation}
Here $P_4$ is a local name for this quartic polynomial, not a Jacobi polynomial.

\begin{theorem}[Uniform weighted kernel expansion]\label{ka:thm:kernel}
There exist constants $C,\gamma_{\mathrm K}>0$, independent of $j,k\in\N$, such that
\begin{equation}\label{ka:eq:kernel-expansion}
 \Kh(j,k)=\frac{\cg}{\sqrt p}e^{-\sqrt5\xi_{\mathrm K}^2}
      \left(1+\frac{P_4(\xi_{\mathrm K})}p\right)+\mathcal E_K(j,k),
\end{equation}
where
\begin{equation}\label{ka:eq:kernel-remainder}
 |\mathcal E_K(j,k)|\le
 C p^{-5/2}(1+|\xi_{\mathrm K}|^8)e^{-\gamma_{\mathrm K}\xi_{\mathrm K}^2}.
\end{equation}
In particular, after possibly changing $C,\gamma_{\mathrm K}$,
\begin{align}
 0<K(j,k),\qquad
 0<\Kh(j,k)&\le C p^{-1/2}e^{-\gamma_{\mathrm K}\xi_{\mathrm K}^2},\label{ka:eq:kernel-bound}\\
 \left|\Kh(j,k)-\cg p^{-1/2}e^{-\sqrt5\xi_{\mathrm K}^2}\right|
 &\le C p^{-3/2}e^{-\gamma_{\mathrm K}\xi_{\mathrm K}^2}.\label{ka:eq:kernel-weak-error}
\end{align}
The bounds are absolute, Gaussian-weighted bounds. No uniform relative expansion in the far tails is asserted.
\end{theorem}

\subsection{Uniform saddle calculation from the exact generating function}
The calculation below gives its own uniform remainder argument. It continues
the earlier saddle-point work with N.~M.~Temme recorded in the author's
unpublished working revision of this paper, dated 4 September 2026,
Appendix~C (associated with \cite{alkauskas2012}, not an arXiv version);
general accounts of the method are
\cite{temme}, \cite[\S2.4(iv)]{DLMF} and \cite[Chapter~VIII]{flajoletsedgewick}.
The coefficients of $\beta(z)^k$ also admit the classical lattice
local-limit viewpoint \cite{gned}; Theorem~\ref{ka:thm:kernel} is proved
here by the direct uniform saddle calculation below.

Use the increment law with probability generating function $\beta$
from Appendix~\ref{app:exact-kernel}.
Let $\psi(t)=\log\beta(e^t)$ for real $t<\log(2\phi)$. Differentiating gives
\begin{equation*}
 \psi'(0)=1,\quad \psi''(0)=\frac2{\sqrt5},\quad
 \psi'''(0)=\frac65,\quad \psi^{(4)}(0)=\frac{34}{5\sqrt5}.
\end{equation*}
For $\sigma=j/k>0$, the saddle $t=t(\sigma)$ is the unique real solution of
$\psi'(t)=\sigma$. Indeed $\psi''(t)$ is the positive variance of the tilted distribution, and $\psi'(t)$ increases from zero to infinity. With $r=e^t$, explicitly
\begin{align}
 \sigma&=\frac{5r}{(2\phi-r)(1+2\phi^{-1}r)},\label{ka:eq:saddle-eqn}\\
 r(\sigma)&=\frac{\phi}{4\sigma}
 \left(3\sigma-5+\sqrt{25-30\sigma+25\sigma^2}\right).\notag
\end{align}
Write
\[
 I(\sigma)=\sigma t-\psi(t),\qquad v=\psi''(t).
\]

\begin{lemma}[Compact-ratio saddle expansion]\label{ka:lem:saddle}
Uniformly when $\sigma=j/k$ ranges over a fixed compact subinterval of $(0,\infty)$, as $k\to\infty$,
\begin{equation*}
 [z^j]\beta(z)^k=
 \frac{e^{-kI(\sigma)}}{\sqrt{2\pi k v}}
 \left(1+\frac{b(\sigma)}k+O(k^{-2})\right),
\end{equation*}
where
\begin{equation}\label{ka:eq:saddle-b}
 b(\sigma)=\frac{\psi^{(4)}(t)}{8v^2}
         -\frac{5\psi'''(t)^2}{24v^3}.
\end{equation}
The implicit constant may depend on the chosen compact interval, but not on $j,k$ within it.
\end{lemma}
\begin{proof}
The saddles $r$ stay in a compact subinterval of $(0,2\phi)$. For the tilted probabilities
\[
 \mathsf p_h(r)=\frac{[z^h]\beta(z)\,r^h}{\beta(r)},
\]
the product $\mathsf p_0(r)\mathsf p_1(r)$ is therefore bounded below by a positive constant $\gamma$. Let
$\chi_r(\theta)=\beta(re^{i\theta})/\beta(r)$. Since all $\mathsf p_h$ are nonnegative and sum to one,
\[
 |\chi_r(\theta)|^2
 =1-2\sum_{h<m}\mathsf p_h\mathsf p_m\bigl(1-\cos((h-m)\theta)\bigr)
 \le 1-2\gamma(1-\cos\theta).
\]
It follows that $|\chi_r(\theta)|\le e^{-c_0\theta^2}$ on $[-\pi,\pi]$, with one $c_0>0$ throughout the compact tilt interval.

Cauchy extraction gives
\begin{equation*}
 [z^j]\beta(z)^k=
 \frac{e^{-kI(\sigma)}}{2\pi}
 \int_{-\pi}^{\pi}\chi_r(\theta)^k e^{-ij\theta}\dd\theta.
\end{equation*}
The contribution of $|\theta|>k^{-2/5}$ is $O(e^{-c_0 k^{1/5}})$ inside the integral. In the complementary interval put $y=\sqrt k\,\theta$. An analytic logarithm of $\beta(re^{i\theta})$ exists uniformly for sufficiently small $|\theta|$. Its derivatives through order six are uniformly bounded, and $v$ is bounded below by a positive constant. Taylor expansion of the exponent yields
\begin{align*}
 k\bigl(\psi(t+iy/\sqrt k)-\psi(t)-i\sigma y/\sqrt k\bigr)
 ={}&-\frac{vy^2}{2}
 -\frac{i\kappa_3y^3}{6\sqrt k}
 +\frac{\kappa_4y^4}{24k}
 +\frac{i\kappa_5y^5}{120k^{3/2}}
 +O(k^{-2}|y|^6),
\end{align*}
where $\kappa_s=\psi^{(s)}(t)$. On $|y|\le k^{1/10}$ the nonquadratic exponent is uniformly bounded. Expanding its exponential through order $k^{-3/2}$, the omitted terms are bounded by
\[
 C k^{-2}(1+|y|^{12})e^{-v y^2/2}
\]
after including the quadratic exponential. This follows directly from the exponential Taylor remainder: the retained cubic term is $O(k^{-1/2}|y|^3)$, so any higher product can be reduced to the displayed polynomial bound on this interval.

The terms of orders $k^{-1/2}$ and $k^{-3/2}$ are odd in $y$ and integrate to zero over the symmetric interval. The even polynomial of order $k^{-1}$ is
\[
 \frac{\kappa_4y^4}{24}-\frac{\kappa_3^2y^6}{72}.
\]
Extending these Gaussian integrals to $\R$ introduces only an exponentially small error. Their fourth and sixth moments are $3/v^2$ and $15/v^3$ relative to the zeroth Gaussian integral. This gives \eqref{ka:eq:saddle-b}; the integrated omitted terms are $O(k^{-2})$ relative to the leading Gaussian integral. All constants are uniform on the original compact interval.
\end{proof}

\subsection{Exact symmetric phase and amplitude}
Set
\[
 m=\frac{j+k}{2},\qquad u=\frac{j-k}{j+k},\qquad
 \sigma=\frac{1+u}{1-u},\qquad
 D(\sigma)=\sqrt{25-30\sigma+25\sigma^2}.
\]
Direct differentiation of \eqref{ka:eq:saddle-eqn} gives the exact variance identity
\begin{equation*}
 v=\frac{\sigma D(\sigma)}5.
\end{equation*}
As a function of $\sigma$, it satisfies
\[
 \psi'''(t)=vv',\qquad
 \psi^{(4)}(t)=v\bigl((v')^2+vv''\bigr).
\]
Consequently
\begin{align*}
 b(\sigma)&=\frac18v''-\frac{(v')^2}{12v}\\
 &=-\frac{\sqrt5(5\sigma^2-11\sigma+5)(5\sigma^2+2\sigma+5)}
 {60\sigma(5\sigma^2-6\sigma+5)^{3/2}}.
\end{align*}
In particular $b(1)=\sqrt5/40$. This is the coefficient $a_1(\sigma)$
in that unpublished working revision, Appendix~C, p.~38,
since its $W=\sqrt{25-30\sigma+25\sigma^2}$ equals
$\sqrt5\sqrt{5-6\sigma+5\sigma^2}$.

Let
\begin{equation*}
 J(u)=(1-u)I\left(\frac{1+u}{1-u}\right).
\end{equation*}
The Legendre relations $I'(\sigma)=t$ and $I''(\sigma)=1/v$ give
\begin{equation*}
 J(0)=J'(0)=0,\qquad
 J''(u)=\frac{2\sqrt5}{(1-u^2)\sqrt{1+4u^2}}.
\end{equation*}
Thus $J$ is even. The prefactor in Lemma~\ref{ka:lem:saddle}, after multiplying by $\sqrt\sigma$ as required by \eqref{ka:eq:coefficient}, becomes
\[
 \frac{\cg}{\sqrt m}(1+4u^2)^{-1/4}.
\]
Also, $b(\sigma)/k=A(u)/m$, where
\begin{equation*}
 A(u)=\frac{\sqrt5(3+2u^2)(1-21u^2)}
 {120(1-u^2)(1+4u^2)^{3/2}}.
\end{equation*}
Therefore, uniformly on every fixed $|u|\le u_0<1$,
\begin{equation*}
 \Kh(j,k)=\frac{\cg}{\sqrt m}(1+4u^2)^{-1/4}e^{-mJ(u)}
 \left(1+\frac{A(u)}m+O(m^{-2})\right).
\end{equation*}

For reference, the Taylor data needed below are
\begin{align}
 J(u)&=\sqrt5\left(u^2-\frac{u^4}{6}+\frac{u^6}{3}
                         +O(u^8)\right),\label{ka:eq:J-taylor}\\
 (1+4u^2)^{-1/4}&=1-u^2+O(u^4),\label{ka:eq:amplitude-taylor}\\
 A(u)&=\frac{\sqrt5}{40}+O(u^2).\label{ka:eq:A-taylor}
\end{align}
All Taylor remainders here are uniform on a sufficiently small fixed neighborhood of zero, since their defining functions are analytic there.

\subsection{Proof of the global kernel theorem}
\paragraph{The near-diagonal expansion}
Retain $p,\xi_{\mathrm K}$ from \eqref{ka:eq:root-coords} and put $\epsilon=\xi_{\mathrm K}^2/p$. The exact relations are
\begin{equation}\label{ka:eq:eps-relations}
 m=p(1+\epsilon/2),\qquad
 u^2=\frac{\epsilon(4+\epsilon)}{(2+\epsilon)^2}.
\end{equation}
Choose a sufficiently small fixed $u_0>0$. The region $|u|\le u_0$ is equivalently a region $0\le\epsilon\le\epsilon_0$ with sufficiently small fixed $\epsilon_0$. From \eqref{ka:eq:J-taylor}--\eqref{ka:eq:eps-relations},
\begin{align}
 mJ(u)&=\sqrt5\,\xi_{\mathrm K}^2-\frac{5\sqrt5}{12}\frac{\xi_{\mathrm K}^4}{p}
                           +O\left(\frac{\xi_{\mathrm K}^6}{p^2}\right),\notag\\
 (p/m)^{1/2}(1+4u^2)^{-1/4}
 &=1-\frac{5\xi_{\mathrm K}^2}{4p}+O\left(\frac{\xi_{\mathrm K}^4}{p^2}\right),\label{ka:eq:amplitude-root}\\
 \frac{A(u)}m&=\frac{\sqrt5}{40p}+O\left(\frac{\xi_{\mathrm K}^2}{p^2}\right).
 \label{ka:eq:correction-root}
\end{align}
The saddle remainder $O(m^{-2})$ is $O(p^{-2})$ in this region.

For an explicit error estimate in the exponential, write
\[
 d=\sqrt5\,\xi_{\mathrm K}^2-mJ(u),\qquad
 d_0=\frac{5\sqrt5\xi_{\mathrm K}^4}{12p}.
\]
Then $|d-d_0|\le C\xi_{\mathrm K}^6/p^2$ and $|d|\le C\epsilon\xi_{\mathrm K}^2$.
The elementary inequality $|e^d-1-d|\le\tfrac12|d|^2e^{|d|}$ yields
\[
 \left|e^d-1-d_0\right|
 \le C p^{-2}(|\xi_{\mathrm K}|^6+|\xi_{\mathrm K}|^8)e^{C\epsilon_0\xi_{\mathrm K}^2}.
\]
Choose $\epsilon_0$ so that $C\epsilon_0<\sqrt5/2$ (enlarging a preliminary fixed Taylor bound before this choice if necessary). Combining this with \eqref{ka:eq:amplitude-root} and \eqref{ka:eq:correction-root}, all omitted products are bounded, after the leading factor is included, by
\[
 C p^{-5/2}(1+|\xi_{\mathrm K}|^8)e^{-\gamma_{\mathrm K}\xi_{\mathrm K}^2}.
\]
For example, the amplitude--phase cross term has size $\xi_{\mathrm K}^6/p^2$, and the squared phase correction has size $\xi_{\mathrm K}^8/p^2$. This proves \eqref{ka:eq:kernel-expansion}--\eqref{ka:eq:kernel-remainder} near the diagonal. The finitely many near-diagonal pairs below any fixed saddle threshold are absorbed by increasing $C$.

\paragraph{The complementary region}
Outside $|u|\le u_0$, the ratio $\sigma$ is either at least some fixed $\sigma_+>1$ or at most some fixed $\sigma_-<1$. Positivity gives, for every $t<\log(2\phi)$,
\begin{equation*}
 \Kh(j,k)=\sqrt\sigma\,[z^j]\beta(z)^k
 \le\sqrt\sigma\,\exp\bigl(k\psi(t)-jt\bigr).
\end{equation*}
Choose $t_+>0$ small enough that $\psi(t_+)/t_+<\sigma_+$. When $\sigma\ge\sigma_+$, the exponent is at most
$-j(t_+-\psi(t_+)/\sigma_+)$. The factor $\sqrt\sigma\le\sqrt j$ is absorbed into a weaker exponential. For $\sigma\le\sigma_-$, choose $t_-<0$ close enough to zero that
$\psi(t_-)-\sigma_-t_-<0$. The exponent is then at most a negative constant times $k$, and $\sqrt\sigma\le1$. Hence, throughout the complementary region,
\begin{equation*}
 \Kh(j,k)\le C e^{-\eta(j+k)}
\end{equation*}
for some $C,\eta>0$.

Writing $N=j+k$, we have $1\le p\le N/2$ and
\[
 \xi_{\mathrm K}^2=N\left(1-\sqrt{1-u^2}\right),\qquad
 \kappa N\le\xi_{\mathrm K}^2\le N
\]
for a fixed $\kappa>0$ in this region. Choose $\gamma_{\mathrm K}>0$ smaller than the near-diagonal exponent and than $\eta,\sqrt5$. Exponential decay then gives
\[
 e^{-\eta N}\le C p^{-5/2}(1+|\xi_{\mathrm K}|^8)e^{-\gamma_{\mathrm K}\xi_{\mathrm K}^2},
\]
and the same bound holds for the absolute value of the two explicit Gaussian terms in \eqref{ka:eq:kernel-expansion}. For the latter comparison use $\xi_{\mathrm K}^2\ge\kappa N$ to absorb the extra powers of $p$. Thus their difference from $\Kh$ satisfies \eqref{ka:eq:kernel-remainder} in the complementary region as well.

Finally, all fixed polynomials in $\xi_{\mathrm K}$ can be absorbed into a weaker Gaussian, and $p\ge1$. Applying this observation to \eqref{ka:eq:kernel-expansion}--\eqref{ka:eq:kernel-remainder} proves \eqref{ka:eq:kernel-bound}--\eqref{ka:eq:kernel-weak-error}. This completes Theorem~\ref{ka:thm:kernel}.

\paragraph{Two checks on the result}
On the diagonal,
\begin{equation*}
 K(n,n)=\frac{\cg}{\sqrt n}
 \left(1+\frac{\sqrt5}{40n}+O(n^{-2})\right).
\end{equation*}
In midpoint coordinates, \eqref{ka:eq:J-taylor} instead gives
\[
 -mJ(u)=-\frac{\sqrt5(j-k)^2}{4m}
        +\frac{\sqrt5(j-k)^4}{96m^3}
        +O\left(\frac{(j-k)^6}{m^5}\right).
\]
In a window $j-k\asymp k^{2/3}$ the quartic phase correction
is of order $k^{-1/3}$. This explains why the absolute
Gaussian-weighted statement of Theorem~\ref{ka:thm:kernel}, with its
explicit correction and specified coordinates, is the useful uniform
form for the chain calculation.

\subsection{Complete-chain control beyond fixed length}\label{sec:pending-majorants}
The estimates above do not bound the growth of the fixed-length remainder
constants. The full summation is instead established in
Section~\ref{sec:closure-sum} by a killed Green potential and a total
half-line residual estimate. Section~\ref{sec:closure-transfer} controls
all upper resolvent corrections and the entire lower-index sector in
the Wiener algebra, giving a total absolute coefficient-transfer error
of order $n^{-1}$.

This proof does not require a separate construction of length-by-length
summable defect envelopes or recursive principal coefficients. Those
stronger forms are not asserted here. The fixed-length correction theory is preserved in the frozen supplementary
material; it is not used to infer a third eigenvalue term or summability
of the fixed-length remainder constants.

\section{Kernel collapse}
\label{app:kernel-collapse}

We prove the identity $2\mathsf A+\mathsf B=2\mathsf C+\mathsf D$
directly from the coefficients of the kernel \eqref{t2:eq:kernel}.
Here and below $i,j\ge1$.  The key is the algebraic generating
function of the symmetric kernel products.

\begin{proposition}[Kernel-pair generating function]
	\label{prop:kernel-pair-gf}
	For $|s|\le\phi^{-2}$ and $|t|\le1$, the series
	\[
	\mathcal T(s,t)=\sum_{i,j\ge1}K(i,j)K(j,i)s^it^j
	\]
	converges absolutely and satisfies
	\begin{equation}\label{eq:kernel-pair-gf}
		\boxed{\displaystyle
			\mathcal T(s,t)=\frac12\left(
			\frac{4\phi^2-s-t+4\phi^{-2}st}
			{\sqrt{\bigl(4\phi^2-s-t+4\phi^{-2}st\bigr)^2-100st}}-1
			\right).}
	\end{equation}
	The square root is the analytic branch equal to $4\phi^2$ at
	$(s,t)=(0,0)$.
\end{proposition}

\begin{proof}
	By Proposition~\ref{prop:exact-kernel}, $K(i,j)\ge0$ and
	$\sum_{i\ge1}K(i,j)=1$, so $0\le K(i,j)\le1$.
	In particular, with $a_{i,j}=K(i,j)K(j,i)$,
	\[
	\sum_{j\ge1}a_{i,j}\le\sum_{j\ge1}K(j,i)=1.
	\]
	This proves the asserted absolute convergence.
	
	Direct substitution of $\beta$ gives the rational identity
	\begin{equation}\label{eq:kernel-pair-row}
		\frac{y\beta'(x)}{1-y\beta(x)}
		=\frac{\beta(y)}{1-x\beta(y)}
		-\frac{\beta(0)}{1-x\beta(0)}.
	\end{equation}
	Comparing coefficients of $x^{i-1}$ yields
	\begin{equation}\label{eq:kernel-pair-row-coeff}
		\sum_{j\ge1}K(i,j)y^j=\beta(y)^i-\beta(0)^i.
	\end{equation}
	
	Fix $0<|t|\le1$, $|s|\le\phi^{-2}$, and write
	$R=\sqrt{|t|}$.  For a function holomorphic on an annulus
	containing $|u|=R$, let
	\[
	\operatorname{CT}_u F(u)
	=\frac1{2\pi i}\int_{|u|=R}F(u)\,\frac{du}{u}.
	\]
	On this circle $|u|=|t/u|=R\le1$.  Hence
	$|\beta(u)|,|\beta(t/u)|\le1$, and the following sums may be
	interchanged absolutely.  By the defining column series for the
	kernel and \eqref{eq:kernel-pair-row-coeff},
	\begin{align*}
		\mathcal T(s,t)
		&=\operatorname{CT}_u
		\sum_{i\ge1}s^i
		\bigl(u\beta'(u)\beta(u)^{i-1}\bigr)
		\bigl(\beta(t/u)^i-\beta(0)^i\bigr)\\
		&=\operatorname{CT}_u\mathcal F(u),
	\end{align*}
	where
	\begin{equation}\label{eq:kernel-pair-F}
		\mathcal F(u)
		=\frac{s\,u\beta'(u)\beta(t/u)}
		{1-s\beta(u)\beta(t/u)}.
	\end{equation}
	The terms containing $\beta(0)^i$ have zero constant term,
	since $u\beta'(u)\beta(u)^{i-1}$ has only strictly positive
	powers of $u$.
	
	The substitution $u\mapsto t/u$ preserves the circle $|u|=R$;
	the orientation reversal is offset by the differential $du/u$.
	Consequently
	\begin{equation}\label{eq:kernel-pair-CT-sym}
		\operatorname{CT}_u\mathcal F(t/u)
		=\operatorname{CT}_u\mathcal F(u).
	\end{equation}
	Set
	\begin{align*}
		h&=4\phi^2-s-t+4\phi^{-2}st,\\
		\alpha&=4\phi^2-s+t-4\phi^{-2}st,
		&\gamma&=2\phi+2\phi^{-1}s.
	\end{align*}
	A rational simplification of \eqref{eq:kernel-pair-F} gives
	\begin{equation}\label{eq:kernel-pair-rational}
		\mathcal F(u)+\mathcal F(t/u)
		=\frac{h}{\alpha-\gamma(u+t/u)}
		-1-\frac{u}{2\phi-u}
		-\frac{t/u}{2\phi-t/u}.
	\end{equation}
	Indeed, the common quadratic factor in the denominators on the
	left is
	\[
	(2\phi-u)(2\phi u-t)
	-s(1+2\phi^{-1}u)(u+2\phi^{-1}t)
	=\alpha u-\gamma(u^2+t),
	\]
	so \eqref{eq:kernel-pair-rational} is a direct rational identity.
	
	The last two fractions in \eqref{eq:kernel-pair-rational}
	have zero constant term.  Taking constant terms, using
	\eqref{eq:kernel-pair-CT-sym}, gives
	\begin{equation}\label{eq:kernel-pair-main-CT}
		2\mathcal T(s,t)
		=h\operatorname{CT}_u
		\frac1{\alpha-\gamma(u+t/u)}-1.
	\end{equation}
	For completeness, put $r=\phi^{-2}$.  On the stated
	parameter domain, the triangle inequality gives
	\begin{align*}
		|\alpha|-|\gamma|(1+|t|)
		&\ge4\phi^2-r-1-4r^2
		-2(2\phi+2\phi^{-1}r)\\
		&=\frac{5\sqrt5-9}{2}>0.
	\end{align*}
	As $|u+t/u|\le2\sqrt{|t|}\le1+|t|$ on $|u|=R$,
	the geometric and binomial expansions below converge absolutely:
	\begin{align}
		\operatorname{CT}_u\frac1{\alpha-\gamma(u+t/u)}
		&=\frac1\alpha\sum_{m\ge0}\binom{2m}{m}
		\left(\frac{\gamma^2t}{\alpha^2}\right)^m\notag\\
		&=\frac1{\sqrt{\alpha^2-4\gamma^2t}}.
		\label{eq:kernel-pair-binomial}
	\end{align}
	Here we used
	$\operatorname{CT}_u(u+t/u)^{2m}=\binom{2m}{m}t^m$;
	odd powers contribute zero.  The square-root branch in
	\eqref{eq:kernel-pair-binomial} is fixed by the displayed
	convergent expansion.  Finally, direct algebra yields
	\[
	\alpha^2-4\gamma^2t=h^2-100st.
	\]
	Substitution into \eqref{eq:kernel-pair-main-CT} proves
	\eqref{eq:kernel-pair-gf} for $t\ne0$; the case $t=0$
	follows by continuity.
\end{proof}

\begin{theorem}[Kernel collapse]
	\label{thm:kernel-collapse}
	Define the coefficients
	\begin{equation}\label{eq:W1-W2-kernel}
		W_1(n)=K(n,n),\qquad
		W_2(n)=K(n,n)^2+
		\sum_{\substack{j\ge1\\j\ne n}}
		\frac{K(n,j)K(j,n)}{1-\rho^{j-n}}
	\end{equation}
	and the constants
	\[
	\mathsf A=\sum_{n\ge1}W_2(n)\rho^{2n},\qquad
	\mathsf B=\sum_{n\ge1}W_1(n)^2\rho^{2n},
	\]
	\[
	\mathsf C=\frac2{21+5\sqrt{21}},\qquad
	\mathsf D=\frac1{16+12\sqrt2}.
	\]
	Then
	\begin{equation}\label{eq:kernel-collapse}
		\boxed{\displaystyle
			2\mathsf A+\mathsf B=2\mathsf C+\mathsf D.}
	\end{equation}
\end{theorem}

\begin{proof}
	Write $a_{n,j}=K(n,j)K(j,n)=a_{j,n}$ as above.
	For every nonzero integer $d$ one has
	$|1-\rho^d|^{-1}\le(1-|\rho|)^{-1}$; hence the
	convergence bound in the proposition justifies the pairing below.
	For $j\ne n$, the elementary difference-of-cubes factorization
	gives
	\begin{equation}\label{eq:kernel-pair-cancellation}
		\begin{aligned}
			\frac{\rho^{2n}}{1-\rho^{j-n}}
			+\frac{\rho^{2j}}{1-\rho^{n-j}}
			&=\frac{(\rho^n)^3-(\rho^j)^3}{\rho^n-\rho^j}\\
			&=\rho^{2n}+\rho^{n+j}+\rho^{2j}.
		\end{aligned}
	\end{equation}
	Consequently,
	\begin{align*}
		2\mathsf A+\mathsf B
		&=3\sum_{n\ge1}a_{n,n}\rho^{2n}
		+2\sum_{n<j}a_{n,j}
		(\rho^{2n}+\rho^{n+j}+\rho^{2j})\\
		&=2\mathcal T(\rho^2,1)+\mathcal T(\rho,\rho).
	\end{align*}
	The term $\mathsf B$ supplies exactly the diagonal contribution
	in the second equality.
	
	Set $r=\phi^{-2}$, so that $\rho=-r$ and
	$r^2-3r+1=0$.  For the numerator $h(s,t)$ in
	\eqref{eq:kernel-pair-gf}, direct substitution gives
	\[
	\begin{array}{c|c|c}
		(s,t) & h(s,t) & \sqrt{h(s,t)^2-100st} \\ \hline
		(r^2,1) & 25r & 5\sqrt{21}\,r \\
		(-r,-r) &30r&20\sqrt2\,r
	\end{array}
	\]
	(the square roots are positive in these two real cases).
	Therefore, by \eqref{eq:kernel-pair-gf},
	\begin{align*}
		\mathcal T(\rho^2,1)
		&=\frac12\left(\frac5{\sqrt{21}}-1\right)
		=\frac2{21+5\sqrt{21}}=\mathsf C,\\
		\mathcal T(\rho,\rho)
		&=\frac12\left(\frac3{2\sqrt2}-1\right)
		=\frac1{16+12\sqrt2}=\mathsf D.
	\end{align*}
	This proves \eqref{eq:kernel-collapse}.
\end{proof}

\section{Technical estimates for all-length comparison and Wiener localization}\label{app:technical-estimates}
Results in this Appendix are need for the same Section~\ref{sec:asymptotic}.

\subsection{Summed residual estimate}\label{app:residual-proof}
We use the square-root mesh \eqref{st1:eq:mesh}, the sampled functions
$a_m^{\mathrm s}$ and the residuals \eqref{st1:eq:em}.
\begin{proof}[Proof of Lemma~\ref{st1:lem:residuals}]
Write $R=\sqrt n$.  By \eqref{st1:eq:rooterror},
\begin{eqnarray*}
 \norm{e_0}_{\ell^1(\mu)}
 \le \frac C{s_n}\sum_{i>n}\frac{\mu_i}{s_i}
       e^{-\gamma_{\mathrm K}(s_i-s_n)^2}
 \le \frac C{R^2}\sum_{i>n}\mu_i e^{-\gamma_{\mathrm K}(s_i-s_n)^2}
 \le CR^{-2},
\end{eqnarray*}
where the last Gaussian mesh sum is uniformly bounded by the observation
in the proof of Lemma~\ref{st1:lem:barrier}.
For $m\ge1$, compare
\[
 a_m(x_i)=\int_0^\infty g(x_i-y)a_{m-1}(y)\dd y
\]
with $P_na_{m-1}^{\mathrm s}(i)$. First replace $\Kn(i,k)$ by $g(x_i-x_k)$.  Summing the resulting
absolute error over $m,i,k$ and using
$\sum_m a_m=f\in L^\infty$ gives
\begin{eqnarray*}
\sum_{m\ge1}\sum_{i,k>n}
 \mu_i\mu_k
 \abs{\Kn(i,k)-g(x_i-x_k)}a_{m-1}(x_k)\le
 C\sum_{i,k>n}\frac{\mu_i\mu_k}{s_is_k}
 e^{-\gamma_{\mathrm K}(s_i-s_k)^2}
 \le C\sum_{i>n}\frac{\mu_i}{s_i^2}
 \le \frac{C}{R}.
\end{eqnarray*}
To justify the last reduction, use $2(s_is_k)^{-1}\le\left(s_i^{-2}+s_k^{-2}\right)$. By symmetry it is enough to estimate
\[
 \sum_{i>n}\frac{\mu_i}{s_i^2}
 \sum_{k>n}\mu_ke^{-\gamma_{\mathrm K}(s_i-s_k)^2}.
\]
The inner Gaussian mesh sum is uniformly bounded in $i$, as established
in the proof of Lemma~\ref{st1:lem:barrier} using
$\mu_k\le s_k-s_{k-1}$.  Hence the preceding expression is at most
\[
 C\sum_{i>n}\frac{\mu_i}{s_i^2}
 =\frac C2\sum_{i>n}i^{-3/2}
 \le\frac C{\sqrt n}=\frac CR.
\]
All interchanges here are Tonelli interchanges of nonnegative quantities. For the quadrature part apply Lemma~\ref{st1:lem:quad} to
\[
 h_{i,m}(y)=g(x_i-y)a_{m-1}(y).
\]
For the boundary terms, Tonelli and Lemma~\ref{st1:lem:derivatives} give
\[
 \frac1{4R}\sum_{i>n}\mu_i g(x_i)\sum_{m\ge1}a_{m-1}(0)
 \le\frac CR,
\]
since $\sum_m a_m(0)=f(0)<\infty$ and the sampled Gaussian has bounded
mesh mass.  For the remainder terms, the product rule and
Lemma~\ref{st1:lem:derivatives} imply, for $r=0,1,2$,
\[
 \sum_{m\ge1}|h_{i,m}^{(r)}(y)|
 \le C\sum_{q=0}^r|g^{(q)}(x_i-y)|.
\]
After summing over $i$ with weights $\mu_i$, the Gaussian mesh bounds
make the right-hand side uniformly bounded in $y$.
Thus \eqref{st1:eq:quadR}, followed by Tonelli, gives a total remainder
bounded by
\[
 C\int_0^\infty
 \left((R+y)^{-4}+(R+y)^{-3}+(R+y)^{-2}\right)\dd y
 \le \frac{C}{R}.
\]
All interchanges are Tonelli interchanges under the displayed absolute
majorants.  This proves \eqref{st1:eq:residualsum}.
\end{proof}

\subsection{Localized Wiener estimates}\label{app:localization-proof}
We use $L,U,Q,Z,\Gres,C_U,b,F,E^0,v$ from
\eqref{st1:eq:QZ}--\eqref{st1:eq:vdef}, and put $R=\sqrt n$.
The scalar $b$ in the exponential weights below is the positive barrier
constant from Lemma~\ref{st1:lem:barrier}.
At \(\omega=1\), every occurrence of \(\Gres(1)\) below means the
entrywise nonnegative Green series
\[
 \Gres(1)_{ij}:=\sum_{m\ge0}(\Hker_{UU}^{\,m})_{ij},
\]
inside one of the localized actions or sandwiches explicitly estimated
below.  Their finiteness follows from Lemma~\ref{st1:lem:barrier}.
We never interpret \(\Gres(1)\) as a bounded inverse of
\(I-\Hker_{UU}\) on all of \(\ell^2(U)\).
\begin{proof}[Proof of Lemma~\ref{st1:lem:localization}]
All coefficient matrices and vectors appearing below have nonnegative
entries.  We use two different coefficientwise majorants.  If
$A(\omega)=\sum_{m\ge0}A_m\omega^m$ has positive-semidefinite trace-class
coefficients, then
\[
 \norm{A}_{\Walg(\mathrm{op})}
 =\sum_m\norm{A_m}_{\mathrm{op}}\le\sum_m\Tr A_m.
\]
If instead the coefficient matrices merely have nonnegative entries and
$\sum_{m,i,j}(A_m)_{ij}<\infty$, then
\[
 \sum_m\norm{A_m}_{\mathrm{op}}\le\sum_{m,i,j}(A_m)_{ij},
\]
by the Hilbert--Schmidt bound followed by
$\norm{A_m}_{\HS}\le\sum_{i,j}(A_m)_{ij}$.
For nonnegative vectors the analogous entrywise bound follows from the
$\ell^2$ norm being at most the $\ell^1$ norm.
We specify which majorant is used in each case.

First consider $C_U$.  At the endpoint, Lemma~\ref{st1:lem:barrier} with
source $j=i>n$ gives
\begin{equation*}
 \Gres(1)_{ii}-1
 =[\Gres(1)\Hker_{UU}]_{ii}
 \le C\mu_i e^{b(s_i-R)}.
\end{equation*}
Hence, using $\sqrt{n+r}-\sqrt n\le\sqrt r$,
\[
 \sum_{i>n}|\rho|^{i-n}(\Gres(1)_{ii}-1)
 \le \frac{C}{\sqrt n}\sum_{r\ge1}|\rho|^r e^{b\sqrt r}
 \le Cn^{-1/2}.
\]
For $m\ge1$, the coefficient $Q\Hker_{UU}^mQ$ is positive semidefinite,
because $\Hker\ge0$ by Lemma~\ref{cl:lem:Laguerre}.  It is trace class
since $Q$ is Hilbert--Schmidt and $\Hker_{UU}^m$ is bounded.  The trace
majorant therefore gives
\[
 \sum_{m\ge1}\norm{Q\Hker_{UU}^mQ}_{\mathrm{op}}
 \le\sum_{m\ge1}\Tr(Q\Hker_{UU}^mQ)
 =\sum_{i>n}|\rho|^{i-n}(\Gres(1)_{ii}-1)
 \le Cn^{-1/2},
\]
where the equality uses Tonelli.  Together with the constant coefficient
$Q^2$, whose operator norm is $|\rho|$, this proves \eqref{st1:eq:locCU}.

The exact relation between the two normalizations is
\[
 [\Gres(1)\Hker_{Uj}]_i
 =\sqrt{\mu_i\mu_j}\sum_{m\ge0}[P_n^m\Kn(\cdot,j)](i),
 \qquad i>n.
\]
Indeed, every intermediate factor
$\Hker_{ab}=\sqrt{\mu_a\mu_b}\,\Kn(a,b)$ telescopes to the weighted
kernel $P_n$.  The identity holds termwise, hence for the nonnegative
Green sums.  For $j\le n$, $e^{b(s_j-R)}\le1$, so the all-source
Green estimate gives
\begin{equation}\label{st1:eq:GHentry}
 0\le[\Gres(1)\Hker_{Uj}]_i
 \le C\sqrt{\mu_i\mu_j}.
\end{equation}
Also
\begin{equation}\label{st1:eq:geomsums}
 \sum_{i>n}|\rho|^{(i-n)/2}\sqrt{\mu_i}\le Cn^{-1/4},
 \qquad
 \sum_{j<n}|\rho|^{(n-j)/2}\sqrt{\mu_j}\le Cn^{-1/4}.
\end{equation}
The first estimate is immediate from $i>n$.  For the second, split at
$j=n/2$: on $j\ge n/2$, $\sqrt{\mu_j}=O(n^{-1/4})$, while on
$j<n/2$ the geometric factor is exponentially small in $n$.
Combining \eqref{st1:eq:GHentry} with these sums, the entrywise majorant
gives
\[
 \begin{aligned}
 \norm{b}_{\Walg(\ell^2)}
 &\le C\sqrt{\mu_n}\sum_{i>n}|\rho|^{(i-n)/2}\sqrt{\mu_i}
 \le Cn^{-1/2},\\
 \norm{F}_{\Walg(\mathrm{op})}
 &\le C\left(\sum_{i>n}|\rho|^{(i-n)/2}\sqrt{\mu_i}\right)
       \left(\sum_{j<n}|\rho|^{(n-j)/2}\sqrt{\mu_j}\right)
 \le Cn^{-1/2}.
 \end{aligned}
\]
Here summing coefficient entries at $\omega=1$ is justified by their
nonnegativity.  This proves \eqref{st1:eq:locbF}.

Next, \eqref{st1:eq:rootbound} and \eqref{st1:eq:GHentry} imply
\begin{equation}\label{st1:eq:Eentries}
 E^0_{jj}(1)\le C\mu_j,
 \qquad
 E^0_{jn}(1)\le C\sqrt{\mu_j\mu_n},
 \qquad j<n.
\end{equation}
Indeed, the tail term is bounded by a Gaussian row sum after one factor
is estimated by \eqref{st1:eq:GHentry}.  The coefficients of $E^0_{LL}$
are $\Hker_{LL}$ at degree one and
$\Hker_{LU}\Hker_{UU}^m\Hker_{UL}$ at degree $m+2$, $m\ge0$.
They are positive semidefinite by $\Hker\ge0$ and
$\Hker_{UL}=\Hker_{LU}^*$, so the trace majorant gives
\[
 \norm{ZE^0Z}_{\Walg(\mathrm{op})}
 \le\sum_{j<n}|\rho|^{n-j}E^0_{jj}(1)\le Cn^{-1/2}.
\]
The same split at $j=n/2$ bounds the diagonal sum above.
The second estimate in \eqref{st1:eq:Eentries} and
\eqref{st1:eq:geomsums}, using the entrywise vector majorant, give
\[
 \norm{v}_{\Walg(\ell^2)}
 \le C\sqrt{\mu_n}\sum_{j<n}|\rho|^{(n-j)/2}\sqrt{\mu_j}
 \le Cn^{-1/2}.
\]

Finally, expanding $\Gres$ in \eqref{st1:eq:Edef} with both external indices
$n$ gives
\[
 E^0_{nn}(\omega)=\omega\Hker_{nn}
 +\omega^2\Hker_{nU}\sum_{m\ge0}\omega^m\Hker_{UU}^m\Hker_{Un}.
\]
After expanding the matrix products, this is exactly the sum over strict
upper closed chains based at $n$; the symmetrizing square-root factors
cancel around each closed chain.  Hence $E^0_{nn}(\omega)=\Sser_n(\omega)$,
proving \eqref{st1:eq:E0nn}.
\end{proof}

\section{Canonical coefficient recurrences}\label{app:coefficients}
Use the normalized local eigenpair of Section~\ref{sec:germs}, and put
$p_0=\phi^{-1}$. For each $n$, choose $\varepsilon_n>0$ on which this
eigenpair and its normalization are analytic. We justify its coefficient
recurrences, including composition at the nonzero point $\beta(0)$.
\subsection{Golden coordinates and the normalization functional}

Use the golden coordinate $x$, its inverse $z(x)$, and the transform
$g_f$ from Sections~\ref{sec:weighted} and~\ref{sec:germs}.

We need a disk about $x=0$ on which this transform is valid and which
contains the point $\beta(0)=\frac1{2\phi}.$ \begin{lemma}[A uniform golden disk]
One has
\[
\{|x|<1/3\}\subset x(D_z).
\]
\end{lemma}

\begin{proof}
The inverse coordinate map is
\[
z(x)=\phi^{-1}+\frac{\sqrt5\,x}{1-x}.
\]
For $|x|<1/3$,
\[
|z(x)-1|\le1-\phi^{-1}+\frac{\sqrt5|x|}{1-|x|}
<1-\phi^{-1}+\frac{\sqrt5}{2}=\frac32.
\]
Thus $z(x)\in D_z$, proving the inclusion.
\end{proof}

In particular $\frac1{2\phi}<\frac13.$ For every $R<1/3$, the map
\[
V\longrightarrow A(\{|x|<R\}),
\qquad
f\longmapsto g_f,
\]
is bounded: the compact set $z(\{|x|\le R\})$ lies inside $D_z$, and the
factor $(z+\phi)^2$ is bounded there.

For the functionals $c_j$ of Section~\ref{sec:germs} and the explicit
eigenvector $e_n$, $g_{e_n}(x)=x^{n-1},$ and hence $c_j(e_n)=\delta_{jn}.$ \subsection{The genuine local GE}

For $f\in V$, whenever both $z$ and $z+1$ lie in $D_z$, absolute convergence
of the defining operator series gives
\[
(\Lop_\omega f)(z)
=
\frac1{(z+1)^2}
f\!\left(\frac1{z+1}\right)
+
\omega(\Lop_\omega f)(z+1).
\]
Since $p_0\in D_z,\qquad p_0+1=\phi\in D_z,$ this identity holds on a neighborhood of $z=p_0$.

Apply it to the genuine eigenpair:
\[
\Lambda_n(\omega)G_n(\omega,z)
=
\frac1{(z+1)^2}
G_n\!\left(\omega,\frac1{z+1}\right)
+
\omega\Lambda_n(\omega)G_n(\omega,z+1).
\]

Under the golden coordinate this becomes
\[
\boxed{
\Lambda_n(\omega)g_n(\omega,x)
=
\phi^{-2}g_n(\omega,\rho x)
+
\Lambda_n(\omega)\omega
A_0(x)g_n(\omega,\beta(x)),
}
\]
where
\[
A_0(x)=\beta'(x)=\frac5{(2\phi-x)^2}.
\]

This is GE for an actual analytic local eigenpair.

\subsection{Legitimacy of the nonzero-centre composition}

The point $\beta(0)=\frac1{2\phi}$ is nonzero, so composition cannot be justified at the level of arbitrary
formal $x$-series.  We now justify it analytically.

Choose $\frac1{2\phi}<r<R<\frac13.$ By continuity of $\beta$, choose $\delta>0$ such that
\[
|\beta(x)|\le r
\qquad(|x|\le\delta).
\]

Write
\[
g_n(\omega,x)
=
\sum_{i\ge1}a_i(\omega)x^{i-1}.
\]
Fix a closed parameter disk $|\omega|\le\varepsilon<\varepsilon_n.$ Because $\omega\mapsto G_n(\omega)$ is continuous into $V$, it is bounded
on this compact disk.  Since the golden-coordinate transform
$V\to A(|x|<R)$ is bounded, there is $M<\infty$ such that
\[
|g_n(\omega,x)|\le M
\qquad
(|\omega|\le\varepsilon,\ |x|\le R).
\]
Cauchy's estimate gives
\[
|a_i(\omega)|
\le
MR^{-(i-1)}
\qquad
(|\omega|\le\varepsilon).
\]
Therefore
\[
\sum_{i\ge1}
|a_i(\omega)|\,|\beta(x)|^{i-1}
\le
M\sum_{i\ge1}(r/R)^{i-1}<\infty
\]
uniformly for $|\omega|\le\varepsilon$, $|x|\le\delta$.
Hence
\[
g_n(\omega,\beta(x))
=
\sum_{i\ge1}a_i(\omega)\beta(x)^{i-1}
\]
is a legitimate locally uniformly convergent composition.

\subsection{The kernel action and uniform coefficient extraction}

By \eqref{t2:eq:kernel}, the canonical scaled kernel $K(j,i)$ satisfies
\[
A_0(x)\beta(x)^{i-1}
=
\sum_{j\ge1}K(j,i)x^{j-1}.
\]
The distinct integer kernel $\Kint(j,i)$ is related to it by
\eqref{eq:integer-kernel}.

On $|x|\le\delta$,
\[
|A_0(x)\beta(x)^{i-1}|
\le
C r^{i-1}
\]
for some $C$ independent of $i$.  Cauchy's estimate therefore gives, for
each fixed $j$,
\[
|K(j,i)|
\le
C_j r^{i-1}.
\]
Combining this with the uniform estimate for $a_i(\omega)$,
\[
\sum_{i\ge1}
|K(j,i)a_i(\omega)|
\le
C_jM\sum_{i\ge1}(r/R)^{i-1}<\infty
\]
uniformly for $|\omega|\le\varepsilon$.

It follows that
\[
[x^{j-1}]
\left(
A_0(x)g_n(\omega,\beta(x))
\right)
=
\sum_{i\ge1}K(j,i)a_i(\omega),
\]
and the right-hand side is a uniformly convergent series of analytic
functions of $\omega$ on every smaller closed parameter disk.

Consequently its Taylor coefficients in $\omega$ may also be extracted
termwise.

\subsection{The genuine \texorpdfstring{$q$--$Y$--$W$}{q-Y-W} coefficients}

Since $G_n$ is analytic as a $V$-valued germ, for every $R<1/3$ the
function $g_n$ is analytic as an $A(|x|<R)$-valued germ.  Hence it has a
locally convergent double expansion
\[
g_n(\omega,x)
=
\sum_{v\ge0}\sum_{j\ge1}
q_v^{(j)}(n)\omega^v x^{j-1}.
\]
Equivalently,
\[
a_j(\omega)=\sum_{v\ge0}q_v^{(j)}(n)\omega^v.
\]

At $\omega=0$,
\[
g_n(0,x)=x^{n-1},
\]
so $q_0^{(j)}(n)=\delta_{jn}.$ The normalization
\[
c_n(G_n(\omega))\equiv1
\]
gives
\[
\boxed{
q_v^{(n)}(n)=0
\qquad(v\ge1).
}
\]

The functions $\Yser_n$ and $\Wser_n$ are those of
\eqref{eq:YW-series}, identified with the local spectral germ in
Section~\ref{sec:germs}; in particular $\Wser_n=(1-\Yser_n)^{-1}$.
Their Taylor coefficients are $Y_\ell(n)$ and $W_\ell(n)$, with $W_0(n)=1$.

\subsection{Derivation of the canonical recurrence}

Set
\[
r_{j,n}
=
\frac{\lzero j}{\lzero n}
=
\rho^{j-n}.
\]
Divide GE by $\Lambda_n(\omega)$:
\[
g_n(\omega,x)
=
\frac{\phi^{-2}}{\Lambda_n(\omega)}
g_n(\omega,\rho x)
+
\omega A_0(x)g_n(\omega,\beta(x)).
\]

Since $\phi^{-2}\rho^{j-1}=\lzero j$,
the coefficient of $x^{j-1}$ is
\[
a_j(\omega)
=
\frac{\lzero j}{\Lambda_n(\omega)}a_j(\omega)
+
\omega\sum_{i\ge1}K(j,i)a_i(\omega).
\]
Because
\[
\frac{\lzero j}{\Lambda_n(\omega)}
=
r_{j,n}(1-\Yser_n(\omega)),
\]
we obtain
\[
(1-r_{j,n})a_j(\omega)
=
\omega\sum_{i\ge1}K(j,i)a_i(\omega)
-
r_{j,n}\Yser_n(\omega)a_j(\omega).
\]

By the uniform convergence proved above, Taylor coefficients in $\omega$
may be extracted termwise.  For $v\ge1$,
\[
\boxed{
(1-r_{j,n})q_v^{(j)}(n)
=
\sum_{i\ge1}K(j,i)q_{v-1}^{(i)}(n)
-
r_{j,n}
\sum_{a=1}^v
Y_a(n)q_{v-a}^{(j)}(n).
}
\]

For the resonant index $j=n$, $r_{n,n}=1$.  Since
\[
q_0^{(n)}(n)=1,
\qquad
q_s^{(n)}(n)=0\quad(s\ge1),
\]
the recurrence reduces to
\[
\boxed{
Y_v(n)
=
\sum_{i\ge1}K(n,i)q_{v-1}^{(i)}(n).
}
\]

For $j\ne n$, division by $1-r_{j,n}\ne0$ determines $q_v^{(j)}(n)$.

Every infinite sum in these formulas is analytically defined.

\subsection{Uniqueness and identification with the canonical coefficients}

The normalized recurrence is triangular in perturbation order.

The order-zero coefficients are fixed as above.
Assume all quantities of orders $<v$ have already been determined.
Then the resonant recurrence above uniquely determines $Y_v(n)$.
For every $j\ne n$, because $1-r_{j,n}\ne0,$ the nonresonant equation uniquely determines $q_v^{(j)}(n)$.
Finally $q_v^{(n)}(n)=0$ is fixed by normalization.

Thus there is at most one normalized coefficient system satisfying the
canonical recurrence at each finite order.  The genuine analytic spectral
germ constructed above supplies such a system.  Therefore its Taylor
coefficients coincide order by order with the canonical recursively
defined coefficients.

The direct coefficients are then uniquely determined from $\Wser_n=(1-\Yser_n)^{-1},$ namely
\[
W_v(n)
=
\sum_{a=1}^vY_a(n)W_{v-a}(n).
\]

Together with Section~\ref{sec:germs}, this proves the following realization.

\begin{theorem}[Local spectral realization of the canonical coefficient system]\label{thm:canonical-local-realization}
For every $n\ge1$ there exists $\varepsilon_n>0$ and a unique normalized
analytic eigenpair germ
\[
\Lop_\omega G_n(\omega)
=
\Lambda_n(\omega)G_n(\omega),
\qquad |\omega|<\varepsilon_n,
\]
with
\[
\Lambda_n(0)=\lzero n,
\qquad
G_n(0)=e_n,
\qquad
[x^{n-1}]g_n(\omega,x)\equiv1.
\]
The enclosed spectral point has algebraic multiplicity one.  Moreover
\begin{equation}\label{eq:YW-series}
\Lambda_n(\omega)
=
\lzero n\sum_{\ell\ge0}W_\ell(n)\omega^\ell
=
\frac{\lzero n}
{1-\sum_{\ell\ge1}Y_\ell(n)\omega^\ell},
\end{equation}
where the coefficients $q_v^{(j)}(n)$, $Y_v(n)$ and $W_v(n)$ are exactly
the canonical coefficients determined by the GE recurrences.
\end{theorem}

\section{Laguerre realization for the all-index Wiener theorem}\label{app:all-index-auxiliary}
Put $\tau=1/\sqrt5$; the kernel operator is that of Section~\ref{sec:asymptotic}.
\begin{lemma}[Laguerre realization and diagonal Green bound]
\label{cl:lem:Laguerre}
Let $\dd\nu(x)=xe^{-x}\dd x$ and
$p_j(x)=j^{-1/2}L_{j-1}^{(1)}(x)$. In this orthonormal basis of
$L^2(\nu)$,
\begin{equation}\label{cl:eq:Laguerre-H}
 \Hker_{jk}=\int_0^\infty e^{-\tau x}p_j(x)p_k(x)\dd\nu(x).
\end{equation}
Thus $0\le \Hker\le I$. For $0\le\omega<1$,
\begin{equation}\label{cl:eq:global-Gdiag}
 [(I-\omega \Hker)^{-1}]_{jj}\le1+\tau^{-1}\qquad(j\ge1).
\end{equation}
The same upper bound holds for each diagonal entry of the resolvent of
any coordinate compression of $\Hker$.
\end{lemma}
\begin{proof}
The Laguerre generating function and orthogonality are in
\cite[\S\S18.3, 18.12.13]{DLMF}. For completeness, the required matrix
identity follows directly by multiplying
\[
 \sum_{m\ge0}L_m^{(1)}(x)z^m
 =(1-z)^{-2}\exp\!\left(-\frac{xz}{1-z}\right)
\]
with its $w$ analogue and integrating against $xe^{-(1+\tau)x}\dd x$.
This interchange is absolute for \(|z|,|w|<1\): since
\(\|L_m^{(1)}\|_{L^2(\nu)}=\sqrt{m+1}\) and \(e^{-\tau x}\le1\),
Cauchy--Schwarz bounds the sum of the absolute integrals by
\[
 \left(\sum_{m\ge0}\sqrt{m+1}|z|^m\right)
 \left(\sum_{k\ge0}\sqrt{k+1}|w|^k\right)<\infty.
\]
The resulting double generating function is
\[
 \bigl((1+\tau)-\tau(z+w)+(\tau-1)zw\bigr)^{-2}.
\]
Define
\begin{equation*}
 \beta_s(z)=1-\frac{1-z}{1+s(1-z)}.
\end{equation*}
Then $\beta=\beta_\tau$, and the last rational function is also
$\beta_\tau'(z)/(1-w\beta_\tau(z))^2$.
Its $(j-1,k-1)$ coefficient is $kK(j,k)=\sqrt{jk}\Hker_{jk}$.
This proves \eqref{cl:eq:Laguerre-H}, including positivity as an operator.

By multiplication-operator calculus,
\[
 [(I-\omega \Hker)^{-1}]_{jj}
 =\int_0^\infty\frac{p_j(x)^2}{1-\omega e^{-\tau x}}\dd\nu(x).
\]
For $u>0$,
$(1-e^{-u})^{-1}=1+(e^u-1)^{-1}\le1+u^{-1}$.
Moreover,
\[
 \int_0^\infty p_j(x)^2e^{-x}\dd x=1,
\]
since $L_{j-1}^{(1)}=\sum_{h=0}^{j-1}L_h^{(0)}$ and the latter polynomials
are orthonormal for $e^{-x}\dd x$. This proves
\eqref{cl:eq:global-Gdiag}. Positivity of entries and the Neumann series
show that a compressed resolvent has no larger diagonal entries.
\end{proof}

\raggedbottom

\end{document}